\documentclass{amsart}
\usepackage{amsmath,amssymb,mathrsfs,bbm,yhmath,mathtools,multirow}
\usepackage{amsthm,rotating,epsfig,hyperref,url,bbm,mdframed,stmaryrd}
\usepackage{tikz-cd}
\usepackage{graphicx}
\allowdisplaybreaks

\date{}
\def\tr{\qopname\relax o{tr}}

\theoremstyle{theorem}
\newtheorem{theo}{Theorem}[section]
\newtheorem{lemm}[theo]{Lemma}
\newtheorem{prop}[theo]{Proposition}
\newtheorem{coro}[theo]{Corollary}
\newtheorem{conj}[theo]{Conjecture}

\theoremstyle{remark}
\newtheorem{rema}[theo]{Remark}

\theoremstyle{definition}
\newtheorem{defi}[theo]{Definition}
\numberwithin{equation}{section}
\numberwithin{figure}{section}

\graphicspath{{./pictures/}}	

\author {Ivan Dynnikov and Maxim Prasolov}
\address{\noindent V.A.Steklov Mathematical Institute of Russian Academy of Science, 8 Gubkina Str., Moscow 119991, Russia}
\email{dynnikov@mech.math.msu.su}
\email{0x00002a@gmail.com}
\thanks{The work is supported by the Russian Science Foundation under grant~22-11-00299.}

\title{An algorithm for comparing Legendrian knots}

\begin{document}
\maketitle

\begin{abstract}
We construct an algorithm to decide whether two given Legendrian or transverse
links are equivalent. In general, the complexity of the algorithm is too high for practical
implementation. However, in many cases, when the symmetry group of the link is
small and explicitly known, the most time-consuming part of the algorithm can be bypassed,
thus allowing one to compare many pairs of Legendrian and transverse links in practice.
\end{abstract}

\section{Introduction}

Legendrian knots have been intensely studied starting from groundbreaking work of D.\,Bennequin~\cite{ben}
where he used them to prove nonstandardness of a contact structure in~$\mathbb R^3$.
In the context of contact topology, it is natural to ask for a classification of
Legendrian links in contact three-manifolds, and the most fundamental case of such a manifold
is, of course, the three-sphere~$\mathbb S^3$ endowed with the standard contact structure.

Since any two Legendrian links having the same topological type are related by a sequence of
Legendrian stabilizations
and destabilizations (see~\cite{futa97}), and any sequence of consecutive destabilizations
terminates, the classification of Legendrian links of a fixed topological type amounts
to the description of the set of non-destabilizable Legendrian types within this topological type
and the relations between them via stabilizations and destabilizations.

For certain topological knot types in the three-sphere endowed with the standard contact structure,
such a classification is known. Namely, Ya.\,Eliashberg and M.\,Fraser
have shown that there is only one non-destabilizable Legendrian type representing the unknot~\cite{EF1,EF2}
(which is true for any tight contact manifold, not only for~$\mathbb S^3$).
J.\,Etnyre and K.\,Honda classified Legendrian torus knots and Legendrian figure eight
knots~\cite{etho2001};
J.\,Etnyre, L.\,Ng, and V.\,V\'ertesi did so for Legendrian twist knots~\cite{etngve2013}.
A number of partial classification results for composite and satellite Legendrian knots
are obtained in~\cite{etho2003,etlafato2012,etver2016}.

Legendrian link invariants of algebraic nature constructed in~\cite{che2002,CKESJ,El,fuchs2003,ng2003,ng2011,OST,PC2005}
yield more partial results by allowing to distinguish Legendrian link types in many cases. However,
in general, even for many knot types with small crossing number, the classification of Legendrain
types is unknown. In particular, the following general folklore conjecture is unsettled.

\begin{conj}\label{finitely-many-Leg-conj}
For any link type~$\mathscr L$ in~$\mathbb S^3$, there are only finitely many non-destabilizable
Legendrian link types having topological type~$\mathscr L$.
\end{conj}

The authors' interest in contact topology comes from the tight relation of Legendrian links
with rectangular diagrams of links (or grid diagrams) and the monotonic simplification approach to
the algorithmic recognition of links.
That some nice relation between the subjects exists was observed by Bill Menasco, who shared this observation
with the first present author as early as in 2003. However, only now we can state the exact form of this relation
in full detail.

We denote by~$\xi_+$ the standard contact structure on~$\mathbb S^3$, and
by~$\xi_-$ its mirror image.
With every rectangular diagram of a link~$R$ one naturally
associates two topologically equivalent Legendrian link types denoted~$\mathscr L_+(R)$
and~$\mathscr L_-(R)$, whose Legendrianness is with respect to~$\xi_+$ and~$\xi_-$, respectively.

Exchange moves of rectangular diagrams preserve both types~$\mathscr L_+$ and~$\mathscr L_-$,
whereas stabilizations and destabilizations preserve one of them and change the other
by a Legendrian stabilization or destabilization, respectively. We assign type~I to (de)stabilizations
that preserve~$\mathscr L_+$, and type~II to those that preserve~$\mathscr L_-$.

The easy part of the relation between Legendrian links and rectangular diagrams is the
one-to-one correspondence
between $\xi_+$-Legendrian (respectively, $\xi_-$-Legendrian)
link types and equivalence
classes of rectangular diagrams of links
that is induced by the map~$R\mapsto\mathscr L_+(R)$
(respectively, $R\mapsto\mathscr L_-(R)$), where the equivalence is generated by
exchange moves and type~I (respectively, type~II) stabilizations.

To state the difficult part, we consider the \emph{groupoid of links} in~$\mathbb S^3$
in which morphisms from a link~$L_1$ to a link~$L_2$ are defined as isotopy classes of
orientation preserving homeomorphisms~$(\mathbb S^3,L_1)\rightarrow(\mathbb S^3,L_2)$.
A pair of rectangular diagrams of links~$R_1$ and~$R_2$ accompanied with
a morphism from a link represented by~$R_1$ to the one represented by~$R_2$
is called \emph{a transformation} of rectangular diagrams. In particular, every elementary move
of rectangular diagrams is viewed as a transformation.
The difficult part of the relation between Legendrian links and rectangular diagrams
is the `commutation' property of type~I moves with type~II moves, which
can be formulated, slightly informally, as follows.

\begin{theo}\label{commute-th}
Any transformation~$R\rightsquigarrow R'$ of rectangular diagrams
represented by a sequence of elementary moves can be decomposed into
two transformations~$R\rightsquigarrow R''$ and~$R''\rightsquigarrow R'$
also represented by sequences of elementary moves so that
\begin{enumerate}
\item
the sequence of elementary moves representing~$R\rightsquigarrow R''$ \emph(respectively,
$R''\rightsquigarrow R'$\emph) does not
include type~II \emph(respectively, type~I\emph) stabilizations and destabilizations\emph;
\item
the Legendrian types~$\mathscr L_-$ and~$\mathscr L_+$ associated with the
considered diagrams undergo the same Legendrian stabilizations and destabilizations
during the sequence of elementary moves representing
transformations~$R\rightsquigarrow R''$ and~$R''\rightsquigarrow R'$, respectively,
as those occurring during the original sequence representing the transformation~$R\rightsquigarrow R'$.
\end{enumerate}
Moreover, in this construction, the diagram~$R''$ is uniquely defined up to exchange moves.
\end{theo}

To prove Theorem~\ref{commute-th} we use the technique of~\cite{dp2021,dysha2023}.
This is done in Section~\ref{commute-sec} (where the precise formulation is also given).

A weaker version of the `commutation' of type~I moves with type~II moves
was established in~\cite{bypasses}, where
it allowed, among other things, to prove the so called Jones conjecture. It
also implied that Conjecture~\ref{finitely-many-Leg-conj} is equivalent to the following one.

\begin{conj}
For any link type~$\mathscr L$, there are only finitely many non-simplifiable rectangular
diagrams representing~$\mathscr L$.
\end{conj}

`Non-simplifiable' here means that the diagram does not admit a \emph{simplification},
where by a simplification we mean a sequence of elementary moves
including at least one destabilization and not including stabilizations.
Since deciding whether a given rectangular diagram admits a simplification is algorithmic,
and any sequence of simplifications terminates, the validity of the conjectures
above would mean that the monotonic simplification approach of~\cite{dyn06}
could be extended from the unknot to general links in a reasonable way.

To do so, the set of non-simplifiable rectangular diagrams has to be studied.
The above mentioned relation between rectangular diagrams and Legendrian links
allows us to fully describe the set of \emph{exchange classes}
of rectangular diagrams (i.\,e.\ rectangular
diagrams viewed up to exchange moves) in terms of Legendrian link types
and their symmetry groups. To confine to the case of
non-simplifiable rectangular diagrams, only Legendrian link types that
do not admit destabilizations should be taken into account.

For instance, the classification of Legendrian trefoils and Figure Eight knots
in~\cite{etho2001} can be used to show that there are exactly two
combinatorial types (obtained from one another by an orientation flip)
of non-simplifiable rectangular diagrams representing
each of these topological knot types.

In the present paper we use the connection between Legendrian links and rectangular diagrams
in the opposite direction. Namely, we prove the following statement, which is the main result
of this work.

\begin{theo}\label{main-th}
There exists an algorithm that, given two Rectangular diagrams of links~$R_1$ and~$R_2$,
decides whether or not~$\mathscr L_+(R_1)=\mathscr L_+(R_2)$.
\end{theo}

All combinatorial types of
rectangular diagrams having fixed complexity (where the complexity
is defined as the number of vertices) can be searched in finite time,
so the main difficulty with applying the approach mentioned above comes
from the fact that the symmetry groups of the link and of the Legendrian types in question are involved in the formulation
of the classification of exchange classes representing these Legendrian types,
and these symmetry groups are generally unknown.

To overcome
this difficulty we show that we can cope with knowing just
generating sets of those groups without further investigation of the respective group structures,
and there is an algorithm to find a generating set in terms of sequences of elementary moves.

The method of~\cite{trleg} allows to extend the present approach to transverse links and
to prove that their equivalence is also decidable.

The rest of the paper is organized as follows.
In Section~\ref{links-and-diagrams-sec}, we give the basic definitions used in the formulation of our main result.
Further definitions and auxiliary statements are given in Section~\ref{moves-sec}.
Section~\ref{commute-sec} is devoted to various results about
commutation of transformations of rectangular diagrams.
In Section~\ref{algorithm-sec}, we describe an algorithm for comparing Legendrian links,
and thus prove our main result modulo Theorem~\ref{find-gen-th}, whose proof
is given in Section~\ref{generators-sec}.
In Section~\ref{transverse-sec}, we extend the present approach to transverse
links. Finally, Section~\ref{example-sec} demonstrates several applications
of our approach to distinguishing Legendrian knot types that have not been previously proven
to be inequivalent.

Our proofs rest heavily on several previous works that contain weaker statements than
those we want to use, though the respective techniques allow to establish the stronger
result with no or very little additional argumentation.
In such cases, we do not rewrite the proof completely, which would make the paper
unreasonably long, but provide a `patch' to the original proof. This situation occurs in
the proofs of our Theorems~\ref{repr-by-elem-move-th}, \ref{r3-mor-th}, \ref{I+II=>exch-th},
and~\ref{fertile-alg-exist-th}.

\section{Legendrian links and rectangular diagrams}\label{links-and-diagrams-sec}

By~$\mathbb S^3$ we denote the unit $3$-sphere in~$\mathbb R^4$.
By \emph{a link} in~$\mathbb S^3$ we mean a compact oriented (smooth or PL-) $1$-dimensional
submanifold of~$\mathbb S^3$ whose connected components are numbered. If~$L$ is a link, and~$f:\mathbb S^3\rightarrow\mathbb S^3$ is a (smooth or PL-) homeomorphism, then it is understood that the link~$f(L)$ inherits
the orientation and numbering of components from~$L$.

\begin{defi}
Let~$\xi$ be a \emph{cooriented contact structure} in the three-sphere~$\mathbb S^3$, that is, a smooth cooriented $2$-plane
distribution that locally has the form~$\ker\alpha$, where~$\alpha$ is a differential $1$-form
such that~$\alpha\wedge d\alpha$ does not vanish.

A smooth link in~$\mathbb S^3$ is called \emph{$\xi$-Legendrian} if it is tangent to~$\xi$
at every point.
Two smooth $\xi$-Legendrian links~$L$ and~$L'$ are said to be \emph{equivalent} (or \emph{$\xi$-Legendrian isotopic}) if
there is a smooth isotopy from~$L$ to~$L'$ through regularly parametrized
$\xi$-Legendrian links. If~$\xi$
is standard (see below), then this is equivalent to saying that there is a
diffeomorphism~$\varphi:\mathbb S^3\rightarrow\mathbb S^3$ preserving~$\xi$ such that~$\varphi(L)=L'$
(see, for instance, \cite{Ge}).
\end{defi}

The concepts of a Legendrian link and Legendrian isotopy can be extended naturally to piecewise smooth links
satisfying the restriction that, at every breaking point, the angle between the one-sided tangent lines
to the curve is non-zero. We call such piecewise smooth links \emph{cusp-free}.
This extension is done in such a way that every Legendrian isotopy class of piecewise smooth
Legendrian links is an extension of a unique Legendrian isotopy class of smooth Legendrian links.

This extension is defined in~\cite{EF1,EF2,OP} via a construction called \emph{standard smoothing}.
We do it in a different but equivalent way as follows.

\begin{defi}\label{pl-leg-def}
A cusp-free piecewise smooth link~$L\subset\mathbb S^3$ is called \emph{$\xi$-Legendrian}
if it is a union of smooth arcs each of which is tangent to~$\xi$ at every point.
Two Legendrian piecewise smooth links~$L$ and~$L'$ are \emph{equivalent}
(or \emph{Legendrian isotopic}) if there is a piecewise smooth
map~$F:L\times[0;1]\times[0;1]\rightarrow\mathbb S^3$ such that
\begin{enumerate}
\item
$F(\bullet,0,0)=\mathrm{id}_L$;
\item
$F(L\times\{0\}\times\{1\})=L'$;
\item
$F(L\times\{0\}\times\{t\})$ is a piecewise smooth Legendrian link for all~$t\in[0;1]$;
\item
$F(\bullet,\bullet,t)$ is an embedding~$L\times[0;1]\rightarrow\mathbb S^3$ for all~$t\in[0;1]$;
\item
$F(\{x\}\times[0;1]\times\{t\})$ is a smooth arc transverse to~$\xi$ for all~$x\in L$, $t\in[0;1]$.
\end{enumerate}
\end{defi}

We mostly follow the settings and notation of~\cite{dp17,dp2021,dysha2023}.
We deal with two contact structures on~$\mathbb S^3$, the standard one, denoted~$\xi_+$,
and its mirror image~$\xi_-$.
To define them, we identify~$\mathbb S^3$ with the group~$\mathrm{SU}(2)$ in the standard way and use the following
parametrization of this group:
$$(\theta,\varphi,\tau)\mapsto\begin{pmatrix}\cos(\pi\tau/2)e^{\mathbbm i\varphi}&\sin(\pi\tau/2)e^{\mathbbm i\theta}\\
-\sin(\pi\tau/2)e^{-\mathbbm i\theta}&\cos(\pi\tau/2)e^{-\mathbbm i\varphi}\end{pmatrix},$$
where~$(\theta,\varphi,\tau)\in(\mathbb R/(2\pi\mathbb Z))\times(\mathbb R/(2\pi\mathbb Z))\times[0;1]$.
The coordinate system~$(\theta,\varphi,\tau)$ can also be viewed as the one coming from the join
construction~$\mathbb S^3\cong\mathbb S^1*\mathbb S^1$, with~$\theta$ the coordinate
on~$\mathbb S^1_{\tau=1}$, and~$\varphi$ on~$\mathbb S^1_{\tau=0}$. This coordinate
system will be used throughout the paper.

We define \emph{the standard contact structure}~$\xi_+$ as~$\ker\alpha_+$,
where~$\alpha_+$ is the following right-invariant
$1$-form on~$\mathbb S^3\cong\mathrm{SU}(2)$:
\begin{equation}\label{alpha+-eq}
\alpha_+(X)=\frac12\tr\left(X^{-1}\begin{pmatrix}\mathbbm i&0\\0&-\mathbbm i\end{pmatrix}dX\right)=
\sin^2\Bigl(\frac{\pi\tau}2\Bigr)\,d\theta+\cos^2\Bigl(\frac{\pi\tau}2\Bigr)\,d\varphi.
\end{equation}
Similarly, the mirror image~$\xi_-$ of the standard contact structure is~$\ker\alpha_-$,
where
\begin{equation}\label{alpha--eq}
\alpha_-(X)=\sin^2\Bigl(\frac{\pi\tau}2\Bigr)\,d\theta-\cos^2\Bigl(\frac{\pi\tau}2\Bigr)\,d\varphi.
\end{equation}

We denote by~$\mathbb T^2$ the two-dimensional torus~$\mathbb T^2=\mathbb S^1\times\mathbb S^1$,
and by~$\theta$ and~$\varphi$ the angular coordinates on the first and the second~$\mathbb S^1$-factor, respectively.

\begin{defi}
\emph{An oriented rectangular diagram of a link} is a finite subset~$R\subset\mathbb T^2$
with an assignment of~`$+$' or~`$-$' to every point in~$R$ (this assignment is referred to as \emph{an orientation})
such that every meridian~$\{\theta\}\times\mathbb S^1$
and every longitude~$\mathbb S^1\times\{\varphi\}$ contains either no or exactly two points from~$R$,
and in the latter case one of the points is assigned~`$+$' and the other~`$-$'.
The points in~$R$ are then called \emph{vertices} of~$R$, and the pairs~$\{u,v\}\subset R$
such that~$\theta(u)=\theta(v)$ (respectively, $\varphi(u)=\varphi(v)$) are called \emph{vertical edges}
(respectively, \emph{horizontal edges}) of~$R$.

Any minimal non-empty subset of~$R$ that forms a rectangular diagram of a link
is called \emph{a connected component of~$R$}.

In what follows, by a rectangular diagram of a link we always mean an \emph{oriented}
rectangular diagram of a link whose connected components are \emph{numbered}.
\end{defi}

With every rectangular diagram of a link~$R$ we associate a link,
denoted~$\widehat R$, in~$\mathbb S^3$
as follows. For any point~$v\in \mathbb T^2$ denote by~$\widehat v$ the image of the
arc~$v\times[0;1]$ in~$\mathbb S^3\cong\mathbb S^1*\mathbb S^1=(\mathbb T^2\times[0;1])/{\sim}$
oriented from~$0$ to~$1$ if~$v$ is assigned~`$+$', and from~$1$ to~$0$ otherwise.
The link~$\widehat R$ is, by definition, the union~$\bigcup_{v\in R}\widehat v$.
The connected components of~$\widehat R$ have the form~$\bigcup_{v\in R'}\widehat v$
with~$R'$ a connected component of~$R$. Thus, their numbering is naturally inherited from
the numbering of connected components of~$R$.

To have a visual presentation of~$\widehat R$ it is useful to observe that
a planar diagram of a link in~$\mathbb R^3$ equivalent to~$\widehat R$ can be obtained as follows.
Cut the torus~$\mathbb T^2$ along a meridian and a longitude not passing through a vertex of~$R$
to get a square. For every edge~$\{u,v\}$ of~$R$ join~$u$ and~$v$ by a straight
line segment, and let vertical segments overpass horizontal ones at every crossing point.
Vertical edges are oriented from `$+$' to~`$-$', and the horizontal ones from~`$-$' to~`$+$', see Figure~\ref{rdiagram-fig}.
\begin{figure}[ht]
\includegraphics[scale=.3]{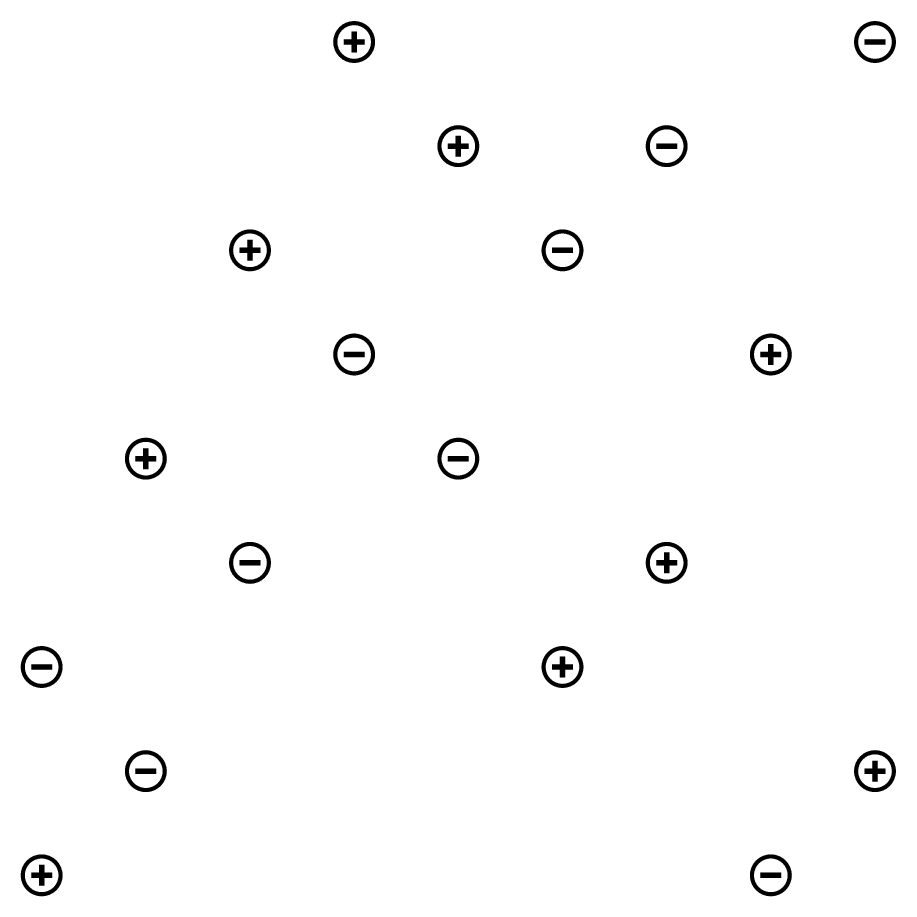}\hskip2cm
\includegraphics[scale=.3]{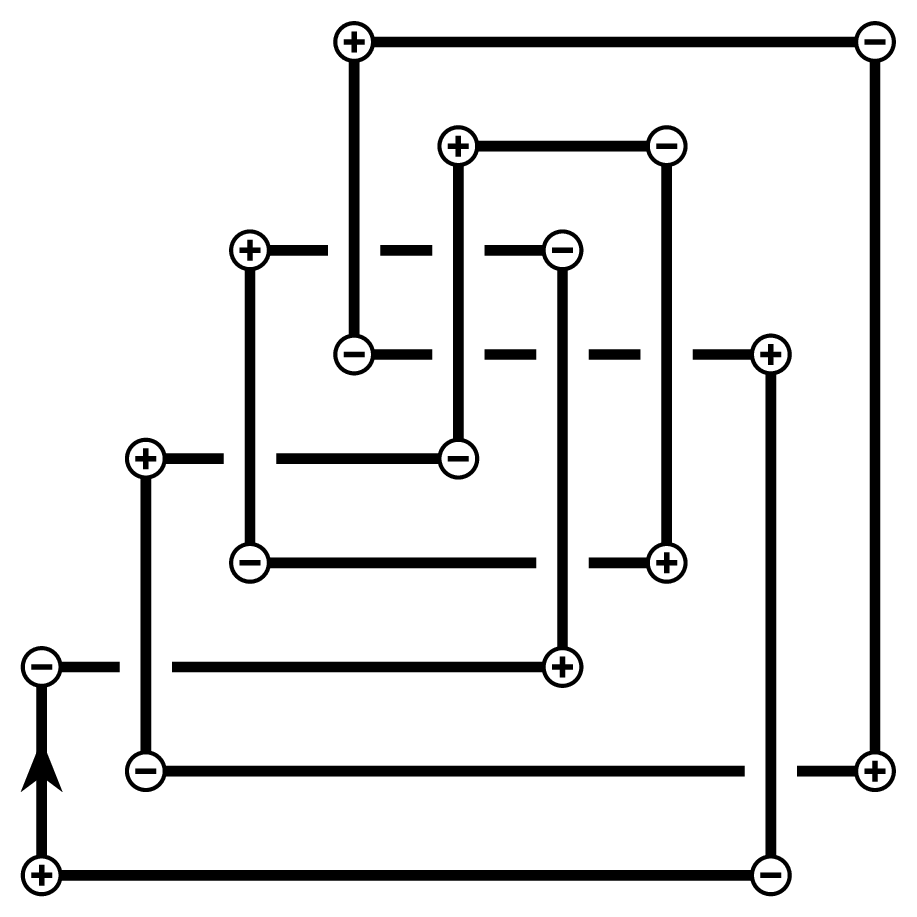}
\caption{A rectangular diagram of a knot and the corresponding planar diagram}\label{rdiagram-fig}
\end{figure}

One can see that every link  of the form~$\widehat R$, where~$R$ is a rectangular diagram,
is $\xi_+$-Legendrian and~$\xi_-$-Legendrian simultaneously. We denote by~$\mathscr L_+(R)$
and~$\mathscr L_-(R)$ the respective $\xi_\pm$-Legendrian isotopy classes of links.

It is known (see~\cite{ngth,OST}) that any $\xi_+$-Legendrian (respectively, $\xi_-$-Legendrian)
isotopy class of links has the form~$\mathscr L_+(R)$ (respectively, $\mathscr L_-(R)$)
for some rectangular diagram of a link~$R$.

One can see that the classes~$\mathscr L_\pm(R)$ are determined by the \emph{combinatorial type} of~$R$,
which is defined as the equivalence class of~$R$ with respect to the following
equivalence relation.

\begin{defi}
Two rectangular diagrams of a link~$R_1$ and~$R_2$ are \emph{combinatorially equivalent}
if one can be taken to the other by a homeomorphism of~$\mathbb T^2$
of the form~$(\theta,\varphi)\mapsto(f(\theta),g(\varphi))$, where~$f$ and~$g$
are orientation-preserving self-homeomorphisms of~$\mathbb S^1$.
\end{defi}

Clearly, the combinatorial type of any rectangular diagram can be presented in combinatorial
terms. For instance, in each class of combinatorially equivalent rectangular diagrams, one
can always pick a diagram whose vertices have coordinates that are rational multiples of~$\pi$.
When we speak about algorithms (in particular, in Theorem~\ref{main-th}) we
assume that some combinatorial way to represent rectangular diagrams
has been fixed.

\section{Moves, morphisms and symmetries}\label{moves-sec}

For two distinct points~$x,y\in\mathbb S^1$ we denote by~$[x;y]$ the arc of~$\mathbb S^1$ such that,
with respect to the standard orientation of~$\mathbb S^1$, it has the starting point at~$x$,
and the end point at~$y$. Accordingly, the interior of this arc is denoted by~$(x;y)$.

\begin{defi}\label{moves-def}
Let~$R_1$ and~$R_2$ be rectangular diagrams of a knot such that,
for some~$\theta_1,\theta_2,\varphi_1,\varphi_2\in\mathbb S^1$, the following holds:
\begin{enumerate}
\item
$\theta_1\ne\theta_2$, $\varphi_1\ne\varphi_2$;
\item
the symmetric difference~$R_1\triangle R_2$ is~$\{\theta_1,\theta_2\}\times\{\varphi_1,\varphi_2\}$;
\item
$R_1\triangle R_2$ contains an edge of one of the diagrams~$R_1$, $R_2$;
\item
none of~$R_1$ and~$R_2$ is a subset of the other;
\item
the intersection of the rectangle~$[\theta_1;\theta_2]\times[\varphi_1;\varphi_2]$
with~$R_1\cup R_2$ consists of its vertices, that is, $\{\theta_1,\theta_2\}\times\{\varphi_1,\varphi_2\}$;
\item
if~$v\in R_1\cap R_2$, then the orientation of~$v$ and the number of the component containing~$v$
are the same for~$R_1$ and~$R_2$.
\end{enumerate}
Then we say that the passage~$R_1\mapsto R_2$ is \emph{an elementary move},
and it is said to be \emph{associated} with the rectangle~$[\theta_1;\theta_2]\times[\varphi_1;\varphi_2]$.

If, additionally, the open annuli~$(\theta_1;\theta_2)\times\mathbb S^1$
and~$\mathbb S^1\times(\varphi_1;\varphi_2)$ are disjoint from~$R_1$ and~$R_2$,
then we say that the elementary move~$R_1\mapsto R_2$ is \emph{local}.

An elementary move~$R_1\mapsto R_2$ is called:
\begin{itemize}
\item
\emph{an exchange move} if~$|R_1|=|R_2|$,
\item
\emph{a stabilization move} if~$|R_2|=|R_1|+2$, and
\item
\emph{a destabilization move} if~$|R_2|=|R_1|-2$,
\end{itemize}
where~$|R|$ denotes the number of vertices of~$R$.
\end{defi}

We distinguish two \emph{types} and four \emph{oriented types} of stabilizations and destabilizations as follows.

\begin{defi}
Let~$R_1\mapsto R_2$ be a stabilization, and let~$\theta_1,\theta_2,\varphi_1,\varphi_2$ be as in Definition~\ref{moves-def}.
Denote by~$V$ the set of vertices of the rectangle~$[\theta_1;\theta_2]\times[\varphi_1;\varphi_2]$.
We say that the stabilization~$R_1\mapsto R_2$ and the destabilization~$R_2\mapsto R_1$
are of \emph{type~\rm I} (respectively, of \emph{type~\rm II}) if
$R_1\cap V\subset\{(\theta_1,\varphi_1),(\theta_2,\varphi_2)\}$
(respectively, $R_1\cap V\subset\{(\theta_1,\varphi_2),(\theta_2,\varphi_1)\}$).

Let~$\varphi_0\in\{\varphi_1,\varphi_2\}$ be such that~$\{\theta_1,\theta_2\}\times\{\varphi_0\}\subset R_2$.
The stabilization~$R_1\mapsto R_2$ and the destabilization~$R_2\mapsto R_1$
are of \emph{oriented type~$\overrightarrow{\mathrm I}$}
(respectively, of \emph{oriented type~$\overrightarrow{\mathrm{II}}$}) if they are of type~I (respectively, of type~II)
and~$(\theta_2,\varphi_0)$ is a positive vertex of~$R_2$.
The stabilization~$R_1\mapsto R_2$ and the destabilization~$R_2\mapsto R_1$
are of \emph{oriented type~$\overleftarrow{\mathrm I}$}
(respectively, of \emph{oriented type~$\overleftarrow{\mathrm{II}}$}) if they are of type~I (respectively, of type~II)
and~$(\theta_2,\varphi_0)$ is a negative vertex of~$R_2$. All these are illustrated in Figure~\ref{oriented-types-fig}.
\end{defi}

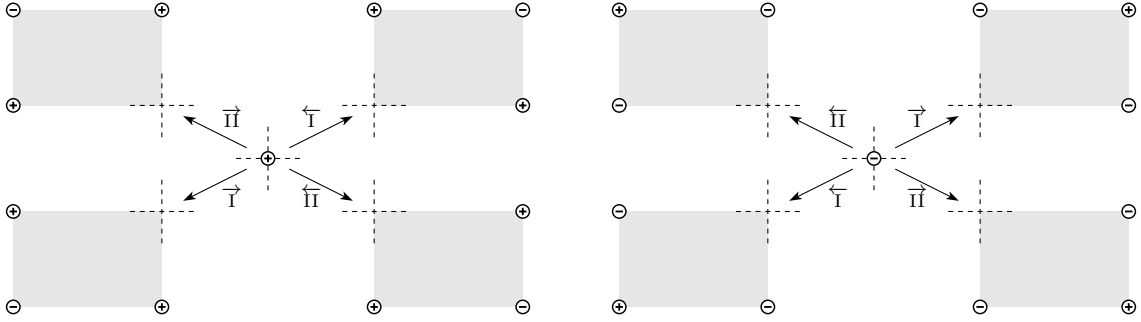
\begin{figure}[ht]
\centerline{\includegraphics[scale=.4]{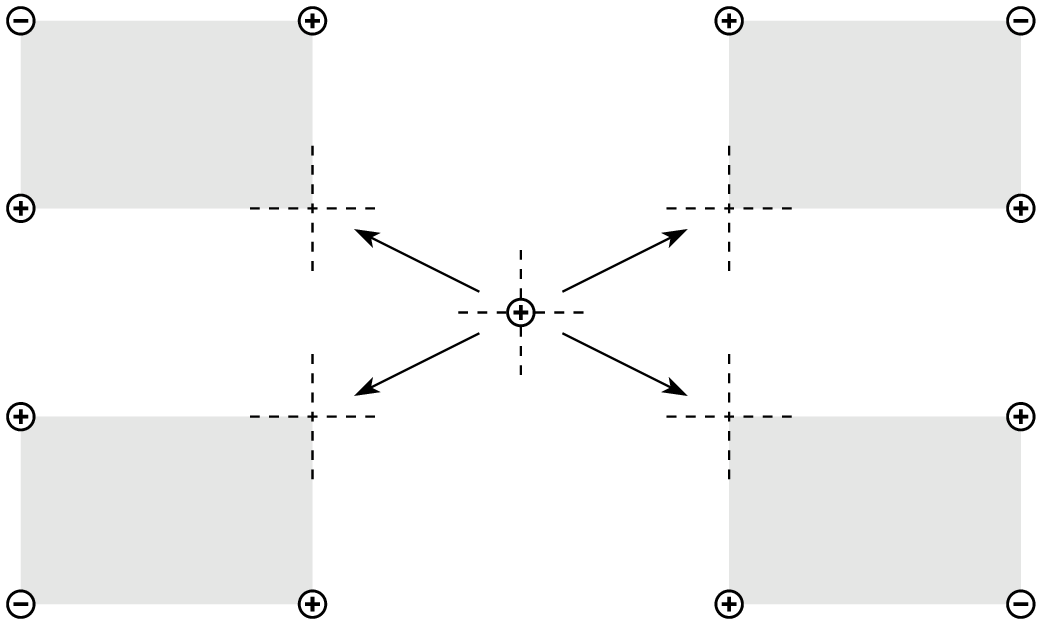}
\put(-88,72){$\scriptstyle\overleftarrow{\mathrm I}$}
\put(-118,72){$\scriptstyle\overrightarrow{\mathrm{II}}$}
\put(-88,42){$\scriptstyle\overleftarrow{\mathrm{II}}$}
\put(-118,42){$\scriptstyle\overrightarrow{\mathrm I}$}
\hskip1cm
\includegraphics[scale=.4]{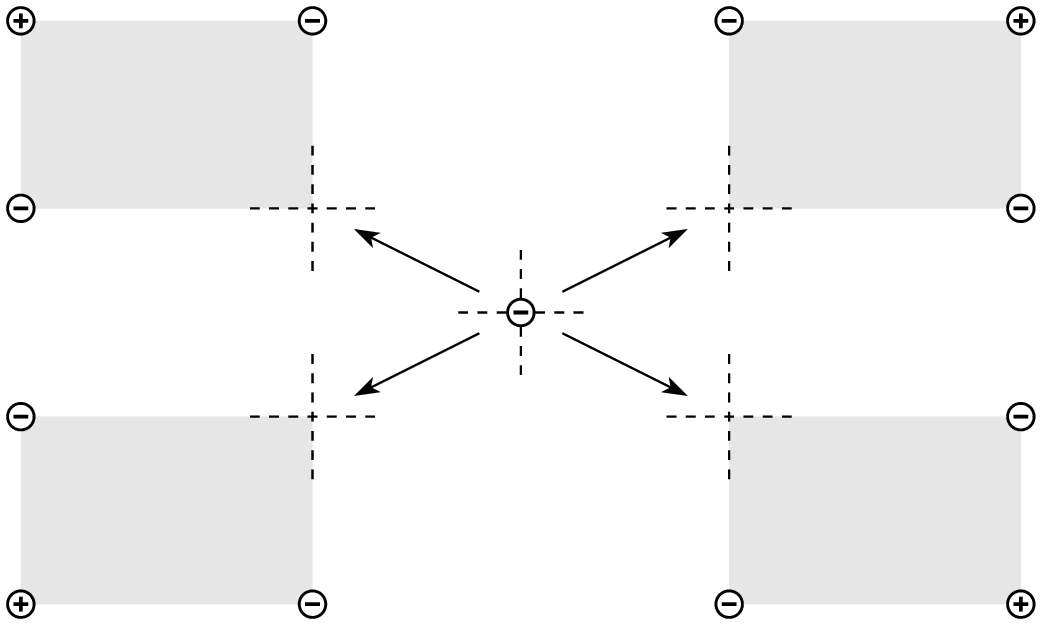}
\put(-88,72){$\scriptstyle\overrightarrow{\mathrm I}$}
\put(-118,72){$\scriptstyle\overleftarrow{\mathrm{II}}$}
\put(-88,42){$\scriptstyle\overrightarrow{\mathrm{II}}$}
\put(-118,42){$\scriptstyle\overleftarrow{\mathrm I}$}
}
\caption{Oriented types of stabilizations}\label{oriented-types-fig}
\end{figure}

Our notation for stabilization types follows~\cite{bypasses}. The correspondence with the notation of~\cite{OST} is as follows:

\centerline{\begin{tabular}{|l|c|c|c|c|}
\hline
notation of~\cite{bypasses}&
$\overrightarrow{\mathrm I}$&$\overleftarrow{\mathrm I}$&$\overrightarrow{\mathrm{II}}$&$\overleftarrow{\mathrm{II}}$\\\hline
notation of~\cite{OST}&
\emph{X:NE}, \emph{O:SW}&
\emph{X:SW}, \emph{O:NE}&
\emph{X:SE}, \emph{O:NW}&
\emph{X:NW}, \emph{O:SE}\\\hline
\end{tabular}}

\begin{rema}
The definition of elementary moves for rectangular diagrams slightly
varies from source to source, and the one given here (which follows~\cite{dysha2023}) is probably the most general one.
However, the equivalences generated by complexity preserving moves (exchange moves
and, when the $\mathbb R^2$-settings are used in the definition of a rectangular diagram,
cyclic permutations) together with stabilizations and destabilizations of selected oriented types
are the same for all definitions.
\end{rema}

The proof of the following statement can be found in~\cite{MM,ngth,OST}.

\begin{theo}\label{leg-rect-th}
\emph{(i)}
Any equivalence class of $\xi_+$-Legendrian \emph(respectively, $\xi_-$-Legendrian\emph) links
has the form~$\mathscr L_+(R)$ \emph(respectively, $\mathscr L_-(R)$\emph) for some rectangular diagram~$R$.

\emph{(ii)}
For two rectangular diagrams of links~$R_1$ and~$R_2$ we have~$\mathscr L_+(R_1)=\mathscr L_+(R_2)$
\emph(respectively, $\mathscr L_-(R_1)=\mathscr L_-(R_2)$\emph) if and only if~$R_1$ and~$R_2$ are related
by a sequence of elementary moves not including type~II \emph(respectively, type~I\emph) stabilizations
and destabilizations.
\end{theo}

\begin{defi}
Let~$L_1$ and~$L_2$ be two PL-links in~$\mathbb S^3$. By \emph{a morphism}
from~$L_1$ to~$L_2$ we call a connected component of
the space of orientation preserving PL-homeomophisms~$(\mathbb S^3,L_1)\rightarrow(\mathbb S^3,L_2)$ (which are also supposed to preserve the orientation and numbering
of components of the link).
The set of all morphisms from~$L_1$ to~$L_2$ is denoted by~$\mathrm{Mor}(L_1,L_2)$.

For any three links~$L_1,L_2,L_3$, morphisms from~$\mathrm{Mor}(L_1,L_2)$ and~$\mathrm{Mor}(L_2,L_3)$ are composed in an obvious way. For any link~$L$, this composition operation turns~$\mathrm{Mor}(L,L)$ into a group, which is called \emph{the symmetry group of~$L$} and denoted~$\mathrm{Sym}(L)$.

If~$R_1$ and~$R_2$ are rectangular diagrams, then the
notation~$\mathrm{Mor}(\widehat R_1,\widehat
R_2)$ and~$\mathrm{Sym}(\widehat R_1)$ will be simplified to~$\mathrm{Mor}(R_1,R_2)$
and~$\mathrm{Sym}(R_1)$, respectively.
\end{defi}

Due to works of Cerf~\cite{cerf}, Munkres~\cite{munk}, and Craggs~\cite{craggs} the definition
of morphisms  and symmetry groups of links would not change if arbitrary homeomorphisms are
used instead of piecewise linear ones; neither it would if PL-homeomorphisms are replaced by
diffeomorphisms provided that the links under consideration are smooth. The same refers to
more general situations when arbitrary subpolyhedra or smooth submanifolds of a three-manifold
are considered instead of links in~$\mathbb S^3$. For this reason, throughout the paper we
silently ignore the difference between smooth, topological and piecewise-linear categories
when appropriate.

\begin{rema}
Note that, the homeomorphisms used in the definition of the symmetry group~$\mathrm{Sym}(L)$
are required to preserve the orientation of~$\mathbb S^3$ as well as that of~$L$.
In the literature, it is common to define the symmetry group without these requirements.
So, the symmetry group in our sense is often smaller than the conventional one.
\end{rema}

\begin{defi}
Let~$L_1$ and~$L_2$ be two links in~$\mathbb S^3$, and let~$\eta$ be a morphism from~$L_1$
to~$L_2$. Suppose that there is an embeddded two-disc~$d\subset\mathbb S^3$ such that
the following holds:
\begin{enumerate}
\item
the symmetric difference~$L_1\triangle L_2$ is a union of two open arcs~$\alpha\subset L_1$,
$\beta\subset L_2$;
\item
$d\cap(L_1\cup L_2)=\partial d=\overline{\alpha\cup\beta}$;
\item
the morphism~$\eta$ is represented by a homeomorphism~$(\mathbb S^3,L_1)\rightarrow(\mathbb S^3,
L_2)$ that is identical outside of an open three-ball~$B$ containing the interior of~$d$
and intersecting~$L_1\cup L_2$ in~$\alpha\cup\beta$ (see Figure~\ref{alpha-beta-fig}).
\end{enumerate}
Then we say that the triple~$(L_1,L_2,\eta)$ is \emph{a $\mathbb D^2$-move associated with~$d$}.
\end{defi}
\begin{figure}[ht]
\centerline{\includegraphics{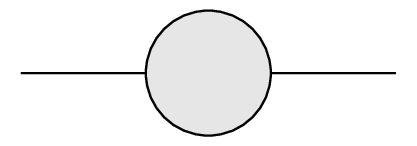}\put(-180,40){$L_1\cap L_2$}
\put(-55,40){$L_1\cap L_2$}\put(-104,32){$d$}\put(-113,68){$\alpha\subset L_1$}
\put(-113,-5){$\beta\subset L_2$}}
\caption{A $\mathbb D^2$-move}\label{alpha-beta-fig}
\end{figure}
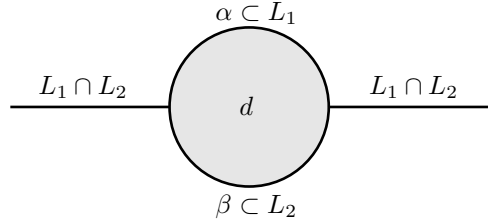

Note that if~$L_1$ and~$L_2$ in this definition are not split, then there may exist
only one morphism~$\eta\in\mathrm{Mor}(L_1,L_2)$ such that~$(L_1,L_2,\eta)$
is a $\mathbb D^2$-move. In general, there is an exact transitive action of~$\pi_2(\mathbb S^3\smallsetminus
L_1)$ on the set of $\mathbb D^2$-moves~$(L_1,L_2,\eta)$, $\eta\in\mathrm{Mor}(L_1,L_2)$.

In what follows, by \emph{a transformation} of rectangular diagrams of links we mean
a triple~$(R_1,R_2,\eta)$ in which~$R_1$ and~$R_2$ are rectangular diagrams, and~$\eta$
is an element of~$\mathrm{Mor}(R_1,R_2)$. We use the notation~$R_1\xmapsto\eta
R_2$ in this case.
One can see that all rectangular diagrams
together will all their transformations form a groupoid.
A sequence of transformations in which any two successive ones are composable
will be referred to as \emph{a chain}.

Due to the following statement all elementary moves are canonically turned into transformations
in the above mentioned sense.

\begin{prop}\label{associated-morphism-prop}
Let~$R_1\mapsto R_2$ be an elementary move associated with a rectangle~$[\theta_1;\theta_2]\times
[\varphi_1;\varphi_2]$. Then there is a unique element~$\eta$ of~$\mathrm{Mor}
(R_1,R_2)$ 
such that~$(\widehat R_1,\widehat R_2,\eta)$ is a $\mathbb D^2$-move associated
with a two-disc contained in the tetrahedron~$[\theta_1;\theta_2]*[\varphi_1;\varphi_2]$.
\end{prop}

\begin{proof}
We use the notation from~Definition~\ref{moves-def}.
The symmetric difference~$\widehat R_1\triangle\widehat R_2$ is a union of two
open arcs~$\alpha,\beta$ such that~$\alpha\subset\widehat R_1$, $\beta\subset\widehat R_2$,
and the closure~$\overline{\alpha\cup\beta}$ is the closed curve~$\{\theta_1,\theta_2\}*\{\varphi_1,\varphi_2\}
\subset\mathbb S^1*\mathbb S^1=\mathbb S^3$. This curve is unknotted, and, moreover,
there is an embedded $2$-disc~$d\subset\mathbb[\theta_1;\theta_2]*[\varphi_1;\varphi_2]$ with~$d\cap(\widehat R_1\cup\widehat R_2)=
\partial d=d\cap\partial([\theta_1;\theta_2]*[\varphi_1;\varphi_2])=\overline{\alpha\cup\beta}$.
This is a consequence of the fact that
the curve~$\overline{\alpha\cup\beta}$ lies on the boundary of the $3$-ball~$[\theta_1;\theta_2]*
[\varphi_1;\varphi_2]$ whose interior is disjoint from~$\widehat R_1\cup\widehat R_2$.
The claim follows.
\end{proof}

\begin{defi}
The morphism specified in Proposition~\ref{associated-morphism-prop} is said
to be \emph{associated} with the corresponding elementary move~$R_1\mapsto R_2$.
Thus, by saying that~$R_1\xmapsto\eta R_2$ is an elementary move we
mean that~$R_1\mapsto R_2$ is an elementary move and~$\eta$ is the corresponding
associated morphism.

If~$s$ is a chain of elementary moves~$R_0\xmapsto{\eta_1}R_1\xmapsto{\eta_2}
\ldots\xmapsto{\eta_N}R_N$,
then by~$\widehat s$ we denote the composition of all the morphisms associated with these moves:
$\widehat s=\eta_N\circ\ldots\circ\eta_2\circ\eta_1$,
and say that the morphism~$\widehat s\in\mathrm{Mor}(R_0,R_N)$ is \emph{induced by~$s$}.
\end{defi}

\begin{theo}\label{repr-by-elem-move-th}
For any rectangular diagrams of a link~$R$ and~$R'$ any morphism in~$\mathrm{Mor}(R,R')$
is induced by a chain of elementary moves.
\end{theo}

\begin{proof}
We assume that~$\mathrm{Mor}(R,R')$ is not empty, since otherwise there is
nothing to prove.

It follows from~\cite[Theorem on page~45]{crom} and~\cite[Proposition~5]{dyn06}
that \emph{some} element of~$\mathrm{Mor}(R,R')$
is induced by a chain of elementary moves. Intuitively, it is obvious that the proofs of these
statements are based on producing
a chain of elementary moves from a morphism that is picked arbitrarily at the very beginning
of the proof, and the morphism induced by the obtained chain must be the same. This implies
that, actually, \emph{any} morphism can be induced by a chain of elementary moves.
But since the question about which morphism is induced by the obtained chain of moves is
not discussed in~\cite{crom,dyn06}, we briefly sketch a more formal argument.

Let~$\eta\in\mathrm{Mor}(R,R')$. Any homeomorphism of~$\mathbb S^3$
representing~$\eta$ is PL-isotopic to identity, and, moreover, the isotopy can
be chosen so that, at every moment, only an open arc of the image of~$\widehat R$ is moving
whereas the rest of the link stays fixed. This means that there exists a sequence
$$(L_0,L_1,\eta_1),(L_1,L_2,\eta_2),\ldots,(L_{N-1},L_N,\eta_N)$$
of $\mathbb D^2$-moves associated with some two-discs~$d_1,d_2,\ldots,d_N$
such that~$L_0=\widehat R$, $L_N=\widehat R'$, and
$$\eta=\eta_N\circ\ldots\circ\eta_2\circ\eta_1.$$
The union~$L_0\cup L_1\cup\ldots\cup L_N$ is then a graph.

Let~$\phi_1,\ldots,\phi_N$ be self-homeomorphisms of~$\mathbb S^3$ representing~$\eta_1,\ldots,
\eta_N$, respectively.
Using the technique of~\cite{dp17}, one can show that
there exist transformations~$R\xmapsto{\widehat s}R_0$
and~$R'\xmapsto{\widehat s'}R_N$ with $s$ and~$s'$ chains of stabilizations
starting from~$R$ and~$R'$, respectively,
and a self-homeomorphism~$\psi$ of~$\mathbb S^3$ such that the following holds:
\begin{enumerate}
\item
$\psi$ represents both morphisms~$\widehat s$ and~$\widehat s'$;
\item
for any~$i=0,1,\ldots,N$, the link~$\psi(L_i)$ has the form~$\widehat R_i$,
where~$R_i$ is a rectangular diagram of a link;
\item
for any~$i=0,1,\ldots,N$, the two-disc~$\psi(d_i)$ is isotopic relative to~$\widehat R_i$ to a disc of the form~$\widehat\Pi_i$,
where~$\Pi_i$ is a rectangular diagram of a surface (see~\cite[Definitions~1 and~9]{dp17}).
\end{enumerate}
Denote by~$\zeta_i$ the element of~$\mathrm{Mor}(R_{i-1},R_i)$
represented by~$\psi\circ\phi_i\circ\psi^{-1}$, $i=1,\ldots,N$. We have
$$\eta=(\widehat s')^{-1}\circ\zeta_N\circ\ldots\circ\zeta_2\circ\zeta_1\circ\widehat s,$$
and for each~$i=1,\ldots,N$, the triple~$(\widehat R_{i-1},\widehat R_i,\zeta_i)$
is a~$\mathbb D^2$-move associated with~$\widehat\Pi_i$. Therefore, it suffices
to establish the statement of the theorem in the particular case of a $\mathbb D^2$-move
associated with a disc represented by a rectangular diagram of a surface.

So, from now on we suppose that~$(R,R',\zeta)$ is a $\mathbb D^2$-move
associated with a disc of the form~$\widehat\Pi$, where~$\Pi$ is a rectangular
diagram of a surface. The rectangles in~$\Pi$ can be numbered~$r_1,r_2,\ldots,r_q$
so that, for any~$i=1,2,\ldots,q$ the intersection of~$\widehat r_i$
with~$\widehat R\cup\widehat r_1\cup\ldots\cup\widehat r_{i-1}$ is an arc.
Then by removing them one-by-one from~$\Pi$ we obtain a sequence of
$\mathbb D^2$-moves~$(\widehat R^0=\widehat
R,\widehat R^1,\chi_1),(\widehat R^1,\widehat R^2,\chi_2),\ldots,(\widehat R^{q-1},\widehat
R^q=\widehat R',\chi_q)$
associated, respectively, with~$\widehat r_1,\widehat r_2,\ldots,\widehat r_q$.
Each transformation~$R^{i-1}\xmapsto{\chi_i}R^i$ is then an elementary move, and we
obviously have
$$\zeta=\chi_q\circ\ldots\circ\chi_2\circ\chi_1,$$
which concludes the proof of Theorem~\ref{repr-by-elem-move-th}
\end{proof}

\begin{theo}\label{find-gen-th}
There is an algorithm that, given a rectangular diagram of a link~$R$,
produces a finite family of chains~$s_1,s_2,\ldots,s_r$ of elementary moves
such that~$\widehat s_i$, $i=1,\ldots,r$, generate the group~$\mathrm{Sym}(R)$.
\end{theo}

The proof of this theorem is given in Section~\ref{generators-sec}.

\section{Commutations of moves}\label{commute-sec}

Each transformation of rectangular diagrams of links can be decomposed into elementary moves
in infinitely many ways. Here we present two results stating that among those decompositions there are always
ones in which elementary moves follow in a certain order. The first result, which is very easy and
mentioned already in~\cite{dyn06} in a weaker form, states that stabilizations can
always be moved to the beginning of the sequence, whereas destabilizations can be postponed
to the end. This will imply Proposition~\ref{sym-no-stab-prop}. The second result, which is much harder,
states, vaguely speaking, that transformations preserving~$\mathscr L_+$ commute with those
preserving~$\mathscr L_-$. This is the matter of Theorem~\ref{r3-mor-th}.

Let~$s$ be a chain of elementary moves~$R_1\mapsto R_2\mapsto\ldots\mapsto R_N$
each rectangular diagram in which has~$m$ connected components. For~$i\in\{1,\ldots,m\}$
and~$t\in\{\overrightarrow{\mathrm I},\overleftarrow{\mathrm I},\overrightarrow{\mathrm{II}},\overleftarrow{\mathrm{II}}\}$ we denote by~$n_{i,t}(s)$ the number of stabilizations
of oriented type~$t$ in~$s$ performed on the connected component number~$i$.

\begin{prop}\label{sym-no-stab-prop}
Let~$R$ be a rectangular diagram of a link having~$m$ components, and let~$s_1,\ldots,s_r$
be chains of elementary moves starting from and arriving at~$R$.
Suppose that the morphisms~$\widehat s_1,\ldots,\widehat s_r$ generate the entire
group~$\mathrm{Sym}(R)$.

Take any subset~$X$ of~$\{1,\ldots,m\}\times\{\overrightarrow{\mathrm I},\overleftarrow{\mathrm I},\overrightarrow{\mathrm{II}},\overleftarrow{\mathrm{II}}\}$, and let~$R'$ be a rectangular diagram of a link obtained
from~$R$ by a chain of stabilizations in which, for every pair~$(i,t)\in X$, there are~$\max_jn_{i,t}(s_j)$
stabilizations of oriented type~$t$ performed on the component number~$i$. Then any element of~$\mathrm{Sym}(\widehat{R'})$
is induced by a chain~$s$ of elementary moves such that~$n_{i,t}(s)=0$ for all~$(i,t)\in X$.
\end{prop}

To prove the statement, we need some preparation.

\begin{lemm}\label{stab-local-lem}
Any stabilization of rectangular diagrams can be decomposed into a chain
of three elementary moves in which the first move is a local stabilization, and the two
others are exchange moves.
\end{lemm}

\begin{proof}
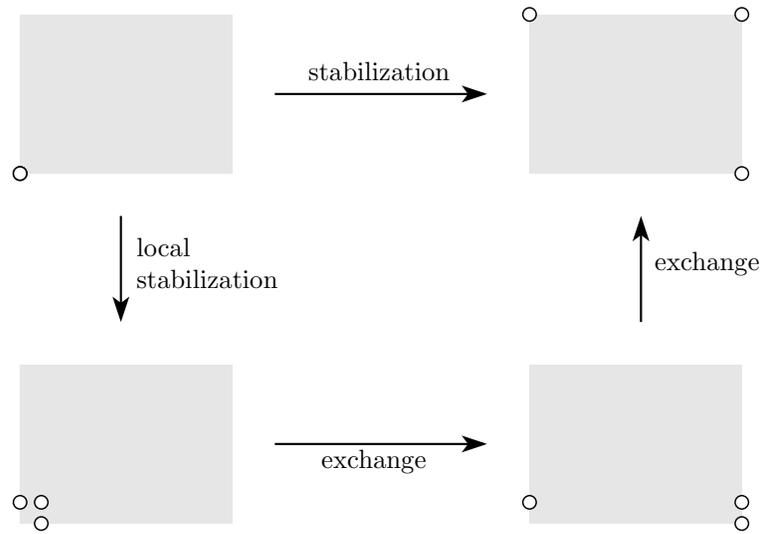
\begin{figure}[ht]
\centerline{\includegraphics[scale=.4]{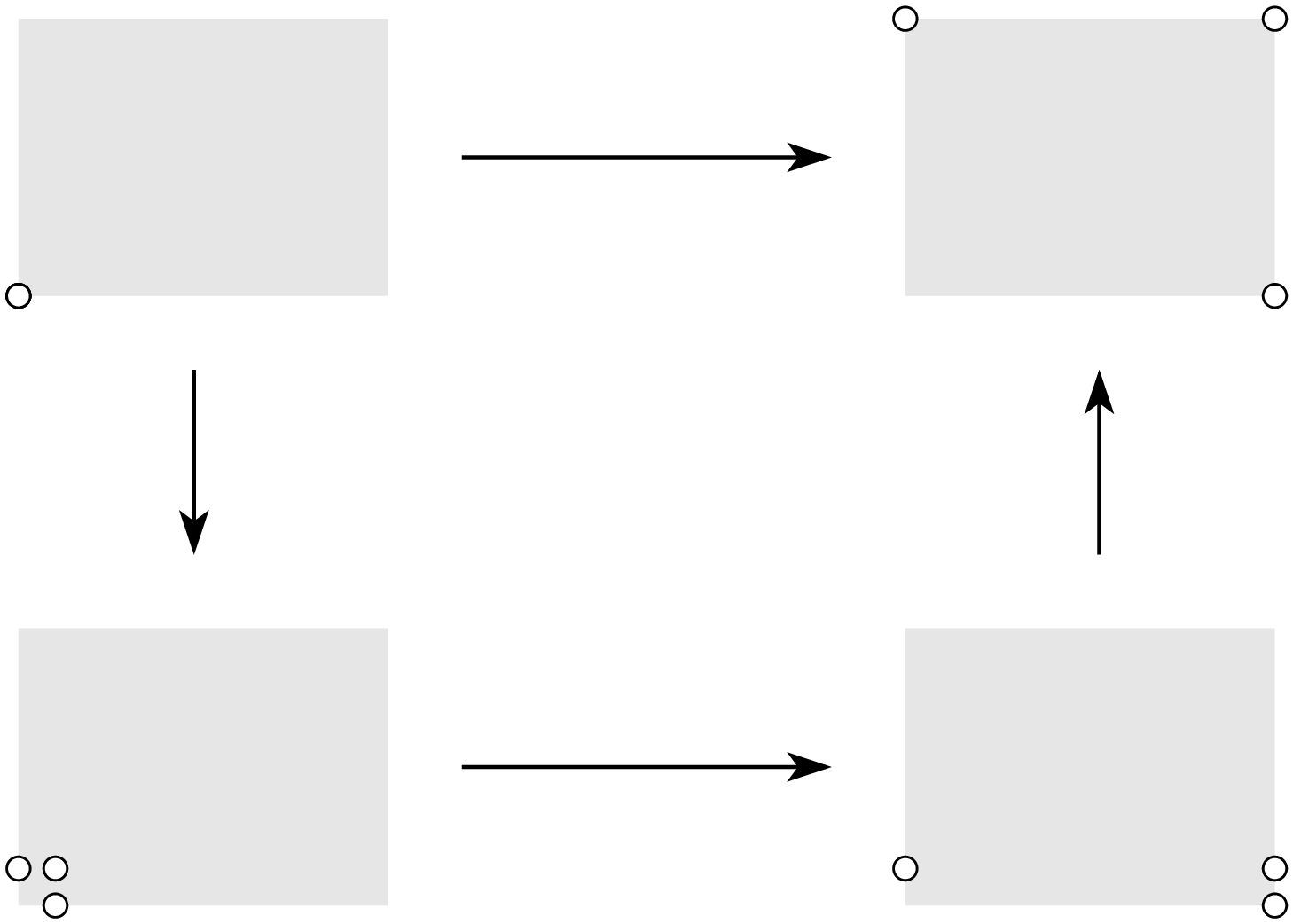}
\put(-168,172){stabilization}
\put(-233,100){\parbox{2cm}{local\\stabilization}}
\put(-163,25){exchange}
\put(-37,100){exchange}}
\caption{Decomposition of a stabilization into a local stabilization followed
by exchange moves}\label{stab-local-global-fig}
\end{figure}
The decomposition is shown in Figure~\ref{stab-local-global-fig} in one of the four possible cases.
The other cases are similar.

Let~$[\theta_1;\theta_2]\times[\varphi_1;\varphi_2]$
be the grey rectangle in the pictures. Then all changes of the links represented by these
diagrams occur in the $3$-ball~$[\theta_1;\theta_2]*[\varphi_1;\varphi_2]$.
One can disturb this $3$-ball slightly to obtain a $3$-ball~$B$ such that the interior of~$B$
intersect each of the four links in an open unknotted arc. This implies that
Figure~\ref{stab-local-global-fig} is a `commutative diagram' for the respective
morphisms.
\end{proof}

\begin{lemm}\label{stab-related-by-exch-lem}
Let~$R\xmapsto{\eta_1}R_1$ and~$R\xmapsto{\eta_2}R_2$ be two stabilizations
of the same oriented type and performed on the same connected component of~$R$.
Then the transformation~$R_1\xmapsto{\eta_2\circ\eta_1^{-1}}R_2$
can be decomposed into exchange moves.
\end{lemm}

\begin{proof}
Due to Lemma~\ref{stab-local-lem} it suffices to prove the claim in the case when
both given stabilizations are local. We can also assume that the two
vertices~$v_1$ and~$v_2$ at which the stabilizations occur form an edge of~$R$. The general
case follows from this one by induction on the length of
the shortest sequence of vertices of~$R$ starting from~$v_1$ and ending at~$v_2$
in which any two successive ones form an edge of~$R$.
\begin{figure}[ht]
\centerline{\includegraphics[scale=.5]{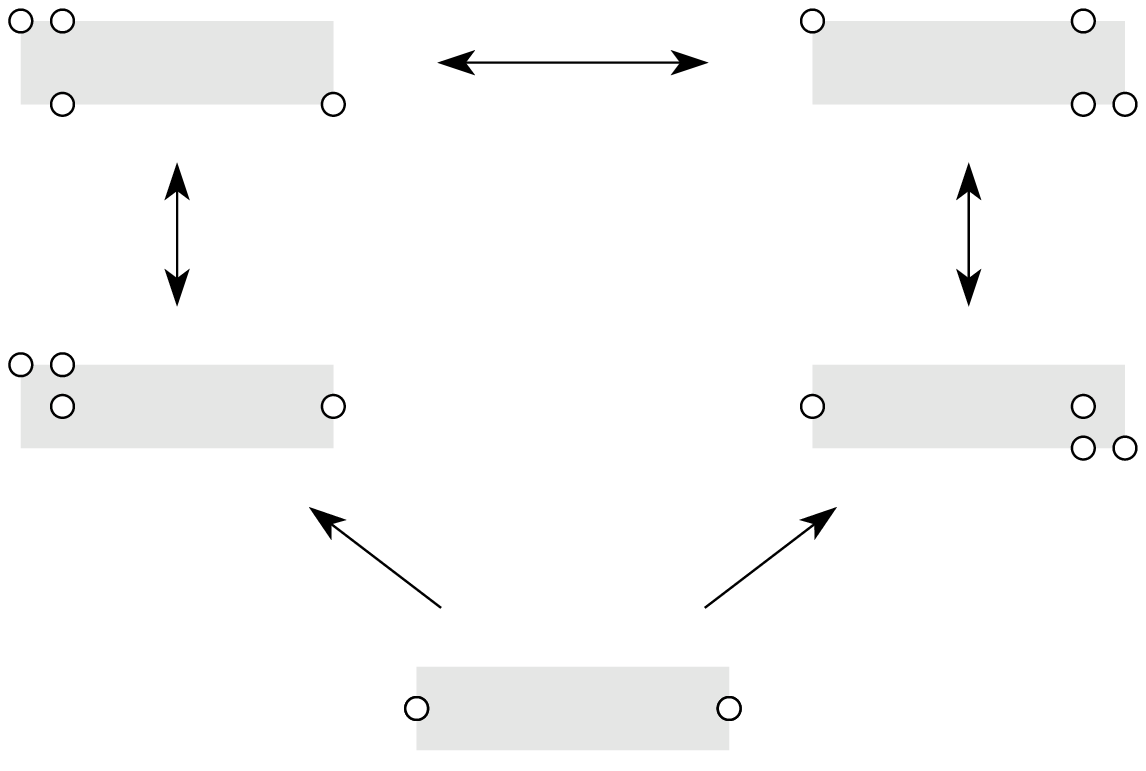}
\put(-157,175){exchange}\put(-229,127){exchange}\put(-38,127){exchange}
\put(-262,44){stabilization at $v_1$}\put(-90,44){stabilization at $v_2$}
\put(-142,-5){$R$}\put(-190,13){$v_1$}\put(-95,13){$v_2$}\put(-237,67){$R_1$}\put(-48,67){$R_2$}}
\caption{The results of two stabilizations of the same oriented type on the same
component are related by exchange moves}\label{stab-local-edge-fig}
\end{figure}
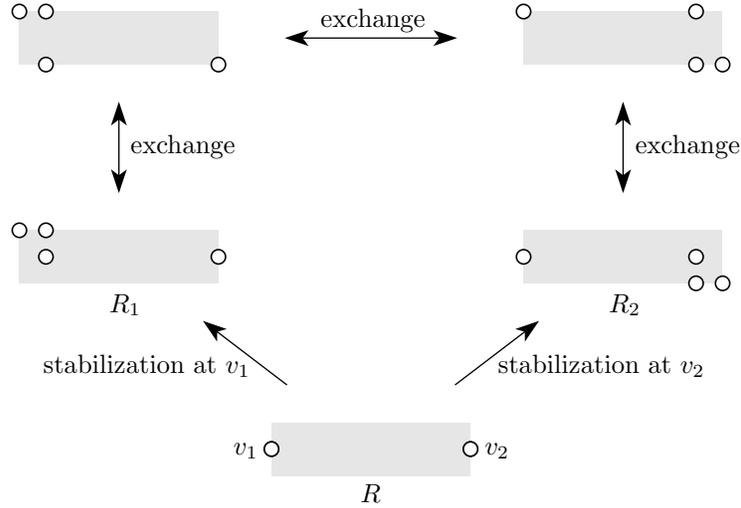

The proof in the case when the two stabilizations are local and the vertices~$v_1,v_2$
form a horizontal edge of~$R$ is illustrated in Figure~\ref{stab-local-edge-fig}. Other possible
cases are symmetric to this one.

Let the grey rectangle in the pictures of Figure~\ref{stab-local-edge-fig}
be $[\theta_1;\theta_2]\times[\varphi_1;\varphi_2]$. Then by a small perturbation
of~$[\theta_1;\theta_2]*[\varphi_1;\varphi_2]$ one can obtain a $3$-ball
that meets each of the links represented by the five diagrams in Figure~\ref{stab-local-edge-fig}
in an open arc. This ensures the `commutativity' of the diagram.
\end{proof}

Let~$R_1\mapsto R_2$ and~$R_1'\mapsto R_2'$ be elementary moves.
We say that the move~$R_1'\mapsto R_2'$ \emph{does the same thing as the
move~$R_1\mapsto R_2$ does} if the following holds:
\begin{enumerate}
\item
$R_1'\smallsetminus R_2'=R_1\smallsetminus R_2$;
\item
$R_2'\smallsetminus R_1'=R_2\smallsetminus R_1$;
\item
the two moves are associated with the same rectangle;
\item
the orientations of vertices in~$R_1'\triangle R_2'$ are the same as those in~$R_1\triangle R_2$;
\item
the connected component of~$R_1'$ modified by the move~$R_1'\mapsto R_2'$
has the same number as that modified by the move~$R_1\mapsto R_2$ has.
\end{enumerate}

\begin{lemm}\label{stab-first-lem}
Let~$R_1\mapsto R_2$ be an arbitrary elementary move, and let~$R_2\mapsto R_3$
be a stabilization. Then there is a chain of elementary moves
$$R_1\mapsto R_2'\mapsto R_3'\mapsto\ldots\mapsto R_n'=R_3$$
that induces the same morphism as the chain~$R_1\mapsto R_2\mapsto R_3$ does,
and the following holds\emph:
\begin{enumerate}
\item
the move~$R_1\mapsto R_2'$ is a stabilization having the same oriented type as~$R_2\mapsto R_3$ has
and modifying the same connected component\emph;
\item
the move~$R_2'\mapsto R_3'$ does the same thing as~$R_1\mapsto R_2$ does\emph;
\item
all other moves are exchanges.
\end{enumerate}
\end{lemm}

\begin{proof}
Let~$r$ be the rectangle with which the move~$R_1\mapsto R_2$ is associated,
and let~$t$ be the oriented type of the stabilization~$R_2\mapsto R_3$.
Choose a vertex~$v$ of~$R_1$ outside of~$r$ on the connected component having
the same number as the one modified by the stabilization~$R_2\mapsto R_3$.
Let~$R_1\mapsto R_2'$ be a local stabilization at the vertex~$v$ such that
the vertices in~$R_2'\smallsetminus R_1$ also lie outside of~$r$. Then
there is a unique rectangular diagram~$R_3'$ and an elementary move~$R_2'\mapsto R_3'$
that does the same thing as~$R_1\mapsto R_2$ does. Fix this~$R_3'$ from now on.

One can see that~$R_2\mapsto R_3'$ is then a stabilization doing the same thing as~$R_1\mapsto R_2'$
does, and the two
chains~$R_1\mapsto R_2\mapsto R_3'$ and~$R_1\mapsto R_2'\mapsto R_3'$
induce the same morphism. An application of Lemma~\ref{stab-related-by-exch-lem} concludes
the proof.
\end{proof}

\begin{proof}[Proof of Proposition~\ref{sym-no-stab-prop}]
Let~$s$ be a chain of elementary moves starting from and arriving at~$R$.
Suppose that~$s$ includes a stabilization of an oriented type~$t$
on the connected component number~$i$. Then it also must include
a destabilization of oriented type~$t$ on this component.

By induction on the number of
elementary moves preceding the first such stabilization, and using Lemma~\ref{stab-first-lem}
we see that there is a chain of elementary moves~$s'$
such that:
\begin{enumerate}
\item
the first move in~$s'$ is
a stabilization of oriented type~$t$ on the connected component number~$i$;
\item
$\widehat s=\widehat s'$;
\item
$n_{i,t}(s')=n_{i,t}(s)$.
\end{enumerate}
Applying the same argument to the reversed chain~$s^{-1}$ we can
additionally achieve:
\begin{enumerate}
\setcounter{enumi}{3}
\item
the last move in~$s'$ is a destabilization of oriented type~$t$ on the connected component number~$i$.
\end{enumerate}

Let~$R\xmapsto\eta R_1$ be a stabilization of oriented type~$t$ on the connected component number~$i$.
It follows from the above argument and Lemma~\ref{stab-related-by-exch-lem}
that there exists a chain of elementary moves~$s''$ starting from and arriving at~$R_1$
such that~$\widehat s''=\eta\circ\widehat s\circ\eta^{-1}$ and~$n_{i,t}(s'')=n_{i,t}(s)-1$.

Now, for~$(i,t)\in X$, let~$N_{i,t}=\max_jn_{i,t}(s_j)$. We may assume without loss of generality
that~$n_{i,t}(s_j)=N_{i,t}$ for all~$j\in\{1,\dots,m\}$. Indeed, if~$n_{i,t}(s_j)<N_{i,t}$
we can add more stabilizations immediatedly followed by the inverse destabilizations
to~$s_j$ to increase~$n_{i,t}(s_j)$ without changing the morphism induced by~$s_j$.

Let~$R\xmapsto{\zeta}R'$ be a chain of stabilizations that includes
exactly~$N_{i,t}$ stabilizations of type~$t$ on the connected component number~$i$
provided~$(i,t)\in X$. By induction on~$\sum_{(i,t)\in X}N_{i,t}$ we obtain
chains~$s_1',\ldots,s_r'$ of elementary moves starting from and arriving at~$R'$
such that
\begin{enumerate}
\item
$n_{i,t}(s_j')=0$ for all~$(i,t)\in X$, $j\in\{1,\ldots,m\}$;
\item
$s_j'=\zeta\circ s_j\circ\zeta^{-1}$ for all~$j\in\{1,\ldots,m\}$.
\end{enumerate}
Clearly, the morphisms~$\widehat s_j'$, $j=1,\ldots,m$, generate the symmetry group~$\mathrm{Sym}(\widehat{R'})$.
The claim follows.
\end{proof}

The following theorem is a reformulation of Theorem~7 in~\cite{bypasses}.

\begin{theo}\label{r3-th}
Suppose that two rectangular diagrams~$R_1$ and~$R_2$ represent isotopic links. Then there is a
third rectangular diagram~$R_3$ such that~$\mathscr L_+(R_3)=\mathscr L_+(R_1)$
and~$\mathscr L_-(R_3)=\mathscr L_-(R_2)$.
\end{theo}

For rectangular diagrams of a link~$R_1$ and~$R_2$ we denote by~$\mathrm{Mor}_+(R_1,R_2)$
(respectively, by~$\mathrm{Mor}_-(R_1,R_2)$)
the subset of~$\mathrm{Mor}(R_1,R_2)$
consisting of all morphisms that can be induced by a sequence of elementary
moves including only exchange moves and type~I (respectively, type~II) stabilizations
and destabilizations. We also denote by~$\mathrm{Mor}_0(R_1,R_2)$ the subset
of~$\mathrm{Mor}(R_1,R_2)$ consisting of all morphisms that can be induced by a sequence of exchange moves. Accordingly, we use the notation~$\mathrm{Sym}_\pm(R)$
and~$\mathrm{Sym}_0(R)$
for~$\mathrm{Mor}_\pm(R,R)$ and~$\mathrm{Mor}_0(R,R)$, respectively.

The method of~\cite{bypasses} allows to upgrade Theorem~\ref{r3-th} as follows.

\begin{theo}\label{r3-mor-th}
Let~$R_1$ and~$R_2$ be rectangular diagrams of a link. Then, for any morphism~$\eta\in\mathrm{Mor}(R_1,R_2)$,
there exists a rectangular diagram of a link~$R_3$ and morphisms~$\eta_+\in\mathrm{Mor}_+(R_1,R_3)$
and~$\eta_-\in\mathrm{Mor}_-(R_3,R_2)$ such that~$\eta=\eta_-\circ\eta_+$.

Equivalently\emph: any morphism~$\eta\in\mathrm{Mor}(R_1,R_2)$ is induced by a chain
of elementary moves in which all type~I stabilizations and destabilizations occur
before type~II ones.
\end{theo}

\begin{proof}
Let~$s$ be a chain of elementary moves representing~$\eta\in\mathrm{Mor}(R_1,R_2)$.
The proof is by induction on the pair~$(n_{\mathrm{II}},n)$ in which~$n_{\mathrm{II}}$ is the number of type~II
(de)stabilizations, and~$n$ is the total number of moves in~$s$. The pairs are ordered lexicographically.

Clearly, if the first move of~$s$ is not a type~II (de)stabilization or
the last move is not a type~I (de)stabilization, then we can make
the induction step by truncating~$s$.

Otherwise, to make the induction step, it suffices to consider the situation when~$n_{\mathrm{II}}=1$
and the only type~II (de)stabilization in~$s$
is the first move, and to prove, in this case, that~$\eta$
can be induced by a chain of elementary moves in which there is exactly
one type~II (de)stabilization, and it occurs after all type~I (de)stabilizations.

So, from now on, we suppose that the first move in~$s$ is a type~II stabilization or a destabilization,
and all other moves are exchanges or type~I (de)stabilizations. If the first move in~$s$
is a destabilization, then the required statement follows by induction on~$n$ from Lemma~\ref{stab-first-lem}
applied to the reversed chain~$s^{-1}$. Thus, the only non-trivial case is when the first move is a type~II
stabilization.

Let~$R_1\mapsto R_1'$ be the first move in~$s$, and let~$s'$ be the chain~$s$ with the first move deleted.
We assume that~$R_1\mapsto R_1'$ is a type~II stabilization and~$s'$ contains only exchanges and type~I stabilizations
and destabilizations (consult Figure~\ref{bypass-proof-fig}).
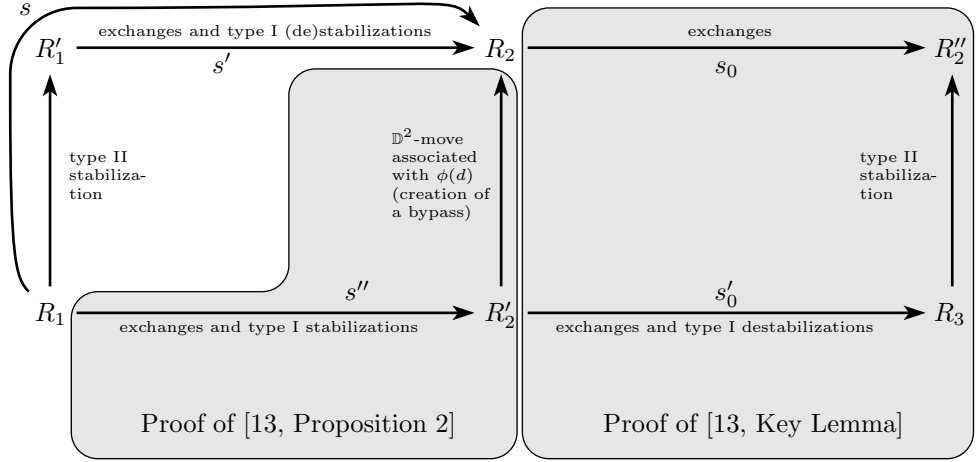
\begin{figure}[ht]
\includegraphics{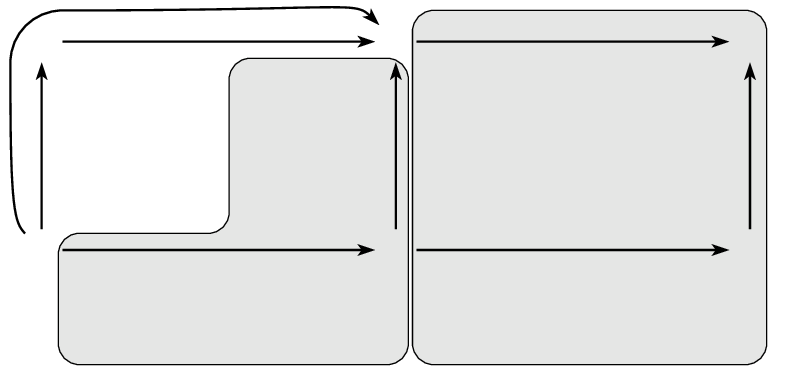}\put(-367,57){$R_1$}\put(-367,157){$R_1'$}
\put(-197,57){$R_2'$}\put(-197,157){$R_2$}
\put(-27,57){$R_3$}\put(-27,157){$R_2''$}
\put(-373,173){$s$}\put(-300,151){$s'$}
\put(-354,110){\parbox{30 pt}{\tiny type~II\\sta\-bi\-li\-za\-tion}}
\put(-55,110){\parbox{30 pt}{\tiny type~II\\sta\-bi\-li\-za\-tion}}
\put(-232,110){\parbox{40pt}{\tiny $\mathbb D^2$-move associated with~$\phi(d)$\\ (creation of\\ a bypass)}}
\put(-343,165){\tiny exchanges and type~I (de)stabilizations}
\put(-120,165){\tiny exchanges}
\put(-335,53){\tiny exchanges and type~I stabilizations}
\put(-170,53){\tiny exchanges and type~I destabilizations}
\put(-327,15){Proof of~\cite[Proposition 2]{bypasses}}
\put(-152,15){Proof of~\cite[Key Lemma]{bypasses}}
\put(-250,65){$s''$}\put(-110,151){$s_0$}\put(-110,65){$s_0'$}
\caption{Induction step in the proof of Theorem~\ref{r3-mor-th}}\label{bypass-proof-fig}
\end{figure}
Proposition~2 in~\cite{bypasses} implies, in the present terms, that
there exists a chain~$s''$ of exchanges and type~I stabilizations starting from~$R_1$
and arriving at some diagram~$R_2'$ such that the Legendrian graphs~$\widehat R_1'\cup\widehat R_1$
and~$\widehat R_2\cup\widehat R_2'$ are Legendrian isotopic (Legendrian isotopy is defined for
Legendrian graphs in exactly the same way as for links; see Definition~\ref{pl-leg-def}).
Moreover, the chain~$s''$, by construction,
`goes in parallel' with~$s'$, which means that the morphisms~$\widehat s'\in\mathrm{Mor}(R_1',R_2)$
and~$\widehat s''\in\mathrm{Mor}(R_1,R_2')$ can be represented by the same self-homemorphism~$\phi$ of~$\mathbb S^3$.

The stabilization~$R_1\mapsto R_1'$ gives rise to a $\mathbb D^2$-move~$(\widehat R_1,\widehat R_1',\eta_0)$
associated with
a two-disc~$d\subset\mathbb S^3$ such that~$d$ is tangent to~$\xi_+$
along~$\widehat R_1\smallsetminus\widehat R_1'$. Therefore, $(\widehat R_2',\widehat R_2,\widehat s'\circ\eta_0\circ\widehat s''^{-1})$
is a $\mathbb D^2$-move associated with the disc~$\phi(d)$, and the homeomorphism~$\phi$ can be adjusted
outside of~$\widehat R_2\cup\widehat R_2'$ so that~$\phi(d)$ is tangent to~$\xi_+$ along~$\widehat R_2'\smallsetminus\widehat R_2$.
Moreover, in terms of~\cite{bypasses}, the pair~$(R_2'\smallsetminus R_2,R_2\smallsetminus R_2')$ is an elementary bypass for~$R_2$
(see~\cite[Definition~7]{bypasses}).

Now, the simplification procedure for~$\phi(d)$ used in the proof of Key Lemma of~\cite{bypasses} (in which we put~$b=1$)
produces two transformations~$R_2\xmapsto{\widehat s_0}R_2''$ and~$R_2'\xmapsto{\widehat s_0'}R_3$
represented by chains of elementary moves~$s_0$ and~$s_0'$, respectively, such that the following conditions are satisfied:
\begin{enumerate}
\item
$s_0$ consists of exchange moves only;
\item
$s_0'$ consists of exchange moves and type~I destabilizations only;
\item
$s_0$ and $s_0'$ `go in parallel', which means that there is a homeomorphism~$\psi:(\mathbb S^3,\widehat
R_2,\widehat  R_2')\rightarrow(\mathbb S^3,\widehat R_2'',\widehat R_3)$
representing both morphisms~$\widehat s_0$ and~$\widehat s_0'$;
\item
the transformation~$R_3\mapsto R_2''$ is a type~II stabilization, and the corresponding $\mathbb D^2$-move~$(\widehat
R_3,\widehat R_2'',\chi)$ is associated with the disc~$\psi(\phi(d))$.
\end{enumerate}
This implies that
$$\chi\circ\widehat s_0'\circ\widehat s''=\widehat s_0\circ\widehat s'\circ\eta_0.$$
Thus, the required chain is obtained by concatenating the following four chains: $s''$, $s_0'$, $(R_3\mapsto R_2'')$,
and~$s_0^{-1}$. This concludes the induction step and the proof of Theorem~\ref{r3-mor-th}.
\end{proof}

\begin{theo}\label{I+II=>exch-th}
For any two rectangular diagrams~$R_1$ and~$R_2$ we have~$\mathrm{Mor}_0(R_1,R_2)=
\mathrm{Mor}_+(R_1,R_2)\cap\mathrm{Mor}_-(R_1,R_2)$.

In other words,
let~$R_1\xmapsto{\eta}R_2$ be a transformation of
rectangular diagrams, and let~$s',s''$ be two chains of elementary
moves inducing~$\eta$ such that
all stabilizations and destabilizations in~$s'$ \emph(respectively, $s''$\emph) are of type~I
\emph(respectively, of type~II\emph). Then~$R_1$ and~$R_2$ are related by a chain
of exchange moves.
\end{theo}

\begin{proof}
This theorem generalizes Theorem~4.2 of~\cite{dysha18} in a way that allows
to use the same proof with an inessential modification. Namely, in the proof
of Proposition~5.1 of~\cite{dysha18}, which is used to establish Theorem~4.2 there,
two embeddings~$\phi_+$ and~$\phi_-$ of a surface into~$\mathbb S^3$ are
claimed to be isotopic due to the triviality of the symmetry group of the knot,
and this is the only place where this triviality is used.

In the present situation, the equalities~$\mathscr L_+(R_1)=\mathscr L_+(R_2)$
and~$\mathscr L_-(R_1)=\mathscr L_-(R_2)$ are certified by the chains~$s'$
and~$s''$, respectively, which induce the same morphism by hypothesis. This implies
that the embeddings~$\phi_+$ and~$\phi_-$ can be chosen isotopic
due to the nature of their origin. The rest of the proof need not be changed.
\end{proof}

\begin{defi}
For a rectangular diagram of a link~$R$, we call the set of combinatorial
types of all rectangular diagrams obtained from~$R$ by a chain of exchange moves
\emph{the exchange class of~$R$} and denote it by~$[R]$.
\end{defi}

According to Theorem~\ref{leg-rect-th} the equality~$[R_1]=[R_2]$ implies~$\mathscr L_\pm(R_1)=\mathscr L_\pm(R_2)$.
For any exchange class~$c$ and a
rectangular diagram of a link~$R$ such that~$[R]=c$ we define~$\mathscr L_\pm(c)$ as~$\mathscr L_\pm(R)$.

For two rectangular diagrams of a link~$R_1$ and~$R_2$, denote by~$\Lambda(R_1,R_2)$ the set of exchange classes~$c$
such that~$\mathscr L_+(c)=\mathscr L_+(R_1)$ and~$\mathscr L_-(c)=\mathscr L_-(R_2)$.

Theorems~\ref{r3-mor-th} and~\ref{I+II=>exch-th} can be summarized as follows.

\begin{coro}\label{count-coro}
For any two rectangular diagrams~$R_1$ and~$R_2$ there is a bijection
$$\Xi:\mathrm{Sym}_+(R_2)\backslash\mathrm{Mor}(R_1,R_2)/\mathrm{Sym}_-(R_1)\rightarrow
\Lambda(R_1,R_2)$$
such that the equality~$\Xi(x)=[R_3]$ is equivalent to
$$x=\{\eta_-\circ\eta_+:\eta_+\in\mathrm{Mor}_+(R_1,R_3),\
\eta_-\in\mathrm{Mor}_-(R_3,R_2)\}.$$
\end{coro}

In many cases, this statement together with the knowledge of the symmetry
group of the given link type~$\mathscr L$ allows one to distinguish Legendrian types within~$\mathscr L$.
For instance, suppose that the symmetry group~$G$ of~$\mathscr L$ is finite,
and we have found~$m=|G|$ rectangular diagrams~$R_1,\ldots,R_m$ such that
their exchange classes are pairwise distinct and~$\mathscr L_+(R_1)=\mathscr L_+(R_2)=\ldots=\mathscr L_+(R_m)$,
$\mathscr L_-(R_1)=\mathscr L_-(R_2)=\ldots=\mathscr L_-(R_m)$. Then for any
rectangular diagram~$R$ such that the exchange class of~$R$ is not contained in~$\{[R_1],\ldots,[R_m]\}$,
we have either~$\mathscr L_+(R)\ne\mathscr L_+(R_1)$ or~$\mathscr L_-(R)\ne\mathscr L_-(R_1)$ (or both).

Examples are given in Section~\ref{example-sec}.

We conclude this section with the proof of Theorem~\ref{commute-th}].
Without condition~(2) in the formulation, this theorem is equivalent to Corollary~\ref{count-coro}.
To prove that condition~(2) can also be maintained for the proclaimed decomposition, we first write it more formally.
To this end, we introduce the following definition.

\begin{defi}
A transformation of rectangular diagrams
that can be decomposed into a chain of elementary moves in which all but one moves are exchanges
and the remaining one is a stabilization or a destabilization will be called \emph{a leap}.
It will be said to be of \emph{type~I} or \emph{type~II} depending on the type of a stabilization
or a destabilization in its decomposition.
\end{defi}

Suppose that a transformation~$R\xmapsto\eta R'$ is decomposed into
a chain of leaps
\begin{equation}\label{RRRRR-orig-eq}
R=R_0\xmapsto{\eta_1}R_1\xmapsto{\eta_2}R_2\xmapsto{\eta_3}\ldots\xmapsto{\eta_N}R_N=R'.
\end{equation}
Let~$k_1,k_2,\ldots,k_p$ and~$\ell_0,\ell_1,\ldots,\ell_{q-1}$ be the maximal subsequences of
the sequence~$0,1,2,\ldots,N$ such that all leaps~$R_{k_i-1}\xmapsto{\eta_{k_i}}R_{k_i}$
are of type~I, and all leaps~$R_{\ell_i}\xmapsto{\eta_{\ell_i+1}}R_{\ell_i+1}$ are of type~II. (We obviously have~$p+q=N$.)
We also put~$k_0=0$, $\ell_q=N$.

The first assertion of Theorem~\ref{commute-th} means that
there are a decomposition of the transformation~$R\xmapsto\eta R'$ into another chain of leaps
\begin{equation}\label{RRRRR-eq}
R=R_0'\xmapsto{\eta_1'}R_1'\xmapsto{\eta_2'}R_2'\xmapsto{\eta_3'}\ldots\xmapsto{\eta_p'}R_p'
\xmapsto{\eta_1''}R_{p+1}'\xmapsto{\eta_2''}R_{p+2}'\xmapsto{\eta_3''}\ldots\xmapsto{\eta_q''}R_N'=R'
\end{equation}
and a collection of morphisms
\begin{equation}\label{zeta-xi-eq}
\zeta_i\in\mathrm{Mor}_-(R_i',R_{k_i}),\
i=0,1,\ldots,p,\quad\text{and}\quad
\xi_{p+j}\in\mathrm{Mor}_+(R_{\ell_j},R_{p+j}'),\
j=0,1,\ldots,q,
\end{equation}
such that the leaps~$R_{i-1}'\xmapsto{\eta_i'}R_i'$, $i=1,\ldots,p$,
(respectively, $R_{p+j-1}'\xmapsto{\eta_j''}R_{p+j}'$, $j=1,\ldots,q$)
are of type~I (respectively, type~II) and the following diagrams are commutative:
\begin{equation}\label{comm-diag-eq}
\begin{matrix}R_{k_{i-1}}&\xmapsto{\eta_{k_{i-1}+1}}\ldots\xmapsto{\ \ \,\eta_{k_i}\ \ \,}&R_{k_i}
&\hskip1.5cm&
R_{\ell_{j-1}}&\xmapsto{\eta_{\ell_{j-1}+1}}\ldots\xmapsto{\ \ \,\eta_{\ell_j}\ \ \,}&R_{\ell_j}
\\[-1mm]
\hskip-2mm\begin{rotate}{-90}$\longmapsfrom$\end{rotate}\put(-17,-10){$\zeta_{i-1}$}&&
\hskip-2mm\begin{rotate}{-90}$\longmapsfrom$\end{rotate}\put(5,-10){$\zeta_i$}&&
\hskip-2mm\begin{rotate}{-90}$\longmapsto$\end{rotate}\put(-28,-10){$\xi_{p+j-1}$}&&
\hskip-2mm\begin{rotate}{-90}$\longmapsto$\end{rotate}\put(5,-10){$\xi_{p+j}$}
\\[4mm]
R_{i-1}'&\xmapsto{\hskip1.27cm\eta_i'\hskip1.27cm}&R_i'
&&
R_{p+j-1}'&\xmapsto{\hskip1.24cm\eta_j''\hskip1.24cm}&R_{p+j}'
\end{matrix}
\end{equation}
$i=1,\ldots,p$, $j=1,\ldots,q$.

\begin{lemm}\label{leap-lem}
Let~$R_1\xmapsto\xi R_2$ be a first type leap, and let~$R_3\xmapsto{\xi'}R_4$ be a transformation such that~$\xi'\in\mathrm{Mor}_+(R_3,R_4)$.
Suppose that there exist~$\zeta\in\mathrm{Mor}_-(R_2,R_4)$ and~$\zeta'\in\mathrm{Mor}_-(R_1,R_3)$ such that
$\zeta\circ\xi=\xi'\circ\zeta'$. Then~$R_3\xmapsto{\xi'}R_4$ is a leap.
\end{lemm}

\begin{proof}
Due to the symmetry~$(R_1,R_3)\leftrightarrow(R_2,R_4)$ it suffices to consider the case when~$|R_2|=|R_1|-2$,
that is, when $\mathscr L_-(R_1)\mapsto\mathscr L_-(R_2)$ is a Legendrian
destabilization. It follows from Lemma~\ref{stab-first-lem} that the transformation~$R_1\xmapsto{\zeta\circ\xi}R_4$
can be represented by a chain of elementary moves in which the last move is a type~I destabilization,
and all the others are not type~I stabilizations or destabilizations. Let~$R_3'\xmapsto{\xi''}R_4$ be the last
move in such a chain, and let~$\zeta''=(\xi'')^{-1}\circ\xi'\circ\zeta'$. We have~$\zeta''\in\mathrm{Mor}_-(R_1,R_3')$.

The morphism~$\eta=(\xi'')^{-1}\circ\xi'=\zeta''\circ(\zeta')^{-1}$ belongs to~$\mathrm{Mor}_+(R_3,R_3')\cap
\mathrm{Mor}_-(R_3,R_3')$, hence by Theorem~\ref{I+II=>exch-th} we have~$(\xi'')^{-1}\circ\xi'\in\mathrm{Mor}_0(R_3,R_3')$.
The composition~$\xi''\circ\eta=\xi'$ is, therefore, a morphism represented by a chain of elementary moves
all of which but the last one are exchange moves, and the last move is a type~I destabilization. The claim follows.
\end{proof}

\begin{proof}[Proof of Theorem~\ref{commute-th}]
We have only to show how to construct the decomposition~\eqref{RRRRR-eq} and
the morphisms~\ref{zeta-xi-eq} satisfying~\eqref{comm-diag-eq}.
By induction in~$N$ the general case can be reduced to the one when the first leap in~\eqref{RRRRR-orig-eq}
is of type~II and the last one is of type~I. This is assumed in the sequel. Observe that, in this case,
we have~$k_p=N$ and~$\ell_0=0$.

For each~$i=1,2,\ldots,p$, let~$R_0\xmapsto{\xi_i}R_i'\xmapsto{\zeta_i}R_{k_i}$ be a decomposition of the
transformation obtained by composing the first~$k_i$ transformations in the chain~\eqref{RRRRR-orig-eq}
such that~$\xi_i\in\mathrm{Mor}_+(R_0,R_i')$ and~$\zeta_i\in\mathrm{Mor}_-(R_i',R_{k_i})$.
For each~$j=1,\ldots,q-1$, let~$R_{\ell_j}\xmapsto{\xi_{p+j}}R_{p+j}'\xmapsto{\zeta_{p+j}}R_N$
be a similar decomposition of the transformation obtained by composing the last~$N-\ell_j$
transformations in~\eqref{RRRRR-orig-eq}. We also put~$\xi_0=\zeta_0=\mathrm{id}_{R_0}$
and~$\xi_N=\zeta_N=\mathrm{id}_{R_N}$.

Now, for~$i=1,\ldots,p$, we put~$\eta_i'=\xi_i\circ\xi_{i-1}^{-1}\in\mathrm{Mor}_+(R_{i-1}',R_i')$,
and for~$j=1,\ldots,q$, put~$\eta_j''=\zeta_{p+j}\circ\zeta_{p+j-1}^{-1}\in\mathrm{Mor}_-(R_{p+j-1}',R_{p+j}')$
(consult Figure~\ref{commutative-fig} for an example of the obtained commutative diagram).
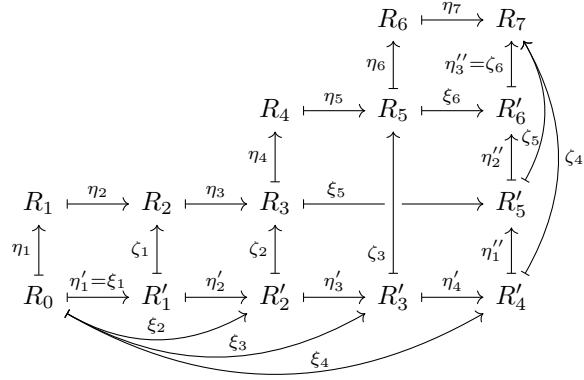
\begin{figure}[ht]
\begin{tikzcd}
&&&R_6\arrow[mapsto]{r}{\eta_7}&R_7\\
&&R_4
\arrow[mapsto]{r}{\eta_5}
&R_5
\arrow[mapsto]{u}{\eta_6}\arrow[mapsto]{r}{\xi_6}
&R_6'
\arrow[mapsto]{u}{\eta_3''=\zeta_6}\\
R_1
\arrow[mapsto]{r}{\eta_2}
&R_2
\arrow[mapsto]{r}{\eta_3}
&R_3
\arrow[mapsto]{u}{\eta_4}\arrow[mapsto]{rr}{\xi_5\hskip15mm}
&&R_5'
\arrow[bend right, mapsto]{uu}{\raisebox{-9mm}{$\zeta_5\hskip-1pt$}}
\arrow[mapsto]{u}{\eta_2''}
\\
R_0
\arrow[mapsto]{u}{\eta_1}
\arrow[mapsto]{r}{\eta_1'=\xi_1}
\arrow[bend right,mapsto]{rr}{\xi_2}
\arrow[bend right,mapsto]{rrr}{\hskip6mm\xi_3}
\arrow[bend right,mapsto]{rrrr}{\hskip12mm\xi_4}
&R_1'
\arrow[mapsto]{u}{\zeta_1}\arrow[mapsto]{r}{\eta_2'}
&R_2'
\arrow[mapsto]{u}{\zeta_2}\arrow[mapsto]{r}{\eta_3'}
&R_3'
\arrow[mapsto,crossing over]{uu}{\raisebox{-43pt}{$\zeta_3$}}\arrow[mapsto]{r}{\eta_4'}
&R_4'
\arrow[mapsto]{u}{\eta_1''}
\arrow[bend right, mapsto]{uuu}[swap]{\zeta_4}\\
\end{tikzcd}
\caption{Morphisms in the proof  of Theorem~\ref{commute-th} in the case $N=7$, $p=4$, $q=3$, $(k_0,k_1,k_2,k_3,k_4)=(0,2,3,5,7)$, $(\ell_0,\ell_1,
\ell_2,\ell_3)=(0,3,5,7)$}\label{commutative-fig}
\end{figure}
The diagrams~\eqref{comm-diag-eq} will be commutative by construction.
All transformations in the chain~\eqref{RRRRR-eq} will be leaps by Lemma~\ref{leap-lem}.
\end{proof}

\section{The algorithm}\label{algorithm-sec}
Here we prove Theorem~\ref{main-th} by presenting the scheme of an algorithm that
decides whether or not given Legendrian links are Legendrian isotopic.
The links are supposed to be represented by rectangular diagrams~$R_1$ and~$R_2$,
which are the input of the algorithm. The number of connected components of them
is denoted by~$m$.

\subsection*{Step 1.}
Check whether~$\widehat R_1$ and~$\widehat R_2$ have the same topological type.
To do so one can use the algorithm for comparing Haken's manifolds, which is described
in detail in S.\,Matveev's book~\cite{mat}. Another option is to use the upper estimate
given in~\cite{cola}
on the number of Reidemeister moves sufficient for transforming one planar link
diagram to another provided they define ambient isotopic links.

If~$\widehat R_1$ and~$\widehat R_2$ are not ambient isotopic, then return `False'.
Otherwise, proceed to the next step.

\subsection*{Step 2.}
Find a rectangular diagram of a link~$R_3$ such that~$\mathscr L_+(R_3)=\mathscr L_+(R_1)$
and~$\mathscr L_-(R_3)=\mathscr L_-(R_2)$.

Such a diagram exists by Theorem~\ref{r3-th}. To find one, we run two processes in parallel:
one enumerates combinatorial types of all rectangular diagrams that can be obtained from~$R_1$ by
a chain of elementary moves in which all stabilizations and destabilizations are of type~I,
and the other enumerates combinatorial types of all rectangular diagrams that can be obtained from~$R_2$ by
a chain of elementary moves in which all stabilizations and destabilizations are of type~II.
Once an intersection of the two lists generated in this way is encountered both processes are stopped.

\subsection*{Step 3.}
Find a family of chains of elementary moves~$s_1,\ldots,s_r$ such that
the morphisms~$\widehat s_1,\ldots\widehat s_r$ generate~$\mathrm{Sym}(R_3)$.
This is possible due to Theorem~\ref{find-gen-th}.

\subsection*{Step 4.}
Pick a chain of stabilizations starting from~$R_3$ such that, for each~$(i,t)\in\{1,\ldots,m\}\times\{\overrightarrow{\mathrm I},
\overleftarrow{\mathrm I}\}$, it includes exactly~$\max_{j=1}^rn_{i,t}(s_j)$ stabilizations of oriented type~$t$
on the component number~$i$, and no type~II stabilizations. Let~$R_3'$ be the resulting diagram.
Pick a similar chain of stabilizations starting from~$R_2$, and denote the result by~$R_2'$.
By construction, we have~$\mathscr L_+(R_1)=\mathscr L_+(R_3')$ and~$\mathscr L_+(R_2')=\mathscr L_+(R_2)$.

\begin{lemm}
We have~$\mathscr L_-(R_3')=\mathscr L_-(R_2')$.
\end{lemm}

\begin{proof}
This follows by induction on the number of stabilizations in the chain producing~$R_2'$ from~$R_2$,
and Lemma~\ref{stab-first-lem}.
\end{proof}

\subsection*{Step 5.}
Search all combinatorial types of rectangular diagrams that can be
obtained from~$R_3'$ by exchange moves. There are only finitely many of them,
since exchange moves preserve the number of vertices of the diagram. If
the combinatorial type of~$R_2'$ appears in the produced list,
then the output of the algorithm is `True'. Otherwise, `False'.
The output is correct due to the following statement.

\begin{lemm}\label{final-lem}
In the situation described above, the following two conditions are equivalent\emph:
\begin{enumerate}
\item
$\mathscr L_+(R_3')=\mathscr L_+(R_2')$\emph;
\item
$R_3'$ and~$R_2'$ can be connected by a chain of exchange moves.
\end{enumerate}
\end{lemm}

\begin{proof}
The implication~$(2)\Rightarrow(1)$ is clear.

Suppose that~$\mathscr L_+(R_3')=\mathscr L_+(R_2')$ holds. Pick a chain of elementary moves~$s$
that produces~$R_2'$ from~$R_3'$ and certifies this equality (that is, uses no type~II (de)stabilizations).

By Proposition~\ref{sym-no-stab-prop}, any element of the symmetry group~$\mathrm{Sym}(R_3')$
can be induced by a chain of elementary moves not including type~I (de)stabilizations.
Therefore, since~$\mathscr L_-(R_3')=\mathscr L_-(R_2')$, there exists a chain of elementary moves~$s'$ such that
\begin{enumerate}
\item
$s'$ produces~$R_2'$ from~$R_3'$;
\item
$s'$ does not include type~I (de)stabilizations;
\item
$s'$ induces the same morphism as~$s$ does.
\end{enumerate}
It now follows from Theorem~\ref{I+II=>exch-th} that~$R_3'$ and~$R_2'$ are related by a chain of exchange
moves.
\end{proof}

\section{Finding generators of the symmetry group}\label{generators-sec}

The aim of this section is to prove Theorem~\ref{find-gen-th}.

\begin{defi}
Let~$M$ be an oriented compact
three-dimensional PL-manifold, and let~$P_0,P_1,P_2$ be three subpolyhedra of~$M$.
We define~$\mathrm{Mor}_{M,P_0}(P_1,P_2)$ to be the set of connected components of
the space of homeomorphisms~$\phi:(M,P_0,P_1)\rightarrow(M,P_0,P_2)$
such that the restriction of~$\phi$ to~$\partial M$ is isotopic to the identity.
The elements of~$\mathrm{Mor}_{M,P_0}(P_1,P_2)$
are called \emph{morphisms from~$P_1$ to~$P_2$ relative to~$P_0$}.
If~$P_0$ is empty, then it is omitted in the notation: $\mathrm{Mor}_M(P_1,P_2)=
\mathrm{Mor}_{M,\varnothing}(P_1,P_2)$.

Similarly, for four subpolyhedra~$P_0',P_0'',P_1,P_2\subset M$
we define~$\mathrm{Mor}_{M,P_0',P_0''}(P_1,P_2)$ to be the set
of connected components of the space of homeomorphisms~$\phi:(M,P_0',P_0'',P_1)\rightarrow
(M,P_0',P_0'',P_2)$ whose restriction to~$\partial M$ is isotopic to the identity.

For two morphisms~$\eta_1\in\mathrm{Mor}_{M,P_0}(P_1,P_2)$
and~$\eta_2\in\mathrm{Mor}_{M,P_0}(P_2,P_3)$
(or $\eta_1\in\mathrm{Mor}_{M,P_0',P_0''}(P_1,P_2)$
and~$\eta_2\in\mathrm{Mor}_{M,P_0',P_0''}(P_2,P_3)$), the composition~$\eta_2\circ\eta_1$
is defined naturally. The group~$\mathrm{Mor}_M(\varnothing,\varnothing)$
is referred to as \emph{the symmetry group of~$M$} and denoted by~$\mathrm{Sym}(M)$.
The group~$\mathrm{Mor}_M(P,P)=\mathrm{Mor}_{M,P}(\varnothing,\varnothing)$
(respectively, $\mathrm{Mor}_{M,P_0}(P,P)=\mathrm{Mor}_{M,P_0,P}(\varnothing,\varnothing)$)
is denoted by~$\mathrm{Sym}_M(P)$ (respectively, by~$\mathrm{Sym}_{M,P_0}(P)$).
\end{defi}

For a morphism~$\eta\in\mathrm{Mor}_{M,P_0}(P_1,P_2)$ (or~$\eta\in\mathrm{Mor}_{M,P_0',P_0''}(P_1,P_2)$)
represented by a self-homeomor\-phism~$\phi$ of~$M$,
we denote by~$\overline\eta$ the element
of~$\mathrm{Sym}_M(P_0)$ (respectively, of~$\mathrm{Sym}_{M,P_0'}(P_0'')$)
represented by~$\phi$.

\begin{defi}
Let~$M$ be an oriented compact three-dimensional PL-manifold, and let~$\mathscr P=(P_1,\ldots,P_N)$
be a collection of subpolyhedra of~$M$.
The collection~$\mathscr P$ is said to be
\emph{fertile} if, for any~$i\in\{1,\ldots,N\}$, the subset
\begin{equation}\label{morpp-eq}
\{\overline\eta:\eta\in\mathrm{Mor}_M(P_i,P_j),\ 1\leqslant j\leqslant N\}
\end{equation}
generates the group~$\mathrm{Sym}(M)$.
\end{defi}

Recall that a compact two-dimensional polyhedron is called \emph{special}
if it is homeomorphic to a CW-complex obtained from a four-valent graph
by attaching two-cells so that each attachment map is an immersion, and
the link of any vertex of the graph in the obtained complex is homeomorphic
to the one-skeleton of the three-simplex.

\begin{theo}\label{fertile-alg-exist-th}
There exists an algorithm that, given a triangulated Haken manifold~$M$, constructs
a fertile family~$\{P_1,\ldots,P_N\}$ of special two-dimensional subpolyhedra of~$M$
such that,
for any~$i=1,\ldots,N$, the complement~$M\smallsetminus P_i$ is a union of open three-balls.
\end{theo}

\begin{proof}
To prove this theorem we revisit the proof of~\cite[Theorem 6.3.3]{mat},
which states that, for any Haken manifold with boundary pattern~$(M,\mathit\Gamma)$,
one can find a characteristic family of special two-dimensional subpolyhedra of~$M$
such that, for each subpolyhedron~$P$ in the family, the complement~$M\smallsetminus P$
is a union of open three-balls. We show below how to modify
this proof to establish our statement.

Recall from~\cite{mat} that by \emph{a boundary pattern} on a compact oriented three-manifold~$M$
one calls a three-valent graph embedded in~$\partial M$ (with no restriction on loops or multiple edges).
One can see that a boundary pattern~$\mathit\Gamma$ can be chosen so that
any self-homeomorphism~$\phi$ of~$\partial M$ that takes~$\mathit\Gamma$ to itself
is isotopic to identity, and~$(M,\mathit\Gamma)$ is Haken provided that~$M$ considered without
boundary pattern is Haken.
We assume that such a~$\mathit\Gamma$ is fixed from now on.
Forgetting~$\mathit\Gamma$ gives rise to an epimorphism~$\mathrm{Sym}_M(\mathit\Gamma)\rightarrow\mathrm{Sym(M)}$.

Recall also that `\emph{characteristic}' in~\cite{mat} means that the produced
set of subpolyhedra viewed up to a homeomorphism preserving~$\mathit\Gamma$
(there may be duplicates
when the equivalence is taken into account, which are discarded) is a complete topological invariant of~$M$.
In particular,  it depends only on the homeomorphism type
of~$(M,\mathit\Gamma)$, but not on the given triangulation.

Now we overview the general principles used by Matveev to construct a characteristic family, using
slightly different terminology and settings.

If~$P'$ is a two-dimensional
subpolyhedron of~$M$, and~$P$ is a proper subpolyhedron of~$P'\subset M$ we say that~$P'$ is obtained from~$P$
by \emph{an extension}.
Two extensions~$P\mapsto P'$ and~$P\mapsto P''$ are said to be \emph{of the same topological type}
if the set~$\mathrm{Mor}_{M,\mathit\Gamma,P}(P',P'')$ is not empty.
If, additionally, this set contains a morphism~$\eta$ such that~$\overline\eta=1\in\mathrm{Sym}_{M,\mathit\Gamma}(P)$,
then the extensions~$P\mapsto P'$ and~$P\mapsto P''$ are said to be \emph{equivalent}.

Matveev introduces an elaborate set of rules for defining extensions that are used to produce a characteristic
family of polyhedra, and the rules come in two kinds, those formulated in purely
topological terms, and triangulation-dependent ones. We call \emph{allowable} all
extensions that satisfy the topological rules introduced in~\cite{mat}.
(Note that which extensions are allowable depends on the chosen boundary
pattern~$\mathit\Gamma$, and this is the way how~$\mathit\Gamma$ affects the construction below.)
Any subpolyhedron of~$M$
produced by a finite number of consecutive allowable extensions
from the boundary~$\partial M$ is called \emph{admissible}.

The algorithm for producing a characteristic family of polyhedra for~$M$ in~\cite{mat} works as follows.
We start with the family~$\mathscr P_0(M,\mathit\Gamma)=\{\partial M\}$ and put~$k=0$. Then
we implement the following steps in a loop:
\begin{itemize}
\item
check whether~$\mathscr P_k(M,\mathit\Gamma)$ contains a polyhedron~$P$ such that~$M\smallsetminus P$ is a union of
open three-balls;
\item
if true, then the process is terminated and the family
$$\mathscr P=\{P\in\mathscr P_k(M,\mathit\Gamma):M\smallsetminus P\text{ is a union of open three-balls}\}$$
is given as the output;
\item
otherwise, compute~$\mathscr P_{k+1}(M,\mathit\Gamma)$ defined as the family of all polyhedra that can be obtained
from a polyhedron in~$\mathscr P_k(M,\mathit\Gamma)$ by allowable extensions satisfying the
triangulation-dependent restrictions.
\end{itemize}

The topological and triangulation-dependent extension rules in~\cite{mat} are designed so as to
satisfy the following properties:
\begin{itemize}
\item
given an admissible polyhedron~$P\subset M$ such that~$M\smallsetminus P$ is not a union
of open three-balls, there is a non-empty enumerable
set~$\mathscr X(P)$ of extensions of~$P$ such that each
equivalence class of allowable extensions~$P\mapsto P'$
has a representative in~$\mathscr X(P)$;
\item
the subset~$\mathscr E(P)$ of~$\mathscr X(P)$ consisting of all extensions that satisfy triangulation-dependent
restrictions contains at least one representative in each topological type of
allowable extensions of~$P$;
\item
the subset~$\mathscr E(P)$ is finite and computable;
\item
there is a number~$K$ depending on the topology of~$(M,\mathit\Gamma)$
such that at most~$K$ consecutive allowable extensions can be applied to~$\partial M$.
\end{itemize}
It follows from these properties that
\begin{itemize}
\item
for each applicable~$k$, the set of homeomorphism types of triples~$(M,\mathit\Gamma,P)$
with~$P\in\mathscr P_k(M,\mathit\Gamma)$, depends only on the homeomorphism type of~$(M,\mathit\Gamma)$;
\item
the process terminates in finitely many steps.
\end{itemize}
This, in turn, implies that the output is a  characteristic family of polyhedra.

What we need to modify in this machinery is the triangulation-dependent rules and,
thus, the definition of~$\mathscr E(P)$, where~$P$ is an admissible subpolyhedron of~$M$.
We do it so as to have, additionally, the following:
\begin{equation}\label{newE-eq}
\forall P'\in\mathscr E(P)\text{ the set }
\{\overline\eta:\eta\in\mathrm{Mor}_{M,\mathit\Gamma,P}(P',P''),\,
P''\in\mathscr E(P)\}\text{ generates }\mathrm{Sym}_{M,\mathit\Gamma}(P).
\end{equation}
Before we explain how to achieve this, let us see why this suffices to
prove the theorem. We claim that, with the new definition of~$\mathscr E(P)$,
each family~$\mathscr P_k$ produced
using the above scheme satisfies the following condition:
\begin{equation}\label{superfertile-eq}
\forall P\in\mathscr P_k\text{ the set }\{\overline\eta:
\eta\in\mathrm{Mor}_{M,\mathit\Gamma}(P,P'),\,P'\in\mathscr P_k\}
\text{ generates }\mathrm{Sym}_M(\mathit\Gamma),\end{equation}
which, in turn, implies that all sets~$\mathscr P_k(M,\mathit\Gamma)$ as well
as the output of the algorithm are fertile families for~$M$.
The claim is proved by induction on~$k$. For~$k=0$ it holds trivially,
since~$\mathscr P_0=\{\partial M\}$ and~$\{\overline\eta:
\eta\in\mathrm{Mor}_{M,\mathit\Gamma}(\partial M,\partial M)\}$
coincides with~$\mathrm{Sym}_M(\mathit\Gamma)$.

Suppose that~\eqref{superfertile-eq} holds for~$k=k_0$ and there is no polyhedron~$P$
in~$\mathscr P_k(M,\mathit\Gamma)$ such that~$M\smallsetminus P$ is a union of open balls.
Let~$P_1\in\mathscr P_{k_0+1}$, and let~$P_1$ be obtained from~$P_0\in\mathscr P_{k_0}$
by an extension.

To make the induction step, it suffices to show that, for any~$P_0'\in\mathscr P_{k_0}$
and any morphism~$\eta\in\mathrm{Mor}_{M,\mathit\Gamma}(P_0,P_0')$
the element~$\overline\eta\in\mathrm{Sym}_M(\mathit\Gamma)$ can be decomposed
into a product of elements from
\begin{equation}\label{inductionstep-eq}
\{\overline\zeta:\zeta\in\mathrm{Mor}_{M,\mathit\Gamma}(P_1,P_1'),\,
P_1'\in\mathscr P_{k_0+1}\}.
\end{equation}
To this end, pick a homeomorphism~$\phi:(M,\mathit\Gamma,P_0)\rightarrow(M,\mathit\Gamma,P_0')$
representing~$\eta$.

Since the set~$\mathscr E(P_0')$ contains representatives of every topological
type of extensions, there exist~$P_1'\in\mathscr E(P_0')$
and a homeomorphism~$\psi:(M,\mathit\Gamma,P_0,P_1)\rightarrow(M,\mathit\Gamma,P_0',P_1')$.
Let~$\zeta$ be the element of~$\mathrm{Mor}_{M,\mathit\Gamma}(P_1,P_1')$
represented by~$\psi$, and~$\chi$ the element of~$\mathrm{Sym}_{M,\mathit\Gamma}(P_0)$
represented by~$\psi^{-1}\circ\phi$. Then we have
$$\overline\eta=\overline\zeta\circ\overline\chi.$$
The element~$\overline\zeta\in\mathrm{Sym}_{M,\mathit\Gamma}$
belongs to the set~\eqref{inductionstep-eq} by construction.
The element~$\overline\chi$ belongs to the subgroup generated by the set~\eqref{inductionstep-eq}
by virtue of~\eqref{newE-eq}.

Thus, it remains only to show that the extension rules can be modified
so that~\eqref{newE-eq} holds, without violating other properties of
the sets~$\mathscr E(P)$ mentioned above.

In what follows we denote by~$\mathscr E_0(P)$, where~$P$ is an admissible subpolyhedron of~$M$,
the set of all extensions of~$P$ satisfying Matveev's extension rules. The set~$\mathscr E(P)\subset\mathscr X(P)$
we substitute it with will be larger, $\mathscr E(P)\supset\mathscr E_0(P)$. So, we should
be concerned, in addition to~\eqref{newE-eq}, only with ensuring
that~$\mathscr E(P)$ is finite and computable.

For an allowable extension~$P\mapsto P'$, we denote by~$[P']_P$
the equivalence class of this extension. The group~$\mathrm{Sym}_{M,\mathit\Gamma}(P)$ acts
on the set~$\{[P']_P:P'\in\mathscr X(P)\}$ in an obvious way. The orbit of~$[P']_P$
under this action will be denoted by~$\mathscr O_P(P')$.

As we will see, for every~$P'\in\mathscr E_0(P)$,
there is a subgroup~$G(P,P')\subset\mathrm{Sym}_{M,\mathit\Gamma}(P)$
with a computable set of generators, and a finite collection of
allowable extensions~$P\mapsto P_i'\in\mathscr E_0(P)$, $i=1,\ldots,s$,
having the same topological type as~$P\mapsto P'$ has such that
\begin{equation}\label{gpp-eq}
\bigcup_{i=1}^s\bigl(G(P,P'_i)\cdot[P'_i]_P\bigr)=\mathscr O_P(P').
\end{equation}
For each~$P'\in\mathscr E_0(P)$ we compute some set of generators~$\{g_1,\ldots,g_p\}$
of~$G(P,P')$ and, for each~$i=1,\ldots,p$, include in~$\mathscr E(P)$
a representative of~$g_i\cdot[P']_P$. One can see that this ensures~\eqref{newE-eq}.

There are several exceptional Haken manifolds that should be treated separately.
These are the orientable Seifert manifolds~$(\mathbb S^2;1/2,-1/3,-1/6)$,
$(\mathbb S^2;1/2,-1/4,-1/4)$, $(\mathbb S^2;2/3,-1/3,-1/3)$, $(\mathbb R\mathrm P^2;1/2,-1/2)$,
orientable manifolds that fiber over the two-torus or the Klein bottle with fiber a circle,
and orientable manifolds that fiber over a closed surface with fiber a closed interval.
In each of these cases, the symmetry group is known, and a set of generators
can be produced explicitly. If~$M$ is from this list, Matveev's algorithm
allows to detect this fact and to reveal the (Seifert) fibration structure.
After that the construction of a fertile family of subpolyhedra
of~$M$ is quite straightforward.
We assume from now on that the given manifold~$M$ is not in this list.

Let~$P\mapsto P'$ be an allowable extension with~$P'\in\mathscr E_0(P)$. We denote
by~$N$ the closure of the union of all the connected components of~$M\smallsetminus P$ having
a non-empty intersection with~$P'$, and by~$\mathit\Delta$ the  boundary pattern of~$N$
that arises in Matveev's algorithm: $\mathit\Delta=(\overline{P\smallsetminus \partial N}\cup\mathit\Gamma)\cap N$.
We also denote by~$\mathrm{Sym}(N,\mathit\Delta)$ (respectively, by~$\mathrm{Sym}(N;P)$)
the group whose elements are connected components
of the space of self-homeomorphisms~$\phi$ of~$(N,\mathit\Delta)$ (respectively, of~$(N,\mathit\Delta,P\cap N)$)
such that the restriction of~$\phi$ to~$\partial N$
is isotopic to identity within the class of homeomorphisms preserving~$\mathit\Delta$
(respectively, coincides with~$\mathrm{id}_{\partial N}$).
The subgroups of~$\mathrm{Sym}(N,\mathit\Delta)$ and~$\mathrm{Sym}(N;P)$
consisting of elements that can be represented by a homeomorphism
preserving each part of the JSJ-decomposition of~$N$
are denoted, respectively, by~$\mathrm{Sym}_{\mathrm{JSJ}}(N,\mathit\Delta)$
and~$\mathrm{Sym}_{\mathrm{JSJ}}(N;P)$.

Any homeomorphism representing an element of~$\mathrm{Sym}(N;P)$
can be continued to a self-homeomorphism of~$M$ whose restriction to~$\overline{M\smallsetminus N}$
is the identity map. This defines an inclusion~$\mathrm{Sym}(N;P)
\hookrightarrow\mathrm{Sym}_{M,\mathit\Gamma}(P)$.
So, we may view the groups~$\mathrm{Sym}(N;P)$ and~$\mathrm{Sym}_{\mathrm{JSJ}}(N;P)$
as subgroups of~$\mathrm{Sym}_{M,\mathit\Gamma}(P)$.

The orbit~$\mathrm{Sym}_{\mathrm{JSJ}}(N;P)\cdot[P']_P$
may be a proper subset of~$\mathscr O_P(P')$, since
not any homeomorphism representing an element of~$\mathrm{Sym}_{M,\mathit\Gamma}(P)$
has to preserve~$N$ or be isotopic to identity on~$(\partial N,\mathit\Delta)$,
and not any homeomorphism representing an element of~$\mathrm{Sym}(N;P)$
has to preserve each part of the JSJ-decomposition of~$N$.
However, the algorithm in~\cite{mat}
is designed so that, for any homeomorphism~$\phi$ representing an element of~$\mathrm{Sym}_{M,\mathit\Gamma}(P)$, the orbit of the equivalence class of the extension~$P\mapsto\phi(P')$
under the action~$\mathrm{Sym}_{\mathrm{JSJ}}(\phi(N);P)$ is represented by at least one extension~$P\mapsto P''$
with~$P''\in\mathscr E_0(P)$.

Thus, to have~\eqref{gpp-eq} hold true it would suffice to let~$P_1',P_2',\ldots,P_s'$ be all elements of~$\mathscr E_0(P)\smallsetminus\{P'\}$ such that the extensions~$P\mapsto P_i'$, $i=1,\ldots,s$,
have the same topological type as~$P\mapsto P'$ has, and to define~$G(P,P')$ so as
to ensure the following:
\begin{equation}\label{gpp0-eq}
\mathrm{Sym}_{\mathrm{JSJ}}(N;P)\cdot[P']_P=G(P,P')\cdot[P']_P.
\end{equation}

Now we consider all kinds of extensions defined in~\cite{mat} and show, in each case,
how to achieve~\eqref{gpp-eq}. Recall that Matveev defines 13 types of extensions,
which are denoted $E_1$, $E_2$, $E_3$, $E_4$, $E_5$, $E_6$, $E_7$, $E_8$, $E_9$,
$E_3'$, $E_4'$, $E_4''$, $E_{10}$ and listed here in the order of their priority.

If the extension~$P\mapsto P'$ is of one of the types $E_3$, $E_4$, $E_3'$, $E_4'$, or $E_4''$,
then one can see that the orbit~$\mathscr O_P(P')$ is finite and no element of this orbit
is sorted out due to the triangulation-dependent rules of~\cite{mat}. Therefore, each element of~$\mathscr
O_P(P')$
already has a representative in~$\mathscr E_0(P)$, and we may put~$G(P,P')=\{1\}$ to
ensure~\eqref{gpp-eq}.

Suppose that~$P\mapsto P'$ is an extension of one of the types~$E_1$, $E_2$, $E_7$, or~$E_8$.
This means this extension is associated, in a certain way, with a two-torus (in the~$E_1$ and~$E_7$ cases) or a longitudinal
annulus (in the~$E_2$ or~$E_8$ cases) $S\subset N$ which is \emph{essential}
in the sense of~\cite[Definition 6.3.15]{mat}.

Namely, in the~$E_1$ and~$E_2$ cases the extension~$P\mapsto P'$
amounts to the addition of~$S$ or two parallel copies of~$S$. In the~$E_7$ and~$E_8$
cases the extension~$P\mapsto P'$ is defined by~$S$ in a more complicated way as follows.
The intersection~$P\cap(N\smallsetminus\partial N)$, in these two cases, is not empty,
and its closure cuts~$N$ into `slices' each of which has the structure of a fiber bundle
over a compact surface with fiber a closed interval. First, the surface~$S$ is deformed by
an isotopy to become \emph{vertical}, which means that, after a proper deformation of
the fiber bundles, the intersection of~$S$
with each `slice' is a union of fibers. Then two parallel copies of~$S$ are added to~$P$.
Finaly, the four-valent edges of the obtained polyhedron are resolved in order
to make the polyhedron simple (see \cite[Definition 1.1.8]{mat}).

Such a resolution may not be unique, so there may be
several extensions associated with a single surface~$S$.
What is crucial here is that the obtained family of extensions is finite and the set of their equivalence
classes depends only on
the isotopy class~$[S]$ of the surface~$S$ in~$(N,\mathit\Delta)$. What depends on the triangulation of~$M$ here
is which concrete surface~$S$ is picked among others whose isotopy
classes are obtained from~$[S]$ by the action of~$\mathrm{Sym}_{\mathrm{JSJ}}(N,\mathit\Delta)$.

The finiteness of~$\mathscr E_0(P)$ in the case of extensions of type~$E_1,E_2,E_7,E_8$
is deduced from \emph{The Third Finiteness Property} \cite[Theorem 6.4.44]{mat}, which
states that there are finitely many orbits of isotopy classes of surfaces suitable for~$S\subset N$
under the action of~$\mathrm{Sym}(N,\mathit\Delta)$, and a collection of surfaces representing
all these orbits can be found algorithmically. To achieve our goal, we make a closer look at
the proof of The Third Finiteness Property. In particular, the analysis of the proof shows
that the algorithm in~\cite{mat} finds a representative in each orbit
of the $\mathrm{Sym}_{\mathrm{JSJ}}(N,\mathit\Delta)$-action on the set  of isotopy classes of surfaces
suitable for~$S$.

There are two possible situations here:
\begin{enumerate}
\item
$S$ is a JSJ-surface of~$N$;
\item
$S$ is contained in a Seifert fibered manifold~$N'\subset N$ obtained as a part of the JSJ-decomposition of~$N$,
and~$S$ is a union of non-exceptional fibers of an algorithmically constructed Seifert fibration.
\end{enumerate}
If~$S$ is a JSJ-surface of~$N$, then
the entire orbit of~$[P']_P$ under the action of~$\mathrm{Sym}_{M,\mathit\Gamma}(P)$
is finite and included into~$\mathscr E_0(P)$. We may put~$G(P,P')=\{1\}$ in this case to have~\eqref{gpp-eq}.

Suppose that~$S$ is contained in a part~$N'\subset N$ of the JSJ-decomposition of~$N$.
Then the algorithm in~\cite{mat} includes an explicit construction of a Seifert fibration~$f:N'\rightarrow F$,
where~$F$ is a compact surface. Moreover, all fibers of this fibration are vertical
with respect to the $I$-bundle structure in each `slice' (after a proper deformation of the $I$-bundle structures).
The surface~$S$ arises as the preimage~$f^{-1}(\alpha)$ of a closed curve or a proper arc~$\alpha\subset F$.

Any surface obtained from~$S$ by a homeomorphism representing
an element of~$\mathrm{Sym}_{\mathrm{JSJ}}(N,\mathit\Delta)$ is isotopic to a surface of the form~$f^{-1}(\alpha')$,
where~$\alpha'$ is obtained from~$\alpha$ by a self-homeomorphism~$\psi$ of~$F$
identical on~$\partial F$. Thus, we have a transitive action of
the mapping class group of~$F$ (with boundary fixed pointwise)
on the orbit of~$[S]$ under the action of~$\mathrm{Sym}_{\mathrm{JSJ}}(N,\mathit\Delta)$.

Due to Lickorish~\cite{lick63,lick64} and Chillingworth~\cite{chil},
one can find algorithmically a finite collection~$\psi_1,\ldots,\psi_p$
of self-homeomorphisms of~$F$ identical on~$\partial F$ and
generating the mapping class group of~$F$. Moreover,
each~$\psi_i$ can be chosen to be either a Dehn twist or a Y-homeomorphism (which
is needed in the case of a non-orientable~$F$). Then each~$\psi_i$ can be lifted to a
self-homeomorphism~$\widetilde\psi_i$ of~$N'$ preserving~$N'\cap P$ and identical on~$\partial N'$.
This is due to the fact that, by construction, the homeomorphism~$\psi_i$
is identical outside of a regular neighborhood of a circle or
a bouquet of two circles, over which the fibration~$f$ has a section.

Continue the homeomorphism~$\widetilde\psi_i$ to the whole~$M$
by~$\widetilde\psi_i|_{M\smallsetminus N'}=\mathrm{id}_{M\smallsetminus N'}$
and let~$g_i$ be the element of~$\mathrm{Sym}_{M,\mathit\Gamma}(P)$ represented by~$\widetilde\psi_i$,
$i=1,\ldots,p$. The above discussion implies that the subgroup~$G(P,P')$
of~$\mathrm{Sym}_{M,\mathit\Gamma}(P)$ generated
by~$g_1,\ldots,g_p$ satisfies~\eqref{gpp0-eq}, and hence~\eqref{gpp-eq}.
This concludes the proof of~\eqref{newE-eq} in the case when the extension~$P\mapsto P'$ is of one
of the types~$E_1$, $E_2$, $E_7$, or~$E_8$.

Suppose that the extension~$P\mapsto P'$ has the type~$E_5$. In this case, $N$ is simple,
and~$P'$ is obtained
from~$P$ by the addition of a $p$-minimal surface~$F\subset N$ (see \cite[Definition 6.3.8]{mat}) with
non-empty boundary. We define~$G(P,P')$ as the subgroup of~$\mathrm{Sym}_{\mathrm{JSJ}}(N;P)=
\mathrm{Sym}(N;P)\subset\mathrm{Sym}_{M,\mathit\Gamma}(P)$
generated by Dehn twist along boundary tori of~$N$. Then, by the nature of the procedure
in the proof of \cite[Proposition 6.3.21]{mat}, each $G(P,P')$-orbit 
contained in~$\mathscr O_P(P')$ is represented in~$\mathscr E_0(P)$.
Technically, it suffices to change the inequality~$\tau_i\leqslant k+2$ to~$\tau_i\leqslant2k+4$
on page 239 of~\cite{mat} to achieve~\eqref{newE-eq} in this case.

If the extension~$P\mapsto P'$ has type~$E_9$, we put~$G(P,P')=\mathrm{Sym}_{\mathrm{JSJ}}(N;P)=
\mathrm{Sym}(N;P)$ and note that the algorithm in~\cite{mat} already includes a procedure
for producing a generating set~$\{g_1,\ldots,g_p\}$ for this group. We only need
to compute a representative~$g_i\cdot[P']_P$ for each~$i=1,\ldots,p$ and include it
in~$\mathscr E(P)$.

Finally, suppose that the extension~$P\mapsto P'$ is of type~$E_{10}$. In this case, the manifold~$N$
carries a Seifert fibration structure~$f:N\rightarrow F$ over a compact surface~$F$ with non-empty boundary,
and the intersection~$P\cap N$ has the form~$f^{-1}(\gamma)$, where~$\gamma$
is a three-valent graph embedded in~$F$ such that~$F\smallsetminus\gamma$ is a union of open
two-discs. Thus, the closure of each connected component of~$N\smallsetminus P$ is
a solid torus. The extension~$P\mapsto P'$ consists in the addition of a pair of meridional discs of each
such solid torus to~$P$. In our case (due to the choice of~$\mathit\Gamma$), we will also have~$f(\mathit\Delta)=
\partial F$.

Suppose an orientation has been fixed on~$N$. Then with every oriented simple closed curve~$\alpha\subset F\smallsetminus\partial F$
one can associate a well defined Dehn twist~$N\rightarrow N$ along~$f^{-1}(\alpha)$ that keeps invariant
each fiber~$f^{-1}(x)$, $x\in F$. One can see that this construction yields an isomorphism~$H_1(F;\mathbb Z)\rightarrow
\mathrm{Sym}(N;P)=\mathrm{Sym}_{\mathrm{JSJ}}(N;P)$. So, in order to achieve~\eqref{gpp-eq}
in this case, we may put~$G(P,P')=\mathrm{Sym}(N;P)$
and produce a set of generators of~$G(P,P')$
from any set of generators of the first homology group~$H_1(F;\mathbb Z)$.

This concludes the proof of Theorem~\ref{fertile-alg-exist-th}.
\end{proof}

\begin{proof}[Proof of Theorem~\ref{find-gen-th}]
The algorithm for producing the required family of transformations starts from
triangulating the sphere~$\mathbb S^3$ so that~$\widehat R$ is a union of edges.
For instance, the triangulation in which every three-simplex has the form~$I*J$,
where~$I$ and~$J$ are closed intervals of~$\mathbb S^1_{\tau=1}$
and~$\mathbb S^1_{\tau=0}$, respectively, between two neighboring vertices of~$\widehat R$
is suitable for that. After a sufficiently fine subdivision we will have an open tubular neighborhood~$U$
of~$\widehat R$ such that the closure~$\overline U$ is a union of three-simplices.
Denote~$\mathbb S^3\setminus U$ by~$M$. Clearly, every element of~$\mathrm{Sym}(\widehat R)$
can be represented by a PL-self-homeomorphism of~$\mathbb S^3$ preserving~$M$,
which gives rise to an isomorphism~$\mathrm{Sym}(\widehat R)\cong\mathrm{Sym}(M)$.

Now we use Theorem~\ref{fertile-alg-exist-th} to find a fertile family of special polyhedra~$\mathscr P=\{P_0,P_1,\ldots,P_N\}$
in~$M$ such that~$M\setminus P_i$ is a union of open three-balls for all~$i$.

The fact that the complement~$M\setminus P_i$ is a union of open three-balls has the following consequences:
\begin{enumerate}
\item
the space of PL-homeomorphisms from~$P_0$ to~$P_i$ has finitely many connected components,
representatives of which can be found algorithmically;
\item
any PL-homeomorphism~$\phi:P_0\rightarrow P_i$ extends to a self-homeomorphism of~$M$
in a unique way up to isotopy.
\end{enumerate}

By construction, each polyhedron~$P_i$ contains the boundary of~$M$, and any homeomorphism~$\phi:P_i\rightarrow P_j$
preserves it. We call such a homeomorphism \emph{good} if its restriction to~$\partial M$
is isotopic to identity. Clearly, this is an algorithmically checkable condition.

Thus, we can find algorithmically a finite collection~$\mathscr A=\{(P_{i_1},\phi_1),(P_{i_2},\phi_2),\ldots,(P_{i_q},\phi_q)\}$
of pairs in which~$\phi_j$ is a good PL-homeomorphism from~$P_0$ to~$P_{i_j}$
such that,
for any $i=1,\ldots,N$ and any good PL-homeomorphism~$\phi:P_0\rightarrow P_i$
there is a pair~$(P_{i_j},\phi_j)$ in~$\mathscr A$ with~$i_j=i$
and~$\phi_j$ isotopic to~$\phi$. Since the collection~$\mathscr P$ is fertile, the elements of~$\mathrm{Sym}(\widehat R)$
represented by extensions of~$\phi_j$ to the entire~$\mathbb S^3$ generate the whole symmetry group.

Now observe that the set of all sequences~$s$ of elementary moves starting from and arriving at~$R$
is enumerable. We create an empty list~$\mathscr B$ and
start a process that generates an infinite list of all such sequences. We denote this process by~$p_*$.

For each sequence~$s$ generated by~$p_*$, we can compute, in combinatorial terms,
a PL-self-homeomorphism of~$(\mathbb S^3,\widehat R)$
representing~$\widehat s$. Moreover, the set of all such homeomorphisms is also enumerable, so we start
a process~$p_s$ that generates an infinite list of all such homeomorphisms. We run~$p_s$
in parallel with~$p_*$ and all other previously started processes.

For each homeomorphism~$\phi$ produced by~$p_s$ we check wether~$\phi|_{P_0}$ coincides with one
of~$\phi_j$, where~$(P_{i_j},\phi_j)\in\mathscr A$. When such a coincidence occurs we remove the pair~$(P_{i_j},\phi_j)$
from~$\mathscr A$ and append~$s$ to~$\mathscr B$.

Once the list~$\mathscr A$ becomes empty, all processes are terminated and~$\mathscr B$ is given as the output.
One can see that this occurs after finitely many steps of the algorithm, and the set~$\{\widehat s:s\in\mathscr B\}$
generates~$\mathrm{Sym}(\widehat R)$ as required.
\end{proof}

\section{Transverse links}\label{transverse-sec}

\emph{A $\xi$-transverse link} is a link in~$\mathbb S^3$ which is transverse to~$\xi$ at every point.
When~$\xi$ is endowed with a coorientation (which we always assume) one also defines
\emph{positively transverse} and \emph{negatively transverse} links by requiring that the
contact $1$-form evaluates positively or, respectively, negatively on every tangent vector to the link provided
that the direction of the vector agrees with the orientation of the link. In the case of~$\xi_\pm$
we use the $1$-forms~$\alpha_\pm$ given by~\eqref{alpha+-eq} and~\eqref{alpha--eq} to set
the coorientation.

Two $\xi$-transverse links~$K$ and~$K'$ are said to be \emph{equivalent} (or \emph{transversely isotopic}) if
there is an isotopy from~$K$ to~$K'$ through $\xi$-transverse links
(for~$\xi=\xi_+$ or~$\xi_-$ this is equivalent to saying that there is a
diffeomorphism~$\varphi:\mathbb S^3\rightarrow\mathbb S^3$ preserving~$\xi$ and its coorientation such that~$\varphi(K)=K'$).

Suppose that~$\xi$ is a cooriented contact structure on~$\mathbb S^3$. Then with
every $\xi$-Legendrain link type~$\mathscr L$ one associates a positively and negatively $\xi$-transverse
link types~$\mathscr T_+(\mathscr L)$  and~$\mathscr T_-(\mathscr L)$, respectively,
by requiring that a representative of~$\mathscr
T_\pm(\mathscr L)$ can be obtained by arbitrarily $C^1$-small deformation of any representative of~$\mathscr L$.
This construction is known to give a one-to-one correspondence between
positive (respectively, negative) $\xi$-transverse link types and $\xi$-Legendrian types
viewed up to negative (respectively, positive) stabilizations; see \cite{EFM,ngth}.

Thus, there are the following four transverse links associated with every rectangular diagram~$R$:
$$\mathscr T_{++}(R)=\mathscr T_+(\mathscr L_+(R)),\quad
\mathscr T_{+-}(R)=\mathscr T_-(\mathscr L_+(R)),\quad
\mathscr T_{-+}(R)=\mathscr T_+(\mathscr L_-(R)),\quad
\mathscr T_{--}(R)=\mathscr T_-(\mathscr L_-(R)),$$
and from Theorem~\ref{leg-rect-th} we have the following.

\begin{theo}[\cite{MM,ngth,OST}]\label{tran-rect-th}
\emph{(i)}
Any positive $\xi_+$-transverse link type has the form~$\mathscr T_{++}(R)$ for
some rectangular diagram~$R$.

\emph{(ii)}
For any two rectangular diagrams of links~$R_1$ and~$R_2$ we have~$\mathscr T_{++}(R_1)=
\mathscr T_{++}(R_2)$ if and only if the diagrams are related by a sequence of elementary moves
not including stabilizations and destabilizations of oriented type~$\overleftarrow{\mathrm{II}}$.

Similar statement holds for~$\mathscr T_{+-}$, $\mathscr T_{-+}$, and~$\mathscr T_{--}$
and oriented \emph(de\emph)stabilization types $\overrightarrow{\mathrm{II}}$,
$\overleftarrow{\mathrm I}$, and~$\overrightarrow{\mathrm I}$, respectively.
\end{theo}

There is a full analogue of Theorem~\ref{main-th} for transverse links, which is as follows.

\begin{theo}\label{main1-th}
There exists an algorithm that, given two rectangular diagrams of links~$R_1$ and~$R_2$,
decides whether or not~$\mathscr T_{++}(R_1)=\mathscr T_{++}(R_2)$.
\end{theo}

\begin{proof}
We follow the lines of the proof of Theorem~\ref{main-th} given in Section~\ref{algorithm-sec},
until Step~5, where a modification occurs. Namely, at Step~5 we check
whether or not the diagrams~$R_2'$ and~$R_3'$ are related by a sequence of elementary moves
that includes only exchange moves and (de)stabilizations of oriented type~$\overrightarrow{\mathrm{II}}$.
This is decidable due to~\cite[Theorem~4.1]{trleg}. The answer is positive if and only if~$\mathscr T_{++}(R_1)=\mathscr T_{++}(R_2)$,
which can be seen, by analogy with Lemma~\ref{final-lem}, as follows.

Suppose that~$\mathscr T_{++}(R_3')=\mathscr T_{++}(R_2')$ holds, which is equivalent to~$\mathscr T_{++}(R_1)=\mathscr T_{++}(R_2)$.
According to Theorem~\ref{tran-rect-th} this means that there exists a sequence~$s$ of elementary moves
not including (de)stabilizations of oriented type~$\overleftarrow{\mathrm{II}}$. Therefore,
there exist rectangular diagrams~$R_2''$ and~$R_3''$ obtained from~$R_2'$ and~$R_3'$, respectively,
by some number of type~$\overrightarrow{\mathrm{II}}$ stabilizations such that~$\mathscr L_+(R_2'')=\mathscr L_+(R_3'')$.
Now the argument from the proof of Lemma~\ref{final-lem} shows that~$R_2''$ and~$R_3''$
are exchange-equivalent.
\end{proof}

\section{Examples}\label{example-sec}

Here we use Corollary~\ref{count-coro} to confirm several previously unsettled conjectures
about non-equivalence of Legendrian knots. The conjectures are
formulated in the Legendrian knot atlas by W.\,Chongchitmate and L.\,Ng~\cite{chong2013}. 
Namely, the $\xi_+$-Legendrian (respectively, $\xi_-$-Legendrian) classes in the first row
(respectively, leftmost column) of each of the tables in Figures~\ref{7_5-fig}--\ref{11n_118-fig}
are conjectured to be pairwise distinct in~\cite{chong2013} (except one case, which is overlooked in~\cite{chong2013}).

The collections of diagrams in Figures~\ref{7_5-fig}--\ref{11n_118-fig}
have the following three properties:
\begin{enumerate}
\item
all diagrams in each column (respectively, row) of the table represent the same
$\xi_+$-Legendrian type (respectively, the same~$\xi_-$-Legendrian type);
\item
all diagrams except those in the leftmost column and the top row represent pairwise distinct
exchange classes;
\item
any exchange class~$c$ such that~$\mathscr L_+(c)$ appears in the top row and~$\mathscr L_-(c)$
appears in the leftmost column is represented by a diagram outside the top row and leftmost column.
\end{enumerate}
All this is verified by an exhaustive search. To confirm the declared equivalences of Legendrian knots
one stabilization suffices in each case.

The knot type and its symmetry group
are specified in the top left corner of each table. All of them can be verified by
the Knotscape program~\cite{knotscape}. In most cases, the symmetry groups
have also been computed previously in the literature. Namely, the cases of
the knots~$7_5$, $7_7$, and~$8_{21}$ are covered by Theorem~6.2 in~\cite{sak90},
the knots~$9_{47}$ and~$9_{49}$ match Example~1.11 in~\cite{ks92},
the knots~$11n_{19}$ and~$11n_{38}$ are Montesinos knots, which are treated in Theorem~1.3
in~\cite{BoZim}.

The tables in Figures~\ref{7_5-fig}--\ref{11n_118-fig} confirm the conjectures of~\cite{chong2013}
by reducing the negation of the latter to a contradiction with Corollary~\ref{count-coro}.
Indeed, in all these cases except for the case of the knot $9_{49}$, merging any two rows
or columns would create a cell with more than~$|G|$ diagrams in it, where~$G$ is the symmetry group.
In the case of the knot~$9_{49}$ there are just two $\xi_+$-Legendrian types in question,
and merging their respective columns would create a cell with exactly two diagrams in
it, which is impossible, since the symmetry group is~$\mathbb Z_3$.

The conjectures of~\cite{chong2013} about Legendrian knots
having knot types~$9_{42}$, $9_{43}$, $9_{44}$, $9_{45}$,
$10_{128}$, and~$10_{160}$ are settled in~\cite{dysha2023} by using the same method.
Those related to knot types~$6_2$ and~$7_6$ are confirmed in~\cite{dynn-pras-7-6,dp2021}
in a different way, but can also be confirmed by counting exchange classes.

In the remaining three unresolved
cases in~\cite{chong2013}, which involve knot types~$7_4$, $9_{48}$, and~$10_{136}$,
the computation of the respective symmetry groups~$\mathrm{Sym}_\pm$ appears to be essential
for applying Corollary~\ref{count-coro}. This is done in~\cite{prsh}, where all
conjectures of~\cite{chong2013} about the corresponding Legendrian types
are also confirmed.

Finally, we note that the top part of the mountain range for the knot~$11n_{19}$ shown in~\cite{chong2013}
contains an error, which is corrected in Figure~\ref{11n_19-fig}.

\begin{figure}[ht]
\begin{tabular}{c|c|cc|cc}
$7_5$, $\mathbb Z_2$
&$\mathscr L_+\left(\raisebox{-14pt}{\includegraphics[scale=.16]{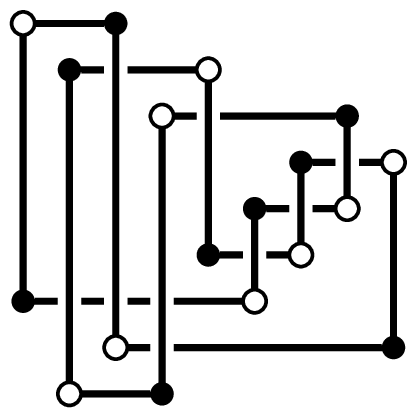}}\right)$
&\multicolumn2{c|}{$\mathscr L_+\left(\raisebox{-14pt}{\includegraphics[scale=.16]{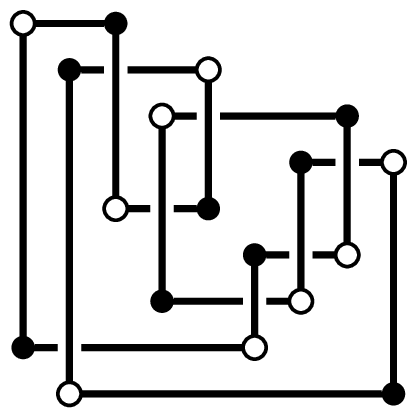}}\right)$}
&\multicolumn2c{$\mathscr L_+\left(\raisebox{-14pt}{\includegraphics[scale=.16]{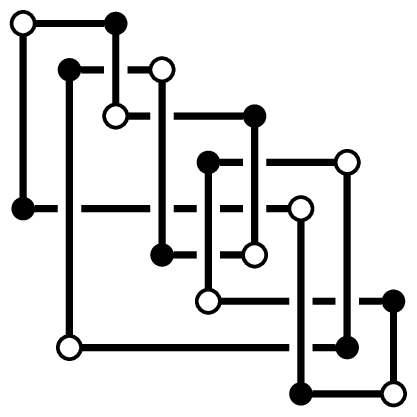}}\right)$}\\\hline
$\mathscr L_-\left(\raisebox{-14pt}{\includegraphics[scale=.16]{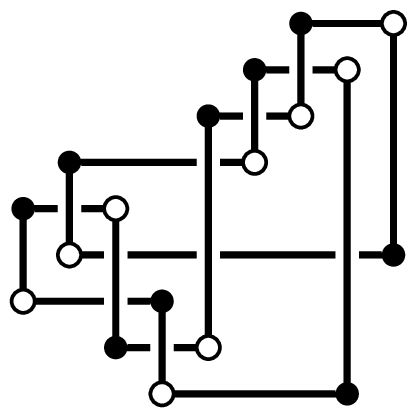}}\right)$
&\raisebox{-14pt}{\includegraphics[scale=.16]{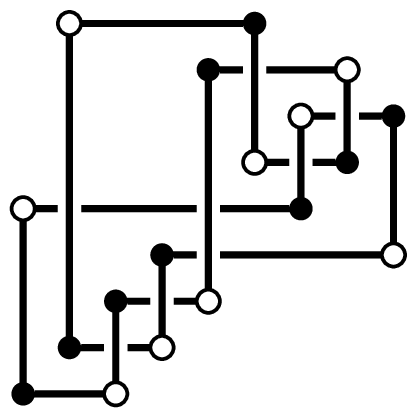}}
&\raisebox{-14pt}{\includegraphics[scale=.16]{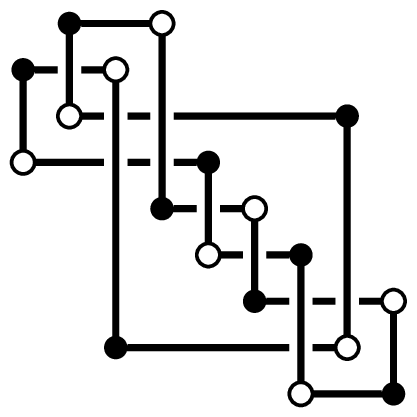}}
&\raisebox{-14pt}{\includegraphics[scale=.16]{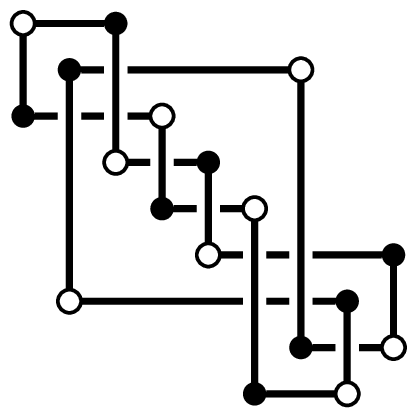}}
&\raisebox{-14pt}{\includegraphics[scale=.16]{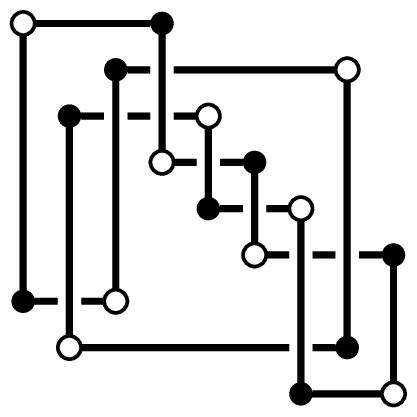}}
&\raisebox{-14pt}{\includegraphics[scale=.16]{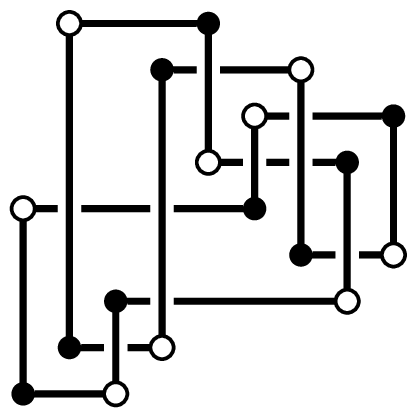}}
\end{tabular}
\caption{Diagrams of the knot $7_5$}\label{7_5-fig}
\end{figure}

\begin{figure}[ht]
\begin{tabular}{c|cc|cc|cc|cc}
$7_7$, $\mathbb Z_4$
&\multicolumn2{c|}{$\mathscr L_+\left(\raisebox{-14pt}{\includegraphics[scale=.16]{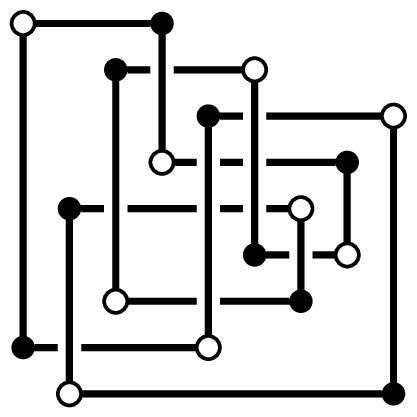}}\right)$}
&\multicolumn2{c|}{$\mathscr L_+\left(\raisebox{-14pt}{\includegraphics[scale=.16]{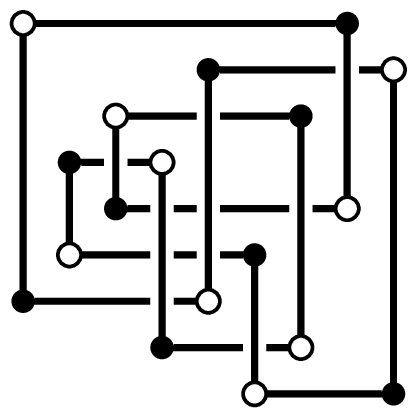}}\right)$}
&\multicolumn2{c|}{$\mathscr L_+\left(\raisebox{-14pt}{\includegraphics[scale=.16]{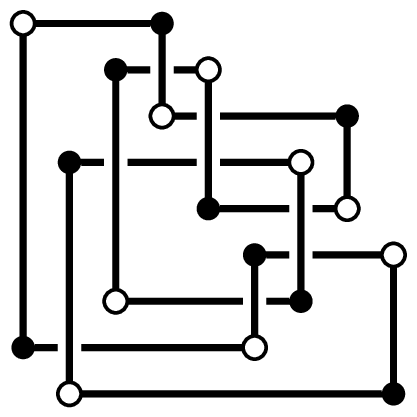}}\right)$}
&\multicolumn2c{$\mathscr L_+\left(\raisebox{-14pt}{\includegraphics[scale=.16]{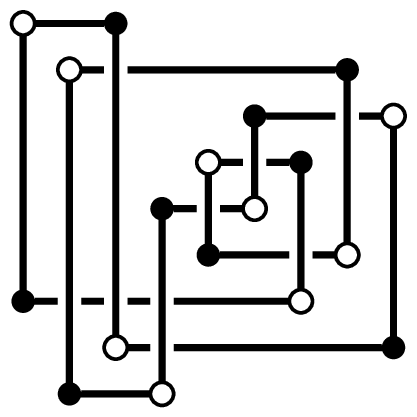}}\right)$}\\\hline
\multirow2*{$\mathscr L_-\left(\raisebox{-14pt}{\includegraphics[scale=.16]{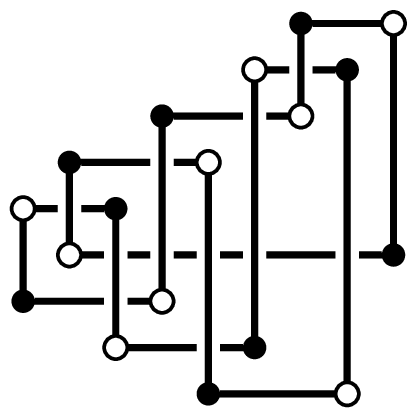}}\right)$}
&\raisebox{-14pt}{\includegraphics[scale=.16]{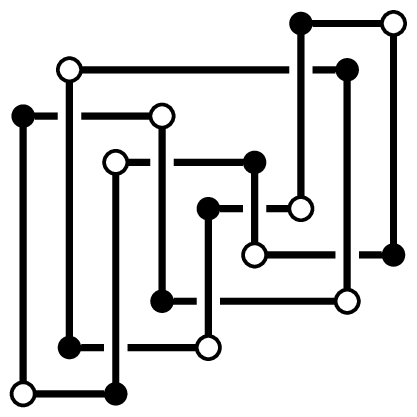}}
&\raisebox{-14pt}{\includegraphics[scale=.16]{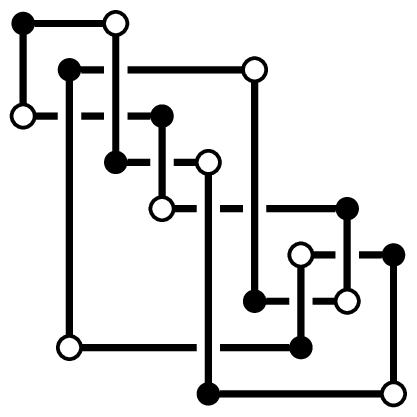}}
&\raisebox{-14pt}{\includegraphics[scale=.16]{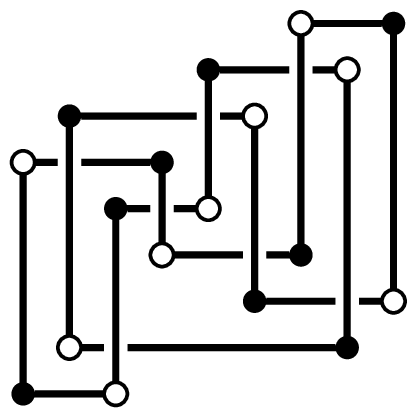}}
&\raisebox{-14pt}{\includegraphics[scale=.16]{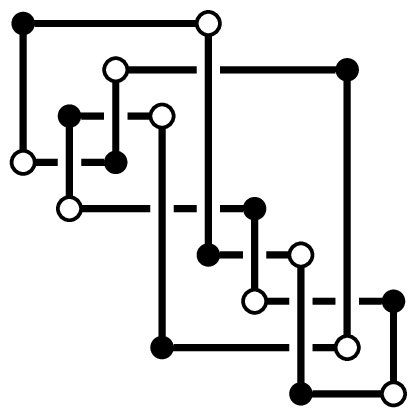}}
&\raisebox{-14pt}{\includegraphics[scale=.16]{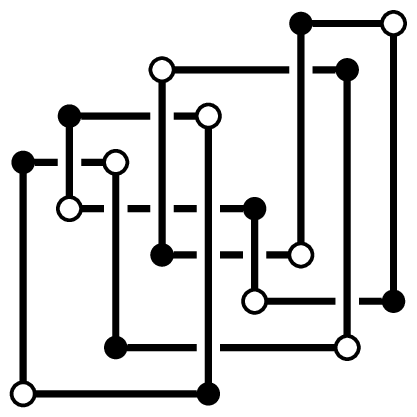}}
&\raisebox{-14pt}{\includegraphics[scale=.16]{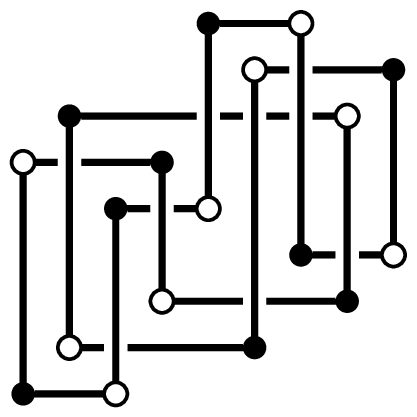}}
&\raisebox{-14pt}{\includegraphics[scale=.16]{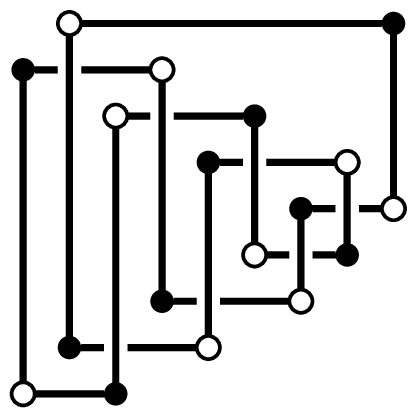}}
&\raisebox{-14pt}{\includegraphics[scale=.16]{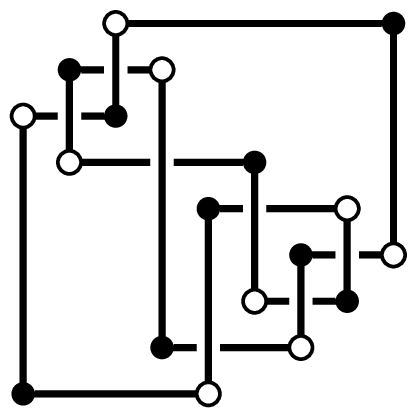}}
\\
&\raisebox{-14pt}{\includegraphics[scale=.16]{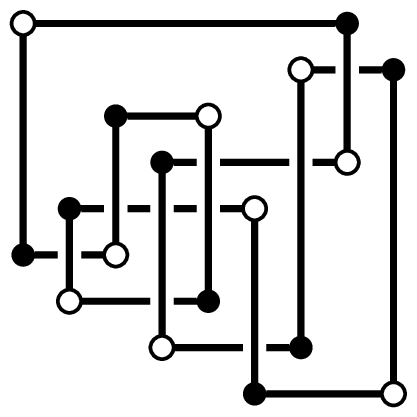}}
&\raisebox{-14pt}{\includegraphics[scale=.16]{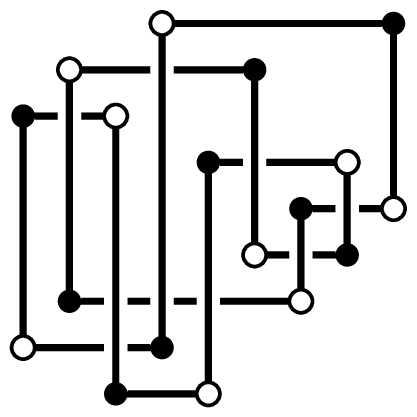}}
&\raisebox{-14pt}{\includegraphics[scale=.16]{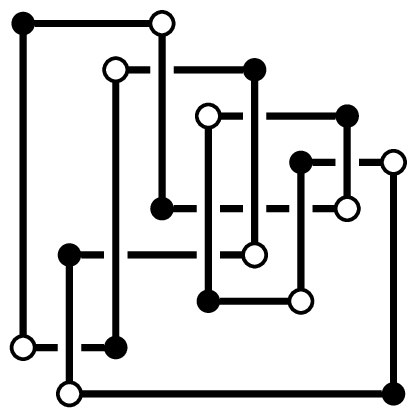}}
&\raisebox{-14pt}{\includegraphics[scale=.16]{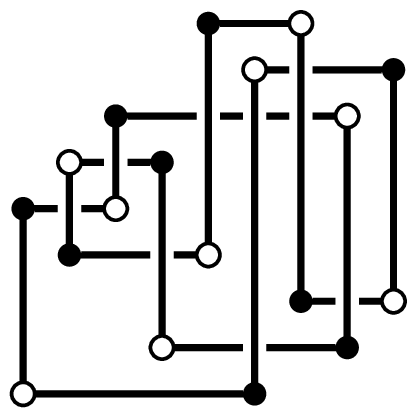}}
&\raisebox{-14pt}{\includegraphics[scale=.16]{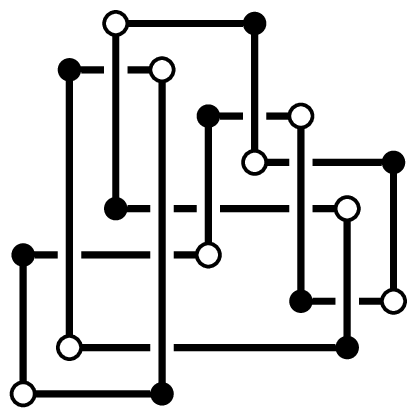}}
&\raisebox{-14pt}{\includegraphics[scale=.16]{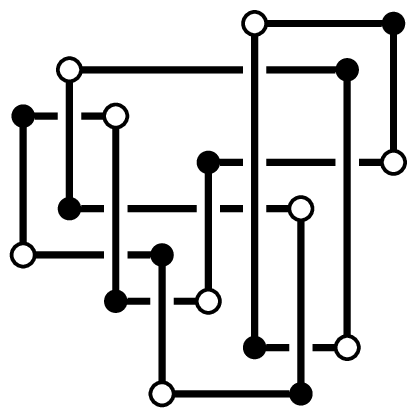}}
&\raisebox{-14pt}{\includegraphics[scale=.16]{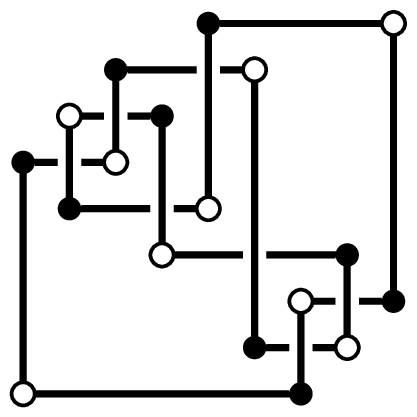}}
&\raisebox{-14pt}{\includegraphics[scale=.16]{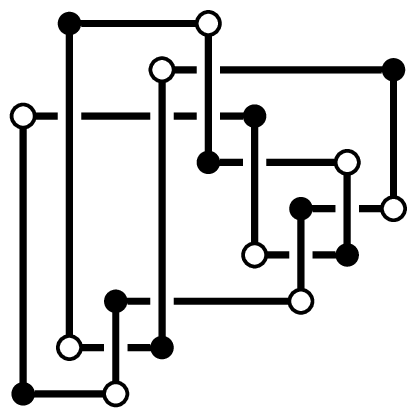}}
\\\hline
\multirow2*{$\mathscr L_-\left(\raisebox{-14pt}{\includegraphics[scale=.16]{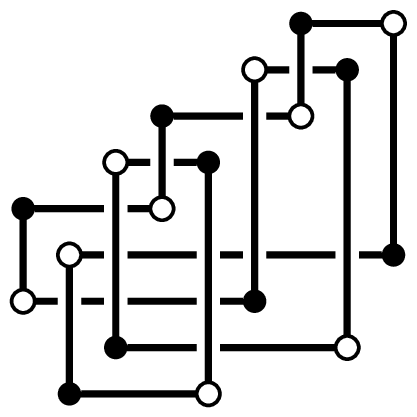}}\right)$}
&\raisebox{-14pt}{\includegraphics[scale=.16]{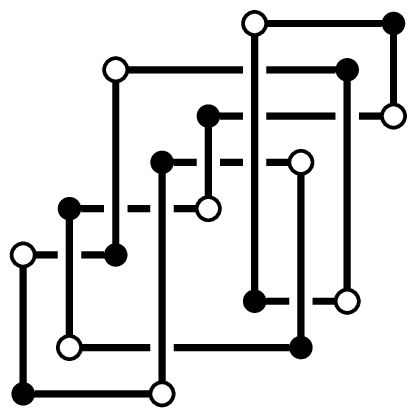}}
&\raisebox{-14pt}{\includegraphics[scale=.16]{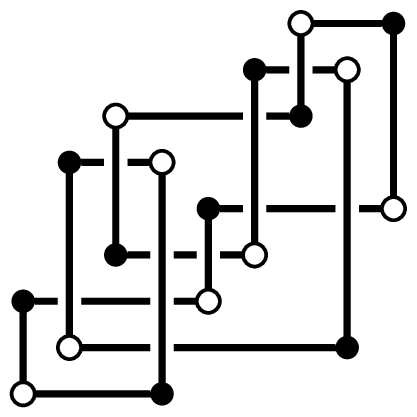}}
&\raisebox{-14pt}{\includegraphics[scale=.16]{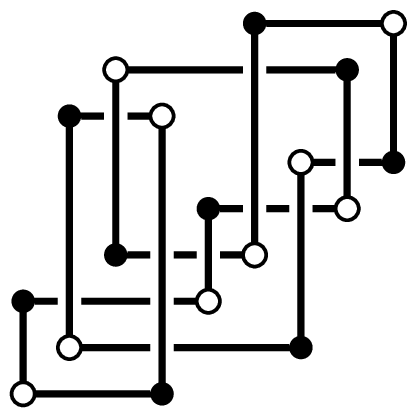}}
&\raisebox{-14pt}{\includegraphics[scale=.16]{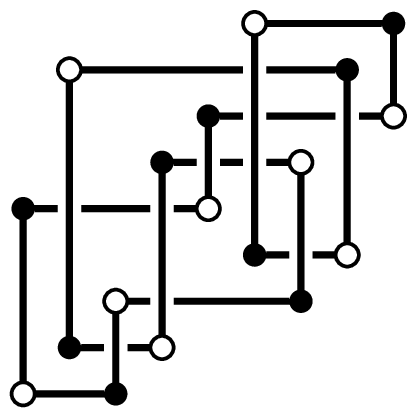}}
&\raisebox{-14pt}{\includegraphics[scale=.16]{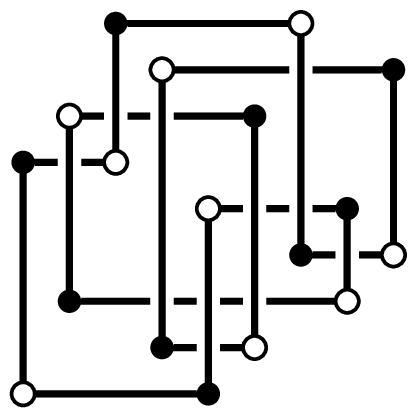}}
&\raisebox{-14pt}{\includegraphics[scale=.16]{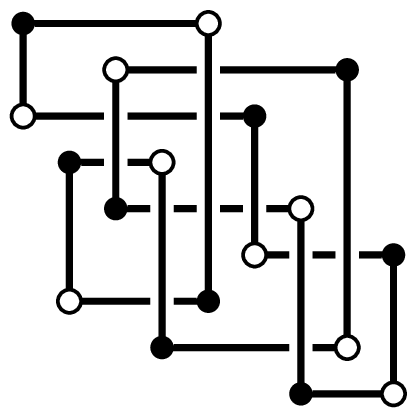}}
&\raisebox{-14pt}{\includegraphics[scale=.16]{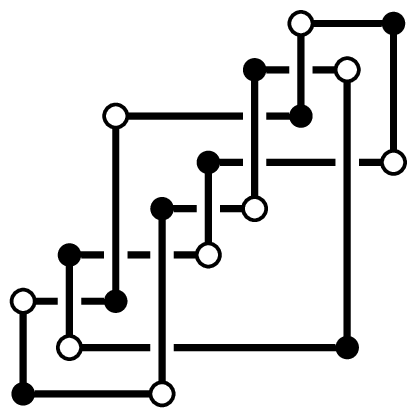}}
&\raisebox{-14pt}{\includegraphics[scale=.16]{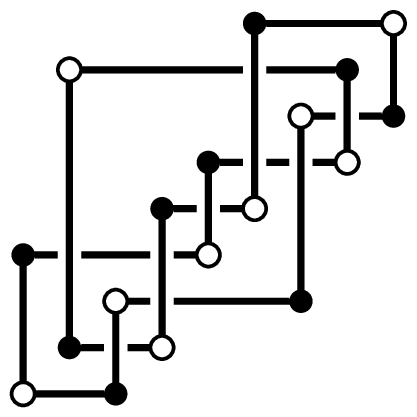}}
\\
&\raisebox{-14pt}{\includegraphics[scale=.16]{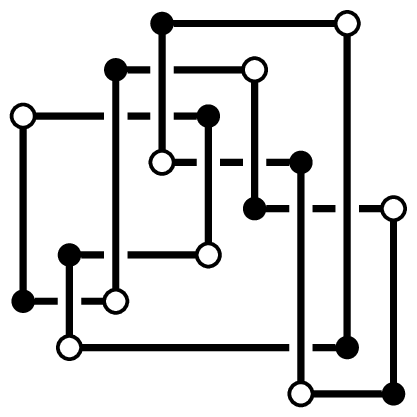}}
&\raisebox{-14pt}{\includegraphics[scale=.16]{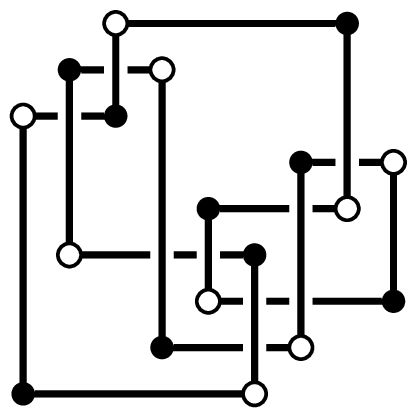}}
&\raisebox{-14pt}{\includegraphics[scale=.16]{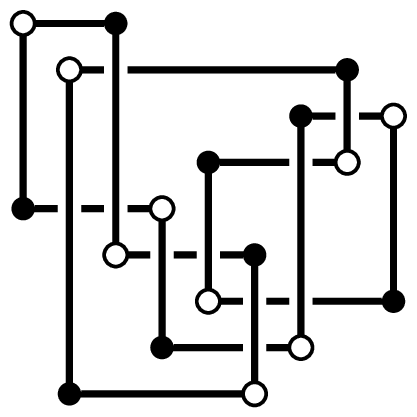}}
&\raisebox{-14pt}{\includegraphics[scale=.16]{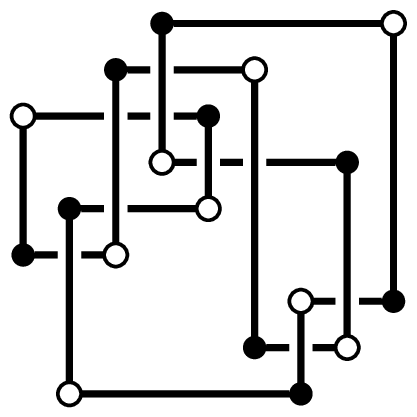}}
&\raisebox{-14pt}{\includegraphics[scale=.16]{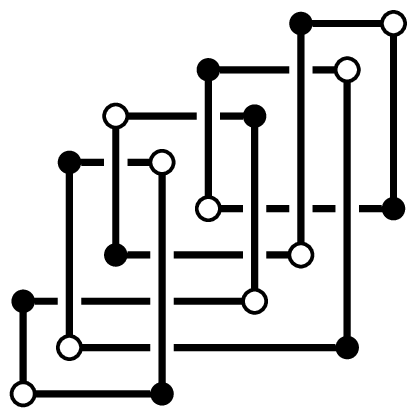}}
&\raisebox{-14pt}{\includegraphics[scale=.16]{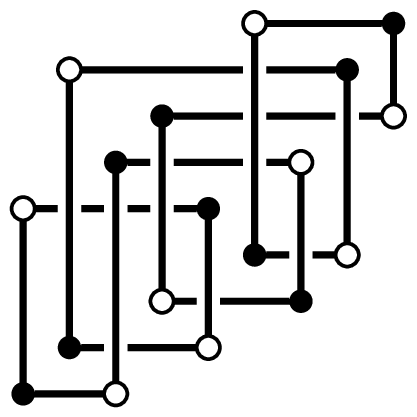}}
&\raisebox{-14pt}{\includegraphics[scale=.16]{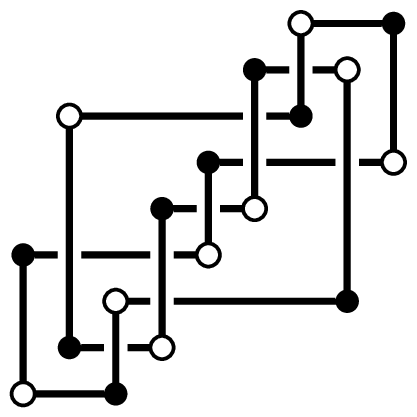}}
&\raisebox{-14pt}{\includegraphics[scale=.16]{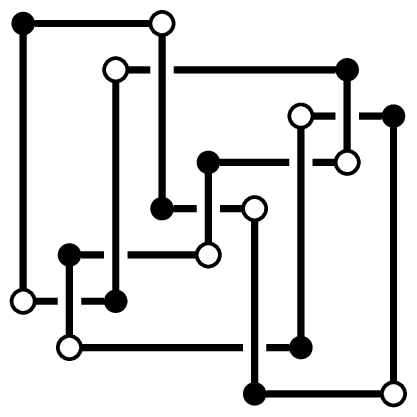}}
\end{tabular}
\caption{Diagrams of the knot $7_7$}\label{7_7-fig}
\end{figure}

\begin{figure}[ht]
\begin{tabular}{c|cc|cc|cc}
$8_{21}$, $\mathbb Z_2$
&\multicolumn2{c|}{$\mathscr L_+\left(\raisebox{-14pt}{\includegraphics[scale=.16]{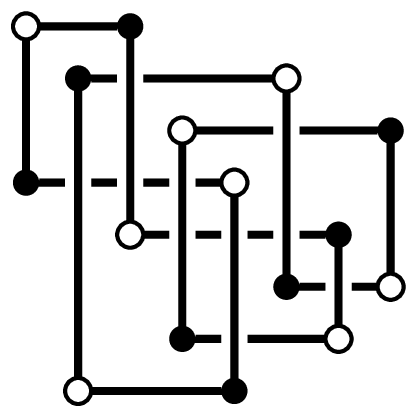}}\right)$}
&\multicolumn2{c|}{$\mathscr L_+\left(\raisebox{-14pt}{\includegraphics[scale=.16]{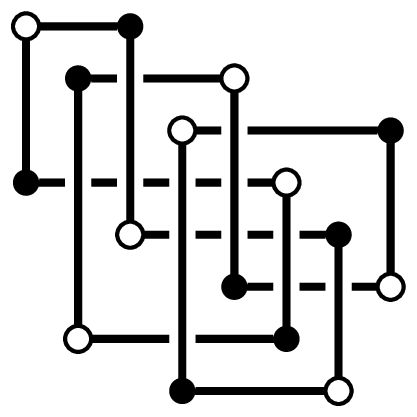}}\right)$}
&\multicolumn2c{$\mathscr L_+\left(\raisebox{-14pt}{\includegraphics[scale=.16]{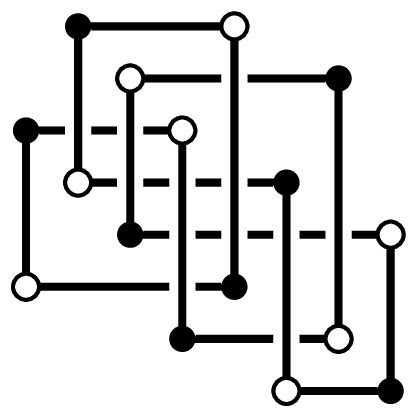}}\right)$}\\\hline
$\mathscr L_-\left(\raisebox{-14pt}{\includegraphics[scale=.16]{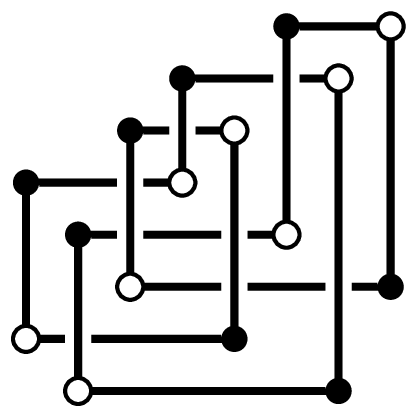}}\right)$
&\raisebox{-14pt}{\includegraphics[scale=.16]{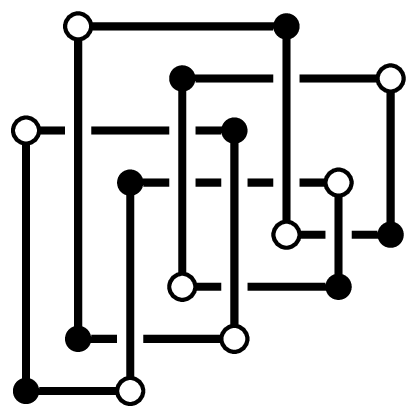}}
&\raisebox{-14pt}{\includegraphics[scale=.16]{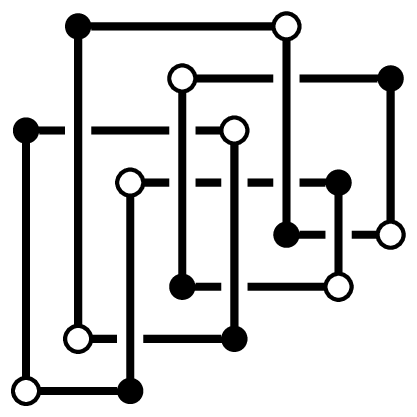}}
&\raisebox{-14pt}{\includegraphics[scale=.16]{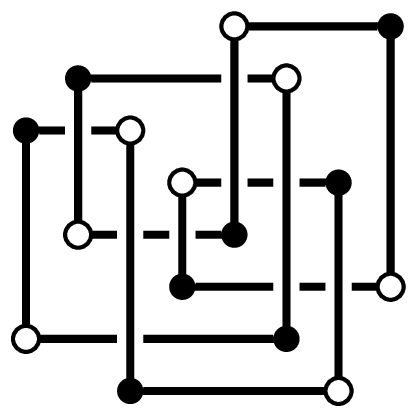}}
&\raisebox{-14pt}{\includegraphics[scale=.16]{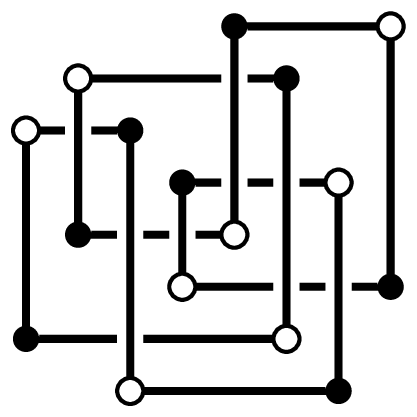}}
&\raisebox{-14pt}{\includegraphics[scale=.16]{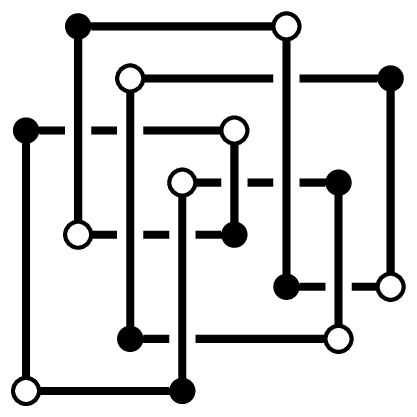}}
&\raisebox{-14pt}{\includegraphics[scale=.16]{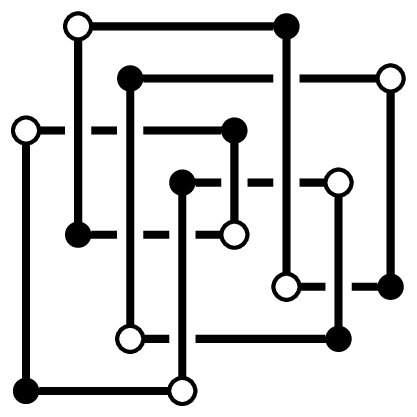}}
\end{tabular}
\caption{Diagrams of the knot $8_{21}$}\label{8_21-fig}
\end{figure}

\begin{figure}[ht]
\begin{tabular}{c|ccc}
$9_{47}$, $\mathbb Z_3$
&\multicolumn3c{$\mathscr L_+\left(\raisebox{-14pt}{\includegraphics[scale=.16]{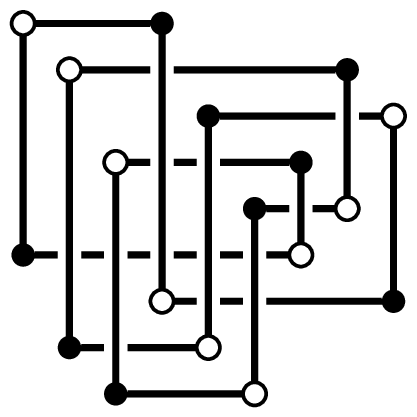}}\right)$}\\\hline
$\mathscr L_-\left(\raisebox{-14pt}{\includegraphics[scale=.16]{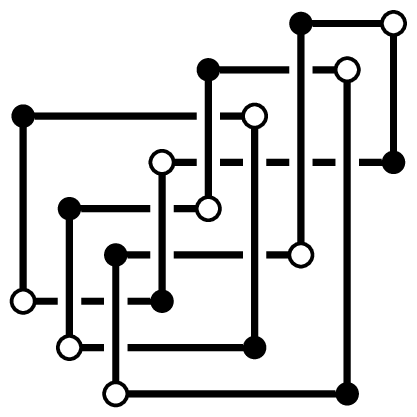}}\right)$
&\raisebox{-14pt}{\includegraphics[scale=.16]{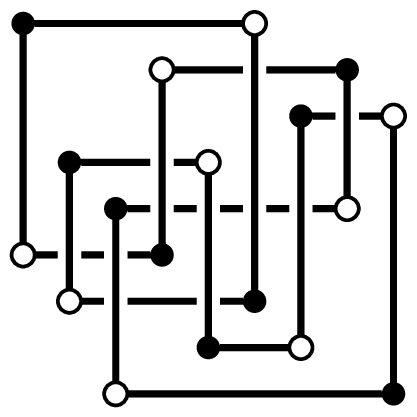}}
&\raisebox{-14pt}{\includegraphics[scale=.16]{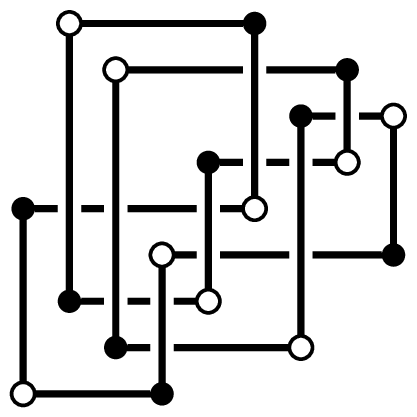}}
&\raisebox{-14pt}{\includegraphics[scale=.16]{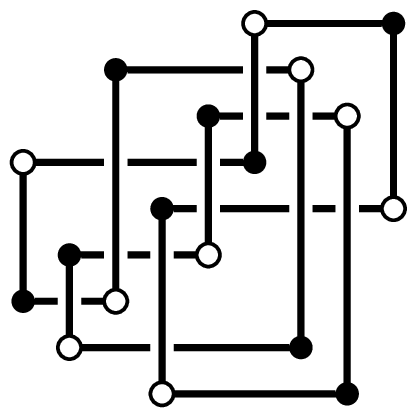}}
\\\hline
$\mathscr L_-\left(\raisebox{-14pt}{\includegraphics[scale=.16]{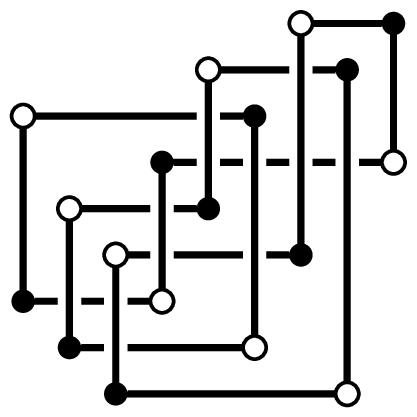}}\right)$
&\raisebox{-14pt}{\includegraphics[scale=.16]{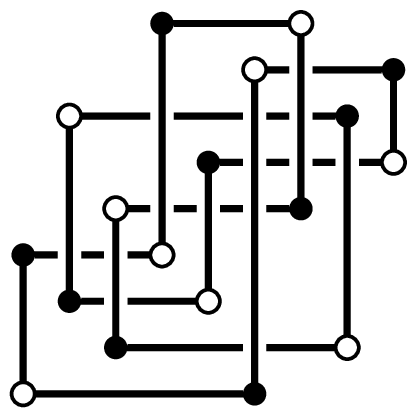}}
&\raisebox{-14pt}{\includegraphics[scale=.16]{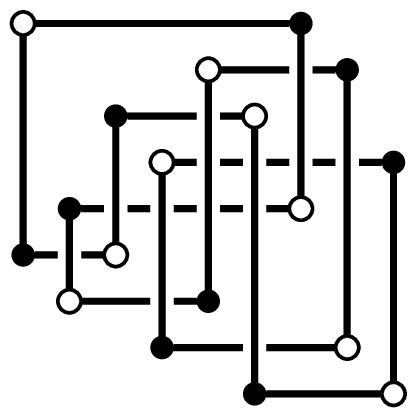}}
&\raisebox{-14pt}{\includegraphics[scale=.16]{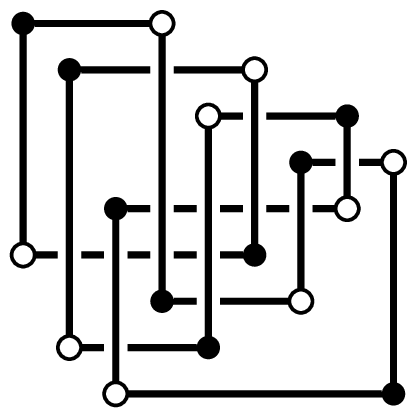}}
\end{tabular}
\caption{Diagrams of the knot $9_{47}$}\label{9_47-fig}
\end{figure}

\begin{figure}[ht]
\begin{tabular}{c|c|c}
$9_{49}$, $\mathbb Z_3$
&$\mathscr L_+\left(\raisebox{-14pt}{\includegraphics[scale=.16]{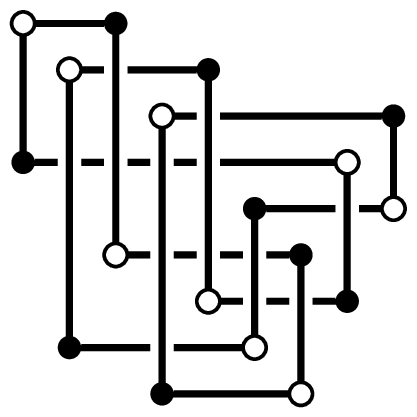}}\right)$
&$\mathscr L_+\left(\raisebox{-14pt}{\includegraphics[scale=.16]{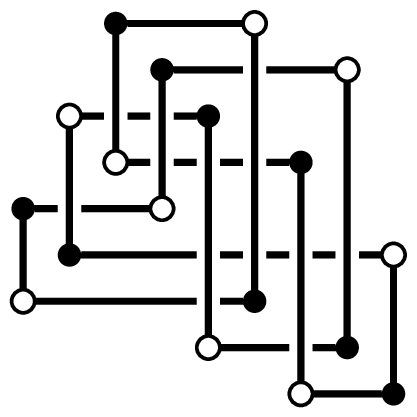}}\right)$\\\hline
$\mathscr L_-\left(\raisebox{-14pt}{\includegraphics[scale=.16]{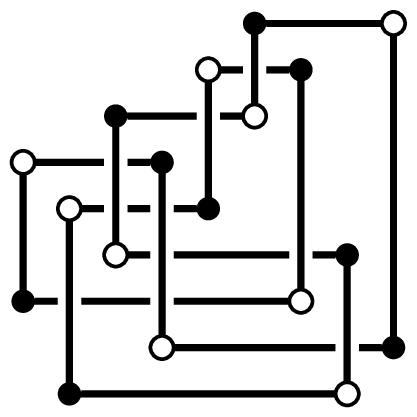}}\right)$
&\raisebox{-14pt}{\includegraphics[scale=.16]{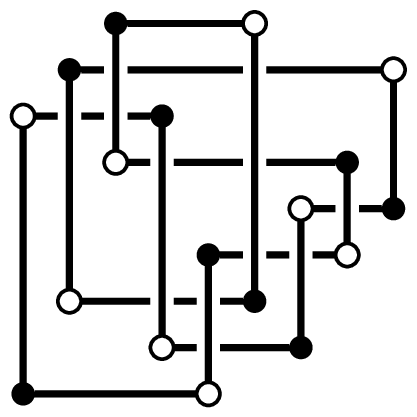}}
&\raisebox{-14pt}{\includegraphics[scale=.16]{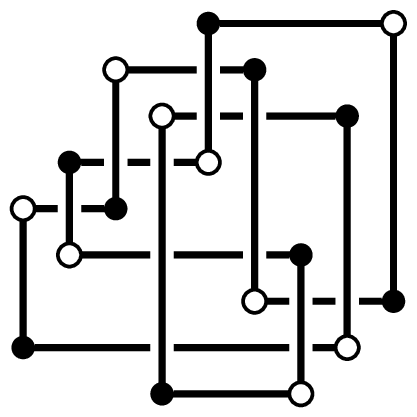}}
\end{tabular}
\caption{Diagrams of the knot $9_{49}$}\label{9_49-fig}
\end{figure}

\begin{figure}[ht]
\begin{tabular}{c|c|c}
$11n_{19}$, $\{1\}$
&$S_-\left(\mathscr L_+\left(\raisebox{-14pt}{\includegraphics[scale=.16]{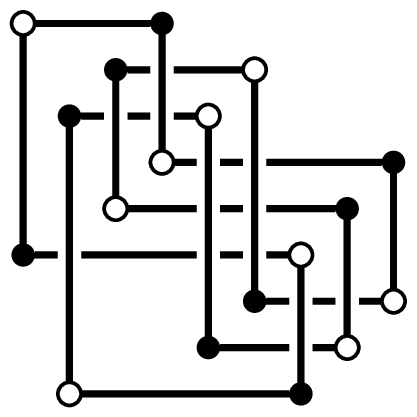}}\right)\right)$
&$S_+\left(\mathscr L_+\left(\raisebox{-14pt}{\includegraphics[scale=.16]{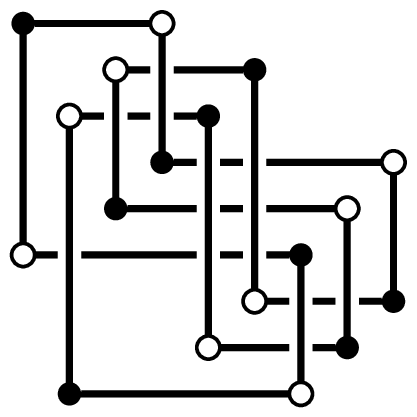}}\right)\right)$\\\hline
$\mathscr L_-\left(\raisebox{-14pt}{\includegraphics[scale=.16]{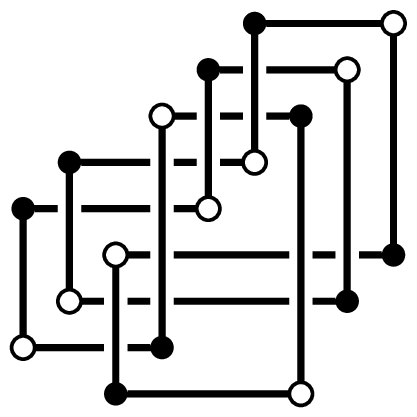}}\right)$
&\raisebox{-14pt}{\includegraphics[scale=.16]{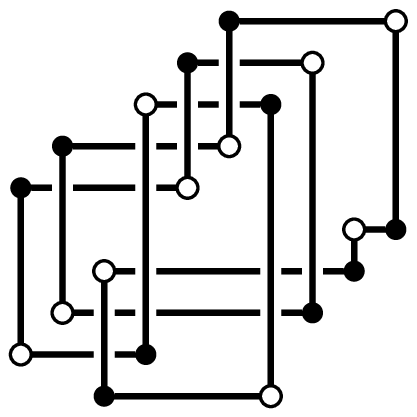}}
&\raisebox{-14pt}{\includegraphics[scale=.16]{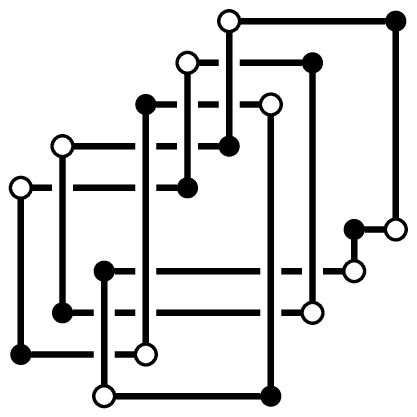}}
\\\hline
$\mathscr L_-\left(\raisebox{-14pt}{\includegraphics[scale=.16]{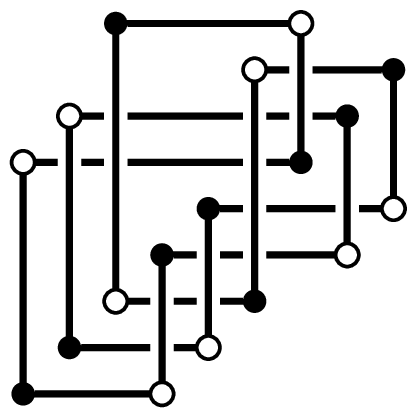}}\right)$
&\raisebox{-14pt}{\includegraphics[scale=.16]{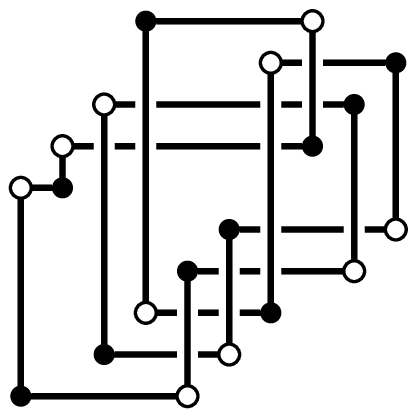}}
&\raisebox{-14pt}{\includegraphics[scale=.16]{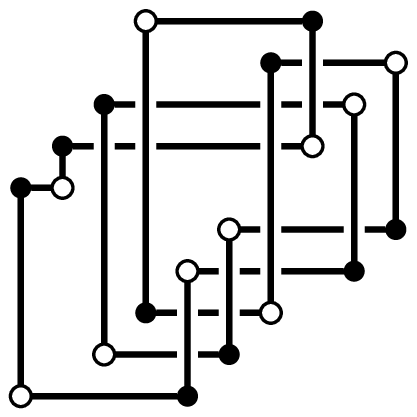}}
\end{tabular}

\includegraphics{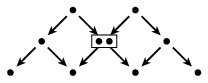}
\caption{Diagrams of the knot $11n_{19}$ and the corrected top part
of the corresponding $\xi_+$-Legendrian mountain range}\label{11n_19-fig}
\end{figure}

\begin{figure}[ht]
\begin{tabular}{c|c|c}
$11n_{38}$, $\{1\}$
&$\mathscr L_+\left(\raisebox{-14pt}{\includegraphics[scale=.16]{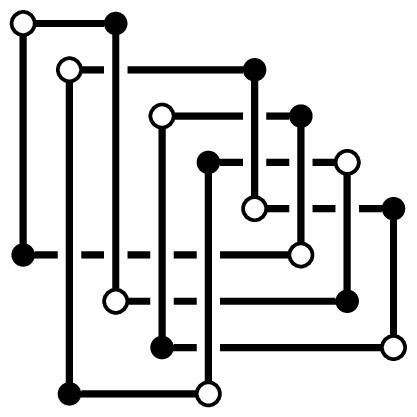}}\right)$
&$\mathscr L_+\left(\raisebox{-14pt}{\includegraphics[scale=.16]{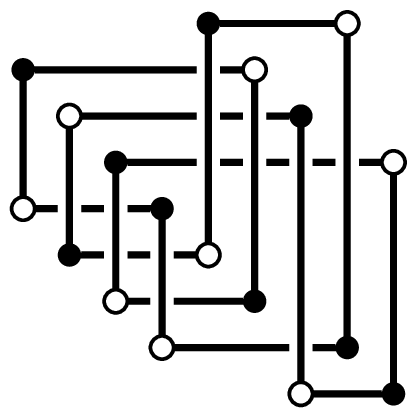}}\right)$\\\hline
$S_-\left(\mathscr L_-\left(\raisebox{-14pt}{\includegraphics[scale=.16]{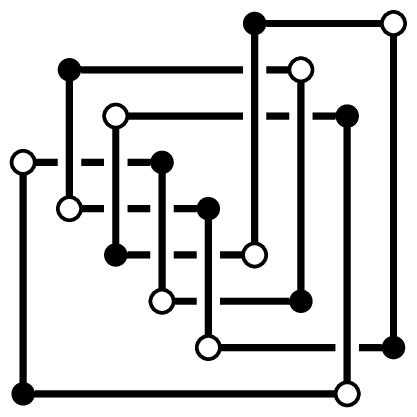}}\right)\right)$
&\raisebox{-14pt}{\includegraphics[scale=.16]{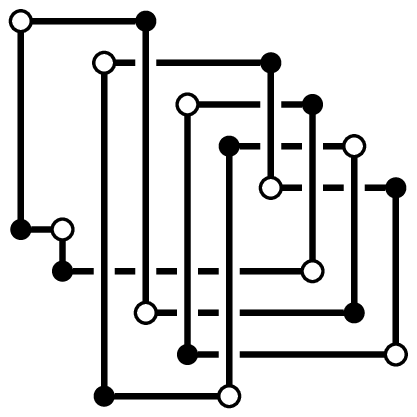}}
&\raisebox{-14pt}{\includegraphics[scale=.16]{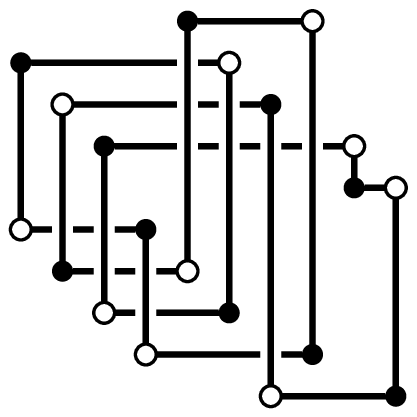}}\\\hline
$S_+\left(\mathscr L_-\left(\raisebox{-14pt}{\includegraphics[scale=.16]{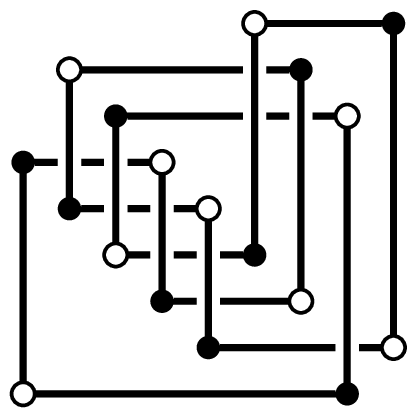}}\right)\right)$
&\raisebox{-14pt}{\includegraphics[scale=.16]{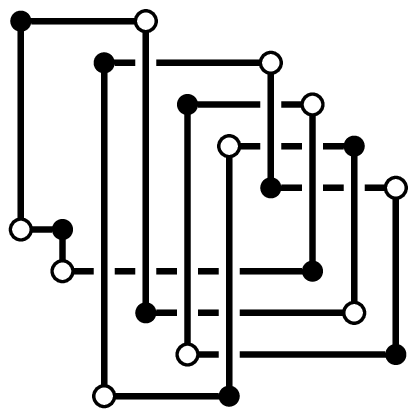}}
&\raisebox{-14pt}{\includegraphics[scale=.16]{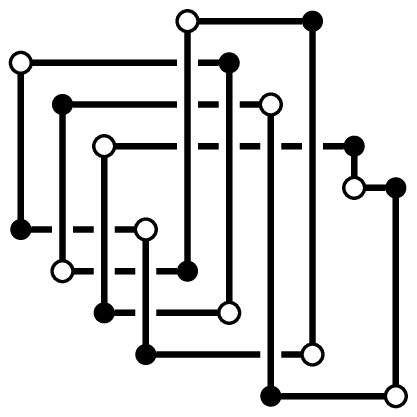}}
\end{tabular}
\caption{Diagrams of the knot $11n_{38}$}\label{11n_38-fig}
\end{figure}

\begin{figure}[ht]
\begin{tabular}{c|c|c}
$11n_{95}$, $\{1\}$
&$\mathscr L_+\left(\raisebox{-14pt}{\includegraphics[scale=.16]{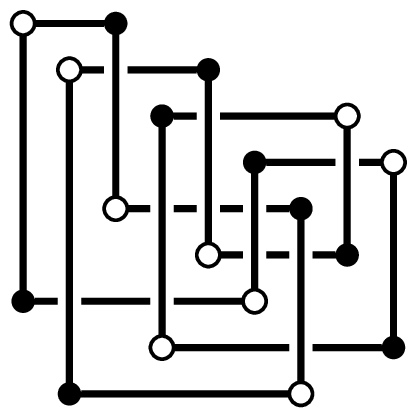}}\right)$
&$\mathscr L_+\left(\raisebox{-14pt}{\includegraphics[scale=.16]{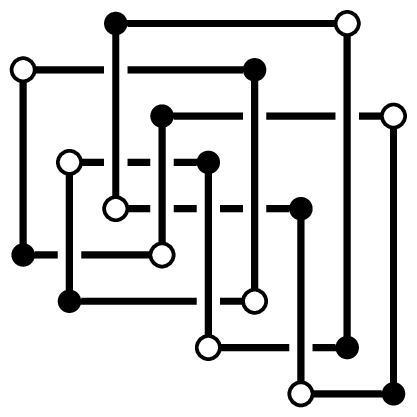}}\right)$\\\hline
$\mathscr L_-\left(\raisebox{-14pt}{\includegraphics[scale=.16]{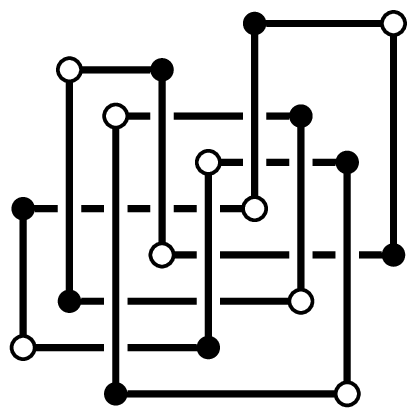}}\right)$
&\raisebox{-14pt}{\includegraphics[scale=.16]{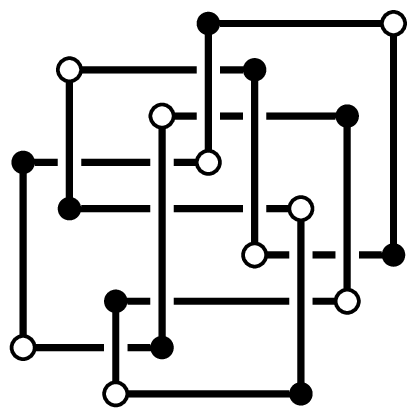}}
&\raisebox{-14pt}{\includegraphics[scale=.16]{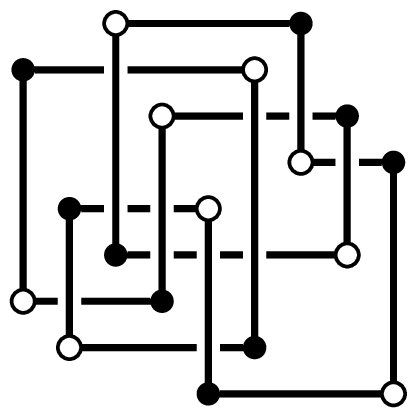}}
\\\hline
$\mathscr L_-\left(\raisebox{-14pt}{\includegraphics[scale=.16]{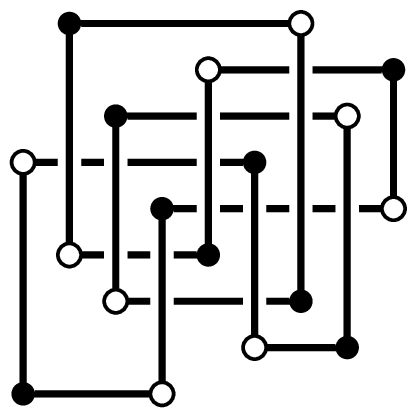}}\right)$
&\raisebox{-14pt}{\includegraphics[scale=.16]{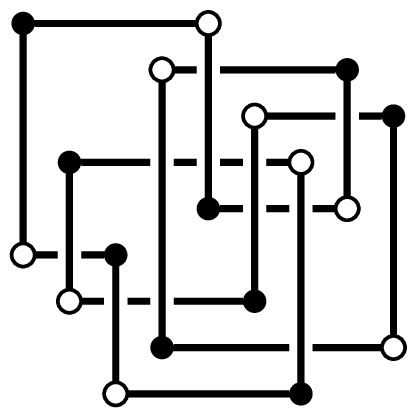}}
&\raisebox{-14pt}{\includegraphics[scale=.16]{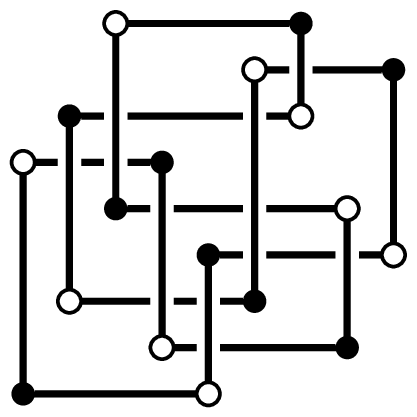}}
\\\hline
$\mathscr L_-\left(\raisebox{-14pt}{\includegraphics[scale=.16]{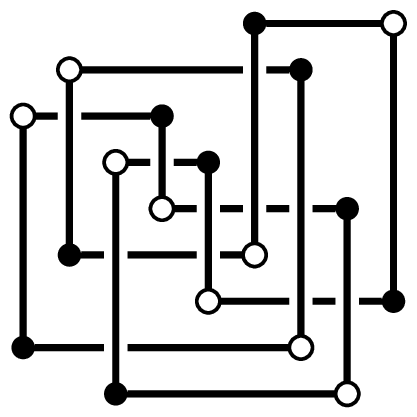}}\right)$
&\raisebox{-14pt}{\includegraphics[scale=.16]{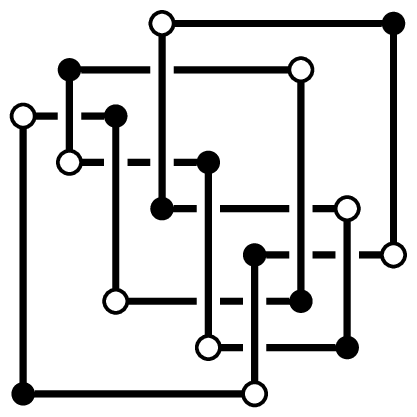}}
&\raisebox{-14pt}{\includegraphics[scale=.16]{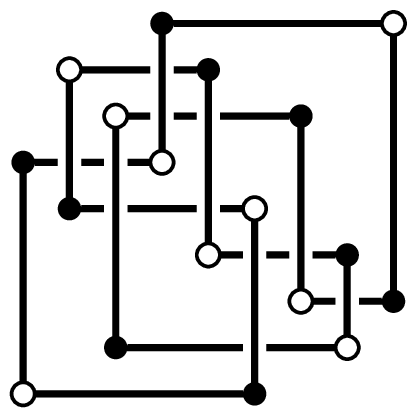}}
\end{tabular}
\caption{Diagrams of the knot $11n_{95}$}\label{11n_95-fig}
\end{figure}

\begin{figure}[ht]
\begin{tabular}{c|c|c}
$11n_{118}$, $\{1\}$
&$\mathscr L_+\left(\raisebox{-14pt}{\includegraphics[scale=.16]{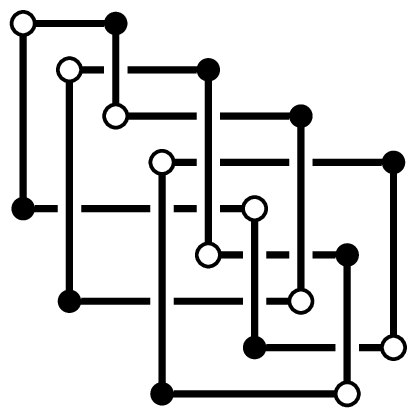}}\right)$
&$\mathscr L_+\left(\raisebox{-14pt}{\includegraphics[scale=.16]{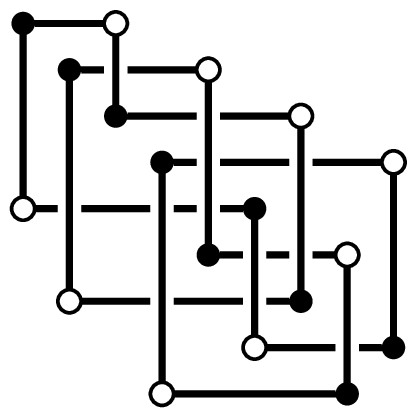}}\right)$\\\hline
$\mathscr L_-\left(\raisebox{-14pt}{\includegraphics[scale=.16]{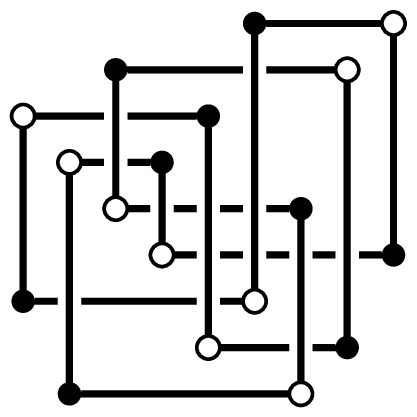}}\right)$
&\raisebox{-14pt}{\includegraphics[scale=.16]{11n_118-1.eps}}
&\raisebox{-14pt}{\includegraphics[scale=.16]{11n_118-2-m.eps}}
\end{tabular}
\caption{Diagrams of the knot $11n_{118}$}\label{11n_118-fig}
\end{figure}

\end{document}